\documentclass{amsart}

\usepackage{times,amsthm}
\usepackage{amsmath}
\usepackage{amssymb}
\usepackage{amsfonts}
\usepackage{mathrsfs}
\usepackage{latexsym}
\usepackage[all]{xy}
\SelectTips {cm}{}
\usepackage{graphicx}
\usepackage{comment}
\usepackage{appendix}
\usepackage{amscd}
\usepackage{a4wide}
\usepackage{mathrsfs}
\usepackage{pdfsync}
\usepackage{pstricks}
\usepackage{pstricks-add}
\usepackage{stmaryrd}

\usepackage{color}
\definecolor{Chocolat}{rgb}{0.36, 0.2, 0.09}
\definecolor{BleuTresFonce}{rgb}{0.215, 0.215, 0.36}

\usepackage[colorlinks,final,backref=page,hyperindex]{hyperref}
\hypersetup{citecolor=BleuTresFonce, linkcolor=Chocolat}

\theoremstyle{plain}
\newtheorem{thm}{Theorem}[section]
\newtheorem{lemma}[thm]{Lemma}
\newtheorem{prop}[thm]{Proposition}
\newtheorem{cor}[thm]{Corollary}
\newtheorem*{thmintro}{Theorem}
\theoremstyle{definition}
\newtheorem{defi}[thm]{Definition}

\theoremstyle{remark}

\newtheorem*{examples}{Examples}

\newtheorem*{remark}{Remark}

\newtheorem*{remarks}{Remarks}

\DeclareMathOperator{\id}{id}

\newcommand{\bt}[3]{{t_{#3}(#1,#2)}}

\newcommand{\Cc}{\mathcal{C}}

\newcommand{\dd}{d}

\newcommand{\dCE}{{\dd_{\psi}}}

\newcommand{\free}{\mathcal{T}}
\newcommand{\cofree}{{\free^{\mathrm{c}}}}
\newcommand{\cofreen}[2]{{\cofree}(#2)^{(#1)}}

\newcommand{\homol}{{H_\bullet}}

\newcommand{\Oo}{\mathcal{O}}

\newcommand{\rhodist}{\rho}

\newcommand{\sS}{\mathbb{S}}

\newcommand{\antishriek}{{\scriptstyle \text{\rm !`}}}
\newcommand{\im}{\mathrm{Im}\ }

\newcommand{\susp}{s}

\newcommand{\Ge}{{\mathcal{G}}}
\newcommand{\Gravity}{\mathcal{H}yc^\antishriek}

\newcommand{\Grav}{\Gravity}

\newcommand{\NeTr}{\mathsf{Nested Tree}}
\newcommand{\NT}{\mathcal{NT}}
\newcommand{\I}{\mathrm{I}}
\newcommand{\cH}{\mathcal{H}}
\newcommand{\Tree}{\mathsf{Tree}}

\newcommand{\BVK}{\mathcal{BV}_\infty}
\newcommand{\Tw}{\mathrm{Tw}}
\newcommand{\la}{\langle}
\newcommand{\ra}{\rangle}

\newcommand{\qi}{\xrightarrow{\sim}}
\newcommand{\g}{\mathfrak{g}}
\newcommand{\BV}{\mathcal{B}\mathcal{V}}

\newcommand{\qBV}{\mathrm{q}\BV}
\newcommand{\qR}{\mathrm{q}R}
\newcommand{\qPo}{\mathrm{q}\Po}

\newcommand{\KK}{\mathbb{K}}

\newcommand{\End}{\mathrm{End}}

\newcommand{\B}{\mathrm{B}}

\newcommand{\oPo}{\overline{\Po}}

\newcommand{\G}{\mathcal{G}}
\newcommand{\N}{\mathcal{N}}
\newcommand{\NN}{\mathbb{N}}

\newcommand{\Ker}{\mathop{\rm Ker }}

\newcommand{\Sy}{\mathbb{S}}

\newcommand{\Po}{\mathcal{P}}
\newcommand{\F}{\mathcal{T}}

\newcommand{\ac}{\scriptstyle \text{\rm !`}}

\newcommand{\Hom}{\mathrm{Hom}}

\newcommand{\epi}{\twoheadrightarrow}
\newcommand{\mono}{\rightarrowtail}

\newcommand{\Y}{\vcenter{\xymatrix@M=0pt@R=6pt@C=6pt{
\ar@{-}[dr] &  &\ar@{-}[dl]  \\
 &\ar@{-}[d] &  \\  & &}}}
\newcommand{\YY}{\vcenter{\xymatrix@M=0pt@R=6pt@C=6pt{
\ar@{-}[dr] &  &\ar@2{-}[dl]  \\
 &\ar@2{-}[d] &  \\  & &}}}
\newcommand{\YYY}{\vcenter{\xymatrix@M=0pt@R=6pt@C=6pt{
\ar@{-}[dr] &  &\ar@3{-}[dl]  \\
 &\ar@3{-}[d] &  \\  & &}}}
\newcommand{\cop}{\vcenter{\xymatrix@M=0pt@R=6pt@C=6pt{
 & \ar@{-}[d] & \\
 &\ar@{-}[dr] \ar@{-}[dl] &  \\  & &}}}

\newcommand{\copL}{\xymatrix@M=0pt@R=6pt@C=6pt{
 & \ar@{-}[d] & \\
 &\ar@{-}[dr] \ar@{-}[dl] &  \\  & &\\  & &\\  & &}}
\newcommand{\YL}{\vcenter{\xymatrix@M=0pt@R=6pt@C=6pt{
\ar@{-}[dr] &  &\ar@{-}[dl]  \\
 &\ar@{-}[d] &   \\  & &\\  & &\\  & &}}}
\newcommand{\YYL}{\vcenter{\xymatrix@M=0pt@R=6pt@C=6pt{
\ar@{-}[dr] &  &\ar@2{-}[dl]  \\
 &\ar@2{-}[d] &  \\  & &\\  & &\\  & &}}}
\newcommand{\LYY}{\vcenter{\xymatrix@M=0pt@R=6pt@C=6pt{
\ar@{-}[dr] &  &\ar@2{-}[dl]  \\
 &\ar@2{-}[d] &  \\  & &\\  & &\\  & &}}}

\newcommand{\XX}{\vcenter{\xymatrix@M=0pt@R=6pt@C=6pt{\ar@{-}[ddrr]&&\ar@{-}[ddll] \\ && \\ &&   }}}

\newcommand{\Ta}{\vcenter{\xymatrix@M=0pt@R=6pt@C=6pt{ \ar@{-}[dddrrr] && \ar@{-}[dl] &&  \\
&&& \ar@{-}[dl]  &  \\ &&&&  \ar@{-}[dl]  \\&&&  \ar@{-}[d] &
\\&&&& }}}
\newcommand{\Tb}{\vcenter{\xymatrix@M=0pt@R=6pt@C=6pt{  & \ar@{-}[dr]&&\ar@{-}[dl] \\
\ar@{-}[dr]&&\ar@{-}[dl]& \\&\ar@{-}[dr]&&\ar@{-}[dl]
\\&&\ar@{-}[d]& \\&&& }}}
\newcommand{\Tc}{\vcenter{\xymatrix@M=0pt@R=6pt@C=6pt{   \ar@{-}[dr]&&\ar@{-}[dl]& \\
&\ar@{-}[dr]&& \ar@{-}[dl] \\\ar@{-}[dr]&&\ar@{-}[dl]&
\\&\ar@{-}[d]&& \\&&& }}}
\newcommand{\Td}{\vcenter{\xymatrix@M=0pt@R=6pt@C=6pt{ && \ar@{-}[dr]&&\ar@{-}[dddlll] \\
 &\ar@{-}[dr]&&& \\ \ar@{-}[dr]&&&& \\& \ar@{-}[d]&&& \\&&&&  }}}
\newcommand{\Te}{\vcenter{\xymatrix@R=3pt@C=3pt{\ar@{-}[drdr] &&\ar@{-}[dl]  *=0{}
\ar@{-}[dr]&& \ar@{-}[ddll] \\ &&& *=0{}& \\&& *=0{} \ar@{-}[d]&&
\\&&&& }}}
\newcommand{\TaC}{\vcenter{\xymatrix@M=0pt@R=6pt@C=6pt{ \ar@{-}[ddddddrrrrrr] && \ar@{-}[dl] && && \\
&&& \ar@{-}[dl]  &&&  \\ &&&&  \ar@{-}[dl]&&  \\&&& &&&
\\&&&&\ar@{-}[dl]&& \\&&&&&\ar@{-}[dl]&\\&&&&&& }}}
\newcommand{\TreeL}{\vcenter{\xymatrix@M=0pt@R=5pt@C=5pt{ \ar@{-}[dr] &
&\ar@{-}[dl] & &  \\
& \ar@{-}[dr] & &\ar@{-}[dl]  & \\
& &\ar@{-}[d] & & \\
& & \\ & & }}}
\newcommand{\TreeR}{\vcenter{\xymatrix@M=0pt@R=5pt@C=5pt{
 & &\ar@{-}[dr] & & \ar@{-}[dl]  \\
& \ar@{-}[dr] & &\ar@{-}[dl]  & \\
& &\ar@{-}[d] & & \\
& & \\ & & }}}

\newcommand{\draftnote}[1]{}

\title{The minimal model for the Batalin-Vilkovisky operad}

\author{Gabriel C. Drummond-Cole}

\address{Northwestern University\\
 Mathematics Department\\ 2033 Sheridan Rd. Evanston\\
  IL 60208-2730\\
USA}
\email{gabriel@math.northwestern.edu}

\author{Bruno Vallette}

\address{Max-Planck-Institut f\"ur Mathematik \\
Vivatsgasse 7\\
53111 Bonn \\
Germany }
\email{brunov@unice.fr}
\begin{document}

\maketitle

\begin{abstract}
The purpose of this paper is to explain and to generalize, in a homotopical way, the result of Barannikov-Kontsevich and Manin which states that 
 the underlying homology groups of some Ba\-ta\-lin-Vil\-ko\-vi\-sky algebras carry  a  Frobenius manifold structure. 
To this extent, we first make the minimal model for the operad encoding BV-algebras explicit. Then we prove a homotopy transfer theorem for the associated notion of homotopy  BV-algebra. The final result provides an extension of  the action 
of the homology of the Deligne-Mumford-Knudsen moduli space of genus $0$ curves on the homology of some BV-algebras to an action via higher homotopical operations organized by the cohomology of the open moduli space of genus zero curves. 
Applications in Poisson geometry and Lie algebra cohomology and to the Mirror Symmetry conjecture are given.
\end{abstract}

\tableofcontents

\section*{Introduction}
The notion of a \emph{Batalin-Vilkovisky algebra}, or \emph{BV-algebra} for short, is made up of a commutative product, a Lie bracket and a unary operator, which satisfy some relations. 
This notion now appears in many fields of mathematics like 
\begin{itemize}
\item[$\diamond$] \textsc{Algebra:} Vertex (operator) algebras \cite{Borcherds86, LianZuckerman93}, Chevalley-Eilenberg cohomology of Lie algebras \cite{Koszul85}, bar construction of $A_\infty$-algebras \cite{TTW10}, 

\item[$\diamond$] \textsc{Algebraic geometry:} Gromov-Witten invariants and moduli spaces of curves (quantum cohomology, Frobenius manifolds) \cite{BarannikovKontsevich98, Manin99, LosevShadrin07}, 
chiral algebras (geometric Langlands program) \cite{BeilinsonDrinfeld04, FrenkelBenZvi04}, 

\item[$\diamond$] \textsc{Differential geometry:}  the sheaf of polyvector fields of an orientable (resp. Poisson or Calabi-Yau) manifold \cite{Koszul85, Ran97, Kontsevich03}, the differential forms of a manifold (Hodge decomposition in the Riemannian case) \cite{BarannikovKontsevich98, TamarkinTsygan00, Sullivan10}, Lie algebroids \cite{YKS95, Xu99, Roger09}, Lagrangian (resp. coisotropic) intersections \cite{BehrendFantechi09, BaranovkyGinzburg10}, 
 
\item[$\diamond$] \textsc{Noncommutative geometry:} the Hoschchild cohomology of a symmetric algebra \cite{TamarkinTsygan00, Tradler08, Ginzburg06, Menichi09} and the cyclic Deligne conjecture \cite{Kaufmann04, TradlerZeinalian06, Costello07, KontsevichSoibelman09, BataninBerger09}, non-commutative differential operators \cite{GinzburgSchedler10},

\item[$\diamond$] \textsc{Topology:} 2-fold loop spaces on topological spaces carrying an action of the circle \cite{Getzler94}, topological conformal field theories, Riemann surfaces \cite{Getzler94},  string topology \cite{ChasSullivan99},

\item[$\diamond$] \textsc{Mathematical physics:} BV quantization (gauge theory) \cite{BatalinVilkovisky81, Witten90, Schwarz93, Roger09}, BRST cohomology \cite{LianZuckerman93, Stasheff98}, string theory \cite{Witten92, WittenZwiebach92, Zwiebach93, PenkavaSchwarz94}, topological field theory \cite{Getzler94}, Renormalization theory \cite{CostelloGwilliam11}.   
\end{itemize}
Nearly all the examples of BV-algebras appearing in the aforementioned fields 
actually have some homology groups as underlying spaces. Therefore they are some shadow of a higher structure: 
that of a \emph{homotopy} BV-algebra.\\

Algebra and homotopy theories do not mix well together a priori. 
The study of the homotopy properties of algebraic structures often introduces infinitely many new higher operations of higher arity.
So one uses the operadic calculus to encode them.

The study of the homotopy properties of algebraic structures often
introduces infinitely many new higher operations of higher arity.

\smallskip

This is the case for Batalin-Vilkovisky algebras, which do not have homotopy invariance properties, like the transfer of structure under homotopy equivalences, see \cite[Section~$10.3$]{LodayVallette10}. To solve this, we have defined, in \cite{GCTV09}, a notion of \emph{homotopy Batalin-Vilkovisky algebra} with the required homotopy properties. To do so, we have constructed a quasi-free, thus cofibrant, resolution of the operad $\BV$ encoding Batalin-Vilkovisky algebras, using the inhomogeneous Koszul duality theory. 

\smallskip

While quite ``small'', this resolution carries a non-trivial internal differential; so it is not minimal in the sense of D. Sullivan \cite{Sullivan77}. The purpose of the present paper is to go even further and to produce the minimal model of the operad $\BV$, that is, a resolution as a  quasi-free operad with a decomposable differential and a certain grading on the space of generators. 

\smallskip

The first main result of this paper is the following computation of the homology groups of the bar construction for the operad $\BV$ as a deformation retract. 
\begin{thmintro}[\ref{thm:MainDefRetract}]
The various maps defined in Section~$2$ form the following deformation retract in the category of differential graded $\Sy$-modules
\begin{eqnarray*}
\xymatrix@C=13pt{     *{\qquad\qquad\qquad\quad\qquad
\BV^{\ac}:=({\qBV}^\antishriek, d_\varphi)  \  \ } \ar@<-1ex>@(dl,ul)[]\ \ar@<0.5ex>[r] & *{\ 
(H_\bullet(\B \, \BV)\cong \overline{T}^c(\delta) \oplus {\mathcal S}^{-1} Grav^*, 0). \quad \ \  \ \quad }  \ar@<0.5ex>[l] &  & \qquad }
\end{eqnarray*}
where $\delta$ denotes a unary operator of degree $2$, where ${\mathcal S}^{-1}$ is the operadic desuspension and where $Grav$ is the operad Gravity isomorphic to the homology $H_\bullet(\mathcal{M}_{0, n+1})$ of the moduli space of genus $0$ curves with marked points.
\end{thmintro}
This result provides the space of generators for the minimal model of the operad $\BV$. 
But, on the opposite to the Koszul duality theory, this graded $\Sy$-module is not endowed with a cooperad structure but with a \emph{homotopy cooperad} structure. This means that there are higher decomposition maps which split elements, not only into $2$ but also into $3$, $4$, etc. Finally, the differential of the minimal model is made up of these decomposition maps. 

\smallskip

Let us recall that the problem of making minimal models explicit in algebraic topology 
is related to  the following notions. Sullivan models \cite{Sullivan77} are dg commutative algebras generated by the (dual of the) rational homotopy groups $\pi_\bullet X \otimes \mathbb{Q}$ of a topological space $X$, where the differential is given by the Whitehead products. 
Quillen models \cite{Quillen69} are dg Lie algebras generated by the (dual of the) rational homology groups $H^\bullet(X, \mathbb{Q})$ of a topological space $X$, where the differential is given by the  Massey products. The Steenrod algebra is an inhomogenous Koszul algebra, whose Koszul dual dg algebra is the $\Lambda$ algebra, see \cite{Priddy70}. The minimal model of the Steenrod algebra is generated by the underlying homology groups of the $\Lambda$ algebra and the differential is related to  the Adams
spectral sequence \cite{BCKQRS66, Wang67}. Recall that the $\Lambda$ algebra is the first page of the Adams spectral sequence, which computes the homotopy groups of spheres. 

\medskip

\noindent
\begin{tabular}{|c|c|c|c|c|}
\hline
& \textsc{Sullivan} & \textsc {Quillen} & \textsc {Steenrod} & \textsc {Batalin-Vilkovisky} \\
\hline 
 \textit{free} & \begin{minipage}[c]{2.7cm}\center  commutative algebra  $S(-)$ \end{minipage}  
 & \begin{minipage}[c]{2.7cm} \center Lie algebra $\mathcal L ie (-)$ \end{minipage}  & 
\begin{minipage}[c]{2.7cm} \center associative algebra $T(-)$ \end{minipage}  & 
\begin{minipage}[c]{2.7cm} \center  operad $\free ( - )$ \end{minipage}  \\
\hline
\textit{generators} & $\pi_\bullet X \otimes \mathbb{Q}$ & $H^\bullet(X, \mathbb{Q})$ & $\Lambda$ & 
$H^\bullet(\mathcal{M}_{0, n+1}) \oplus \bar T^c(\delta)$\\
\hline 
\textit{differential} & Whitehead brackets & Massey products & \begin{minipage}[c]{2.7cm} \center  differentials of  Adams spectral sequence \end{minipage} & homotopy cooperad \\
\hline
\end{tabular}

\medskip 
The operad $\BV$ behaves exactly in the same way as the Steenrod algebra with respect to the inhomogeneous Koszul duality theory, see \cite{GCTV09}. But, in contrast to the Steenrod algebra, we are able to compute, in this paper, the underlying homology groups of its Koszul dual (co)operad together with its algebraic structure. We also provide a topological explanation for our result. It was shown by E. Getzler \cite{Getzler94} that the operad $\BV$ is the homology of the framed little discs operad. Its minimal model is generated by a homotopy cooperad extension of the cooperad $H^\bullet(\mathcal{M}_{0, n+1})$  by a free resolution $T^c(\delta)$ of the circle $S^1$. 

\smallskip

We call the algebras over the minimal model of the operad $\BV$  \emph{skeletal homotopy Batalin-Vilkovisky algebras}, since they involve fewer generating operations than the notion of a homotopy BV-algebra given in \cite{GCTV09}. We provide new formulae for the homotopy transfer theorem for algebras over a quasi-free operad on a homotopy cooperad. To prove them, we have to introduce a new operadic method, based on a refined bar-cobar adjunction, since the classical methods of \cite[$10.2$]{LodayVallette10} (classical bar-cobar adjunction), and of A. Berglund \cite{Berglund09} (homological perturbation lemma) 
failed to apply. This gives the homotopy transfer theorem for skeletal homotopy BV-algebras.

\smallskip

The $d\Delta$-condition, also called the $d\bar d$-lemma or $dd^c$-lemma in \cite{DGMS75}, is a particular condition, coming from K\"ahler geometry, between the two unary operators:  the underlying differential $d$ and the BV-operator $\Delta$. Under this condition, S. Barannikov and M. Kontsevich  \cite{BarannikovKontsevich98}, and Y.I. Manin \cite{Manin99} proved that the underlying homology groups of a dg BV-algebra carry a Frobenius manifold structure. Such a structure is encoded by the homology operad $H_\bullet(\overline{\mathcal{M}}_{0, n+1})$ of the Deligne-Mumford-Knudsen moduli space of stable genus $0$ curves. Its is also called an hypercommutative algebra (with a 
 a compatible non-degenerate pairing). Tree formulae  for such a structure have been given by A. Losev and S. Shadrin in \cite{LosevShadrin07}. 
 
 We show that these results are actually a consequence of the aforementioned homotopy transfer theorem. This allows us to prove them under a weaker and optimal condition, called the non-commutative Hodge-to-de Rham condition. We recover the Losev-Shadrin formulae and thereby explain their particular form. Moreover our approach gives higher non-trivial operations, which are necessary to recover the homotopy type of the original dg BV-algebra. 

\begin{thmintro}[\ref{thm:HomotopyFrob}]
Let $(A, d,  \bullet, \Delta, \langle\,  , \rangle)$ be a dg BV-algebra with non-commutative Hodge-to-de Rham degeneration data. 

The underlying homology groups $H(A,d)$ carry a homotopy hypercommutative algebra structure, which extends the hypercommutative algebras of M. Kontsevich and S. Barannikov \cite{BarannikovKontsevich98}, Y.I. Manin \cite{Manin99},  A. Losev and S. Shadrin \cite{LosevShadrin07}, and J.-S. Park \cite{Park07}, and such that 
the rectified  dg BV-algebra $\mathrm{Rec}(H(A))$
is homotopy equivalent to $A$ in the category of dg BV-algebras. 
\end{thmintro}

In geometrical terms, this lifts the action of the operad of moduli space of genus $0$ stable curves (cohomological field theory) into a certain action of the cooperad of the open moduli space of genus $0$ curves (extended cohomological field theory):
$$\xymatrix@C=30pt{H^{\bullet+1}({\mathcal{M}}_{0, n+1})   \ar@{..>}[r]  \ar[d] &   \End_{H(A)} \\
H_\bullet (\overline{\mathcal{M}}_{0, n+1})  \ar[ur]_{\ \quad [BK-M-LS-P]} \ .&  } $$

We conclude this paper with applications to the Poisson geometry, Lie algebra cohomology and the Mirror Symmetry conjecture. To conclude, this paper develops the homotopy theory for dg BV-algebras (homotopy skeletal BV-algebras, $\infty$-quasi-isomorphisms) necessary to study the 
Mirror Symmetry conjecture, in the same way as the homotopy theory of dg Lie algebras
was used to prove 
the deformation-quantization of Poisson manifolds  by M. Kontsevich in \cite{Kontsevich03}.

% the Lian-Zuckerman conjecture. 

\smallskip

Some of the results of the present paper were announced in \cite{DrummondColeVallette09}. While we were typing it, V. Dotsenko and A. Khoroshkin computed in \cite{DotsenkoKhoroshkin09} the homology of the bar construction of the operad $\BV$ without the action of the symmetric groups. They used the independent method of Gr\"obner basis for shuffle operads  developed in  \cite{DotsenkoKhoroshkin10}. \\

The paper is organized as follows. We begin by recalling the Koszul resolution of the operad $\BV$ given in \cite{GCTV09}.
In the second section, we compute the homology of the Koszul dual dg cooperad $\BV^{\ac}$ and we write it as a deformation retract of $\BV^{\ac}$. In the third section, we recall the notion of {homotopy cooperad} with its homotopy properties: the homotopy transfer theorem for homotopy cooperads. With these tools in hand, we produce the minimal model of the operad $\BV$ at the end of Section~$3$. In a fourth section, we describe the associated notion of algebra, called skeletal homotopy BV-algebras. Section~$5$ deals with a generalization of the bar-cobar adjunction between operads and homotopy cooperads. The last section contains the homotopy transfer theorem for skeletal homotopy BV-algebras and the extension of the result of Barannikov-Kontsevich and Manin. \\

The reader is supposed to be familiar with the notion of an operad and operadic homological algebra, for which  we refer to the book \cite{LodayVallette10}. In the present paper, we use the same notations as used in this reference. 

We work over a field $\KK$ of characteristic $0$ and all the $\Sy$-modules $M=\lbrace M(n)\rbrace_{n\in\NN}$ are reduced, that is, $M(0)=0$. 

\section{Recollection on homotopy BV-algebras}

In this section, we recall the main results of \cite{GCTV09} needed in the rest of the text. In loc.cit.,  we made explicit a resolution of the operad $\BV$ using the Koszul duality theory. It is given by a quasi-free operad on a dg cooperad, which is smaller than the bar construction of $\BV$. 

\subsection{BV-algebras}

\begin{defi}[Batalin-Vilkovisky algebras]
A \emph{differential graded Batalin-Vilkovisky algebra}, or \emph{dg BV-algebra} for short,  is a differential graded vector space $(A, d_A)$ endowed with
\begin{itemize}
\item[$\triangleright$]  a symmetric binary product $\bullet$ of degree
$0$,

\item[$\triangleright$]  a symmetric  bracket $\langle\; ,  \, \rangle$  of degree $+1$,

\item[$\triangleright$]  a unary operator $\Delta$ of degree
$+1$,
\end{itemize}
such that $d_A$ is a derivation with respect to each of them and such that
\begin{itemize}
\item[$\rhd$] the product $\bullet$ is associative,

\item[$\rhd$] the bracket satisfies the Jacobi identity

$$\langle \langle\; ,  \, \rangle,  \, \rangle \ +\ \langle \langle\; ,  \, \rangle,  \, \rangle.(123) \ +\ \langle \langle\; ,  \, \rangle,  \, \rangle.(321)\ = \ 0,$$

\item[$\rhd$] the product $\bullet$ and the bracket $\langle\;
, \,\rangle$ satisfy the Leibniz relation
$$\langle\, \textrm{-} , \textrm{-}\bullet \textrm{-}\, \rangle\ =\
(\langle\,\textrm{-}, \textrm{-}\,\rangle\bullet \textrm{-})\
+ \ (\textrm{-}\bullet \langle\,\textrm{-},\textrm{-}\,\rangle).(12),   $$

\item[$\rhd$] the unary operator $\Delta$ satisfies $\Delta^2=0$,

 \item[$\rhd$] the bracket is the obstruction to $\Delta$ being a derivation with respect to  the product $\bullet$
$$\langle\,\textrm{-} , \textrm{-}\,\rangle\ =\ \Delta (\textrm{-} \bullet \textrm{-})\ -\
(\Delta(\textrm{-}) \bullet \textrm{-})   \ - \ (\textrm{-} \bullet
\Delta(\textrm{-})),$$
\item[$\rhd$] the operator $\Delta$ is a graded derivation with
respect to the bracket
$$\Delta (\langle\,  \textrm{-}, \textrm{-}\, \rangle) \ + \ \langle\Delta(\textrm{-}),
\textrm{-}\,\rangle \ +\ \langle\, \textrm{-}, \Delta(\textrm{-})\rangle \ = \ 0.$$

\end{itemize}
\end{defi}

The operad encoding BV-algebras is the operad defined by generators and relations
$$\BV:=\F(V)/(R),$$
where $\F(V)$ denotes the free operad on the $\Sy$-module
$$V:= \KK_2  \bullet \oplus \KK_2  \la\, ,\, \ra \oplus
\KK \Delta\ , $$ with $\KK_2$ being the trivial representation of the symmetric group $\Sy_2$.
The space of relations $R$ is the sub-$\Sy$-module
of $\F(V)$ generated by the  relations `$\rhd$' given above. The
basis elements $\bullet$, $\,\la\:,\,\ra$, $\,\Delta$ are of degree 0, 1, and 1. Since
 the relations are homogeneous, the operad $\BV$
is graded by this degree, termed the {\em homological degree}. 

We denote by $Com$ the operad generated by the symmetric product $\bullet$ and the associativity relation. We denote by $Lie_1$ the operad generated by the symmetric bracket $\la \; , \, \ra$ and the Jacobi relation; it is the operad encoding Lie algebra structures on the suspension of a space. The operad $\G$ governing Gerstenhaber algebras is defined similarly. Its underlying $\Sy$-module is isomorphic to $Com \circ Lie_1$, on which the operad structure is given by means of distributive laws, see \cite[Section~$8.6$]{LodayVallette10}.

\subsection{Quadratic analogue} 
We consider the homogeneous quadratic analogue $\qBV$ of the operad $\BV$. This operad is defined by the same spaces of generators $V$ and relations except for the inhomogeneous relation 
\[
\begin{pspicture}(-.7,.43)(.7,1.1)
\psline(0,0)(0,.3)
\psline(.2,.7)(.5,1)
\psline(-.2,.7)(-.5,1)
\rput(0,.5){$\bullet$}
\uput[d](0,.1){${\scriptstyle \Delta}$}
\end{pspicture}
-
\begin{pspicture}(-.7,.43)(.7,1.1)
\psline(0,0)(0,.3)
\psline(.2,.7)(.5,1)
\psline(-.2,.7)(-.5,1)
\rput(0,.5){$\bullet$}
\uput[u](-.5,.9){${\scriptstyle \Delta}$}
\end{pspicture}
-
\begin{pspicture}(-.7,.43)(.7,1.1)
\psline(0,0)(0,.3)
\psline(.2,.7)(.5,1)
\psline(-.2,.7)(-.5,1)
\rput(0,.5){$\bullet$}
\uput[u](.5,.9){${\scriptstyle \Delta}$}
\end{pspicture}
-
\begin{pspicture}(-.7,.43)(.7,1.1)
\psline(0,0)(0,.3)
\psline(.2,.7)(.5,1)
\psline(-.2,.7)(-.5,1)
\rput(0,.5){${\scriptstyle \langle\, , \, \rangle}$}
\end{pspicture}
,
\]
\smallskip

\noindent
which is changed into the homogenous relation: 
\[
\begin{pspicture}(-.7,.43)(.7,1.1)
\psline(0,0)(0,.3)
\psline(.2,.7)(.5,1)
\psline(-.2,.7)(-.5,1)
\rput(0,.5){$\bullet$}
\uput[d](0,.1){${\scriptstyle \Delta}$}
\end{pspicture}
-
\begin{pspicture}(-.7,.43)(.7,1.1)
\psline(0,0)(0,.3)
\psline(.2,.7)(.5,1)
\psline(-.2,.7)(-.5,1)
\rput(0,.5){$\bullet$}
\uput[u](-.5,.9){${\scriptstyle \Delta}$}
\end{pspicture}
-
\begin{pspicture}(-.7,.43)(.7,1.1)
\psline(0,0)(0,.3)
\psline(.2,.7)(.5,1)
\psline(-.2,.7)(-.5,1)
\rput(0,.5){$\bullet$}
\uput[u](.5,.9){${\scriptstyle \Delta}$}
\end{pspicture}
.
\]
\smallskip

We denote this homogenous quadratic space of relations by $\qR$. This operad $\qBV=\F(V)/(\qR)$ is also given by means of distributive laws on the $\Sy$-module 
 $$\qBV\cong\G\circ D\cong Com \circ Lie_1 \circ \KK[\Delta]/(\Delta^2)\ ,$$
where $D:= \KK[\Delta]/(\Delta^2)$ is the algebra of dual numbers,  see \cite[Proposition~$3$]{GCTV09}. 
  
\subsection{Koszul dual cooperad of the operad $\qBV$}  
We denote by $s$ the homological suspension, which shifts the homological degree by $+1$. Recall that the \emph{Koszul dual cooperad} of a quadratic operad $\F(V)/(\qR)$ is defined as the sub-cooperad $\mathcal{C}(sV, s^2\qR)\subset \F^c(sV)$ cogenerated by the suspension $sV$ of $V$ with corelators in the double suspension $s^2 \qR$ of $\qR$, see  \cite[Chapter~7]{LodayVallette10}. Namely, it is the ``smallest'' sub-cooperad of the cofree cooperad on $sV$ which contains the corelators $s^2\qR$.

  We denote by ${\mathcal S}^c:=\End^c_{\KK s^{-1}}=\lbrace \Hom((\KK s^{-1})^{\otimes n},\KK s^{-1})\rbrace_{n \in \NN}$ the \emph{suspension} cooperad of endomorphisms
of the one dimensional vector space $s^{-1}\KK$ concentrated in degree $-1$. The desuspension ${\mathcal S}^c \mathcal C$ of a cooperad $\mathcal{C}$ is the cooperad defined by the aritywise tensor product, called the Hadamard tensor product, $({\mathcal S}^c \mathcal C)(n)=({\mathcal S}^c\otimes_{\textrm{H}} \mathcal C)(n):={\mathcal S}^c(n) \otimes \mathcal C(n)$. 
The underlying $\Sy$-module of the Koszul dual cooperad of $\qBV$ is equal to 
$$\qBV^{\ac}\cong T^c(\delta)\circ {\mathcal S}^c Com_1^c \circ {\mathcal S}^c Lie^c\ ,$$ 
where $T^c(\delta)\cong\KK[\delta]\cong D^{\ac}$ is the counital cofree coalgebra on a degree $2$
generator $\delta:=s\Delta$, where  $ Lie^c\cong Lie^*$ is the cooperad encoding Lie coalgebras and where $ Com_1^c\cong Com_{-1}^*$ is the cooperad encoding cocommutative
 coalgebra structures on the suspension of a space, see \cite[Corollary~$4$]{GCTV09}.
 The degree of the elements in $$\KK \delta^m \otimes
{\mathcal S}^c Com_1^c(t) \otimes {\mathcal S}^c Lie^c(p_1)  \otimes \ldots \otimes
{\mathcal S}^c Lie^c(p_t)\subset \qBV^{\ac}$$ is $n+t+2m-2$, for
$n=p_1+\cdots+p_t$.

\subsection{Koszul dual dg cooperad of the operad $\BV$}  

We consider the map $\varphi\, :\, \qR \to V$ defined by
\[
\begin{pspicture}(-.7,.43)(.7,1.1)
\psline(0,0)(0,.3)
\psline(.2,.7)(.5,1)
\psline(-.2,.7)(-.5,1)
\rput(0,.5){$\bullet$}
\uput[d](0,.1){${\scriptstyle \Delta}$}
\end{pspicture}
-
\begin{pspicture}(-.7,.43)(.7,1.1)
\psline(0,0)(0,.3)
\psline(.2,.7)(.5,1)
\psline(-.2,.7)(-.5,1)
\rput(0,.5){$\bullet$}
\uput[u](-.5,.9){${\scriptstyle \Delta}$}
\end{pspicture}
-
\begin{pspicture}(-.7,.43)(.7,1.1)
\psline(0,0)(0,.3)
\psline(.2,.7)(.5,1)
\psline(-.2,.7)(-.5,1)
\rput(0,.5){$\bullet$}
\uput[u](.5,.9){${\scriptstyle \Delta}$}
\end{pspicture}
\longmapsto
\begin{pspicture}(-.7,.43)(.7,1.1)
\psline(0,0)(0,.3)
\psline(.2,.7)(.5,1)
\psline(-.2,.7)(-.5,1)
\rput(0,.5){${\scriptstyle \langle\, , \, \rangle}$}
\end{pspicture}
\]
\smallskip

\noindent
and $0$ 
 on the other relations of $\qR$, so that the graph of $\varphi$  is equal
to the space of relations $R$. The induced map $\qBV^{\ac} \to sV$ extends to a square-zero coderivation $d_\varphi$ on the cooperad $\qBV^{\ac}$,  see  \cite[Lemma~$5$]{GCTV09}. The dg cooperad 
$$\BV^{\ac}:=(\qBV^{\ac}, d_\varphi)$$ is called the \emph{Koszul dual dg cooperad} of the inhomogeneous quadratic operad $\BV$.

We use the notation $\odot$ for the `symmetric' tensor product, that
is, the quotient of the tensor product under the permutation of
terms. In particular, we denote by $\delta^m\otimes L_1 \odot \cdots \odot L_t$
a generic element  of $T^c(\delta)\circ {\mathcal S}^c Com_1^c \circ {\mathcal S}^cLie^c$
 with $L_i\in {\mathcal S}^cLie^c$, for $i=1,\ldots, t$; the elements of ${\mathcal S}^c Com_1^c$ being implicit. Under these notations,  the coderivation $d_\varphi$ is explicitly given by
\begin{eqnarray}\label{Coder}
d_\varphi(\delta^m\otimes L_1 \odot \cdots \odot L_t)=\sum_{i=1}^t (-1)^{\varepsilon_i}  \delta^{m-1}
\otimes L_1\odot \cdots\odot L_i'\odot L_i''\odot \cdots
\odot L_t ,
\end{eqnarray} where $L_i'\odot L_i''$ is Sweedler-type
notation for  the image of $L_i$ under the binary part
$$
{\mathcal S}^c Lie^c \to {\mathcal S}^c Lie^c(2) \otimes ({\mathcal S}^c Lie^c\otimes {\mathcal S}^c Lie^c)\twoheadrightarrow {\mathcal S}^c Lie^c\odot {\mathcal S}^c Lie^c
$$
 of the decomposition
map of the cooperad $S^c Lie^c$.
The sign, given by the Koszul rule, is
equal to $\varepsilon_i=(|L_1|+\cdots+|L_{i-1}|)$. The image of
$d_\varphi$  is equal to $0$ when $m=0$ or when $L_i\in
S^c Lie^c(1)=\KK \, \textrm{I}$ for all $i$.

\begin{remark}
Let us denote the linear dual of $\delta$ by $\hbar:=\delta^*$. This is an element of homological degree $-2$. The Koszul dual operad is defined by 
$\qBV^!:={\mathcal S}{\qBV^{\ac}}^*={\mathcal S} \otimes_{\textrm{H}}{\qBV^{\ac}}^* $, where $S$ stands for the endomorphism operad ${\mathcal S}:=\End^c_{\KK s^{-1}}$. 
Up to a degree shift, the Koszul  dual dg operad $\BV^!:=(\qBV^!, {}^t d_\varphi)$, when viewed as a cohomologically graded differential  $\KK[[\hbar]]$-operad, corresponds  to the Beilinson-Drinfeld operad \cite{BeilinsonDrinfeld04, CostelloGwilliam11}. 
\end{remark}

\subsection{Koszul resolution of the operad $\BV$}
We denote by $\BVK$ the quasi-free operad given by the cobar construction on $\BV^{\ac}$:
$$\BVK := \Omega\,  \BV^{\ac}\cong(\F(s^{-1}\overline{\qBV}^{\ac}),\, d=d_1+d_2),$$
where $d_1$ is the unique derivation which extends the internal differential $d_\varphi$ and where $d_2$ is the unique derivation which extends the infinitesimal (or partial) coproduct of the cooperad $\qBV^{\ac}$, see \cite[Section~$6.5$]{LodayVallette10}.
The total derivation $d=d_1+d_2$  squares to zero
and faithfully encodes the algebraic structure of
the dg cooperad
 on
$\BV^{\ac}$. The space of generators of this quasi-free operad is isomorphic to
$T^c(\delta)\circ {\mathcal S}^c Com_1^c \circ {\mathcal S}^c Lie^c$, up to coaugmentation and desuspension. 

\begin{thm}\cite[Theorem~$6$]{GCTV09}
The operad $\BVK$ is a resolution  of the operad $\BV$
$$\BVK=\Omega\,  \BV^{\ac} =\left(\F(s^{-1}\overline{\qBV}^{\ac}),\,d=d_1+d_2\right)\;\qi \;\BV\ . $$
\end{thm}
It is called the \emph{Koszul resolution} of $\BV$. Notice that it is much smaller than the bar-cobar resolution $\Omega  \B  \, \BV \qi \BV$. The Koszul resolution and the bar-cobar resolution are both quadratic. But they are not minimal resolutions: they are both quasi-free operads with a differential which is the sum of a quadratic term ($d_2$) and a \emph{non-trivial linear term} ($d_1$). 

Algebras over the operad $\BV_\infty$ are called \emph{homotopy BV-algebras}. For an explicit description of this algebraic notion together with its homotopy properties, 
we refer the reader to  \cite{GCTV09}.

\subsection{Homotopy transfer theorem for homotopy BV-algebras}
We consider the data 
\begin{eqnarray*}
&\xymatrix{     *{ \quad \ \  \quad (A, d_A)\ } \ar@(dl,ul)[]^{h}\ \ar@<0.5ex>[r]^{p} & *{\
(H,d_H)\quad \ \  \ \quad }  \ar@<0.5ex>[l]^{i}\ , }&
\end{eqnarray*}
of two chain complexes, where $i$ and $p$ are chain maps and where $h$ has degree $1$. 
It is called a \emph{homotopy retract} when $\textrm{id}_A-i p =d_A  h+ h  d_A$ and when, equivalently, $i$ or $p$ is a quasi-isomorphism. If, moreover, the composite $pi$ is equal to $\textrm{id}_H$, then it is called a \emph{deformation retract}.

\begin{thm}\cite[Theorem~$33$]{GCTV09}
Any homotopy BV-algebra structure on $A$ transfers to $H$ through a homotopy retract such that $i$ extends to an $\infty$-quasi-isomorphism. 
\end{thm}

\section{The homology of $\BV^\antishriek$ as a deformation retract}\label{sec:homotopy}
The purpose of this section is to construct an explicit contracting homotopy for the chain complex $\BV^{\ac}:=(\qBV^\antishriek, d_\varphi)$.  This is a necessary ingredient for the construction of  the minimal model of the operad $\BV$ given in the next section.  
As a byproduct, this computes  the homology of the bar construction of the operad $\BV$ in terms of  the homology of the moduli space $\mathcal{M}_{0,n+1}$ of genus $0$ curves.  The main result of this section is the following theorem. 

\begin{thm}\label{thm:MainDefRetract}
The various maps defined in this section form the following deformation retract:
\begin{eqnarray*}
\xymatrix@C=7pt{     *{\qquad\quad
({\qBV}^\antishriek\cong{T}^c(\delta)\otimes{\G}^\antishriek, d_\varphi\cong \delta^{-1}\otimes d_\psi)  \  \ } \ar@<13ex>@(dl,ul)[]^{\delta\otimes H}\ \ar@<0.5ex>[r] & *{\ 
(\overline{T}^c(\delta)\otimes \I   \oplus 1\otimes \G^{\ac}/ \im d_\psi\cong \overline{T}^c(\delta) \oplus {\mathcal S}^{-1} Grav^*, 0). \quad \ \  \ \quad }  \ar@<0.5ex>[l] &  & \qquad }
\end{eqnarray*}
\end{thm}

\subsection{Trees}\label{subsec:Trees}
A \emph{reduced rooted tree} is a rooted tree 
whose vertices have at least one input.
We consider the category of reduced rooted trees with leaves labeled bijectively from $1$ to $n$, denoted by $\Tree$. 
The trivial tree $|$ is considered to be part of $\Tree$.
Since the trees are reduced, there are only trivial isomorphisms of trees. So we identify the isomorphism classes of trees with the trees themselves; see \cite[Appendix~C]{LodayVallette10} for more details. 

We consider the planar representation of reduced trees provided by shuffle trees, see E. Hoffbeck \cite[$2.8$]{Hoffbeck10}, V. Dotsenko and A. Khoroshkin \cite[$3.1$]{DotsenkoKhoroshkin10}, and \cite[$8.2$]{LodayVallette10}.
We define a total order on the vertices of a tree by reading its planar representation from leaf $1$ to the root by following the internal edges without crossing them. See Figure~\ref{Fig:PlanarRep} for an example. 

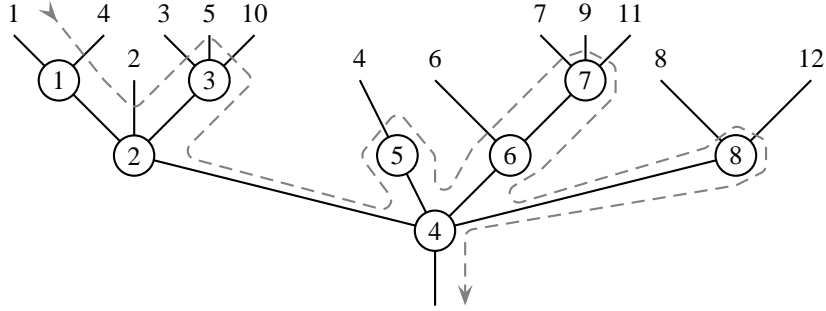
\begin{figure}[h]
\centering

\psset{unit=2cm}
\begin{pspicture}(-1.3,-.2)(4.1,2.3)
\psline(1.5,0)(1.5,.5)
\psline(1.5,.5)(1.25,1)
\psline(1.5,.5)(2,1)
\psline(1.5,.5)(3.5,1)
\psline(1.5,.5)(-.5,1)
\psline(-.5,1)(-1,1.5)
\psline(-.5,1)(-.5,1.5)
\uput[u](-.5,1.5){2}
\psline(-.5,1)(0,1.5)
\psline(-1,1.5)(-1.3,1.8)
\uput[u](-1.3,1.8){1}
\psline(-1,1.5)(-.7,1.8)
\uput[u](-.7,1.8){4}
\psline(0,1.5)(-.3,1.8)
\uput[u](-.3,1.8){3}
\psline(0,1.5)(0,1.8)
\uput[u](0,1.8){5}
\psline(0,1.5)(.3,1.8)
\uput[u](.3,1.8){10}
\psline(1.25,1)(1,1.5)
\uput[u](1,1.5){4}
\psline(2,1)(1.5,1.5)
\uput[u](1.5,1.5){6}
\psline(2,1)(2.5,1.5)
\psline(2.5,1.5)(2.2,1.8)
\uput[u](2.2,1.8){7}
\psline(2.5,1.5)(2.5,1.8)
\uput[u](2.5,1.8){9}
\psline(2.5,1.5)(2.8,1.8)
\uput[u](2.8,1.8){11}
\psline(3.5,1)(3,1.5)
\uput[u](3,1.5){8}
\psline(3.5,1)(4,1.5)
\uput[u](4,1.5){12}
\rput(1.5,.5){\pscirclebox[fillstyle=solid]{4}}
\rput(1.25,1){\pscirclebox[fillstyle=solid]{5}}
\rput(2,1){\pscirclebox[fillstyle=solid]{6}}
\rput(3.5,1){\pscirclebox[fillstyle=solid]{8}}
\rput(-.5,1){\pscirclebox[fillstyle=solid]{2}}
\rput(-1,1.5){\pscirclebox[fillstyle=solid]{1}}
\rput(0,1.5){\pscirclebox[fillstyle=solid]{3}}
\rput(2.5,1.5){\pscirclebox[fillstyle=solid]{7}}
\psline[linecolor=gray, linearc=.06, linestyle=dashed, arrowsize=.1]{>->}(-1.1,2)(-.7,1.5)(-.5,1.3)(0,1.8)(.3,1.5)(-.2,1)
(1.3,.6)(1,1)(1.25,1.3)(1.5,1)(1.5,.6)(1.75,1.05)(2.35,1.65)
(2.5,1.7)(2.7,1.6)(2.7,1.4)(2.2,.9)(1.6,.55)(3.3,1.05)(3.5,1.2)(3.7,1.1)(3.7,.9)(3.5,.8)(1.7,.5)(1.7,0)
\end{pspicture}

\caption{Example of a planar representation of a tree with ordered vertices} \label{Fig:PlanarRep}
\end{figure}

\subsection{Free operad and cofree cooperad}\label{subsec:FreeOp}

The underlying $\Sy$-module of the  free operad $\free(V)$ on an $\Sy$-module $V$ is given by the direct sum 
$\bigoplus_{t\in \Tree} t(V)$, where $t(V)$ is the treewise tensor module obtained by labeling every vertex of the tree $t$ with an element of $V$ according to the arity and the action of the symmetric groups. The operadic composition map is given the the grafting of trees.
Dually, the underlying $\Sy$-module of the conilpotent cofree cooperad $\cofree(V)$ is equal to the same direct sum over trees and  its decomposition map is given by cutting the trees horizontally; see \cite[Chapter~$5$]{LodayVallette10} for more details.

The subcategory of trees with $n$ vertices is denoted by $\Tree^{(n)}$. The number of vertices endows the free operad $\free(V)\cong \bigoplus_{n \in \NN} \free(V)^{(n)}$ and the conilpotent cofree cooperad $\cofree(V)\cong \bigoplus_{n \in \NN} \cofree(V)^{(n)}$ with a weight grading. We represent a labeled tree by $t(v_1, \ldots, v_n)$, using the aforementioned total order on vertices. 

\subsection{Coderivations on the cofree cooperad}\label{subsec:codercofree}
Coderivations on cofree cooperads are characterized by their projection onto the space of generators.  In other words,
\begin{lemma}\label{lemma:cofreecoder}
Let $\eta$ be a homogeneous morphism $\cofree(M)\to M$ of graded $S$ modules.  Then there exists a unique coderivation $d_\eta$ on $\cofree(M)$ extending $\eta$, given on an element of $\cofree(M)$ represented by a decorated tree
by applying $\eta$ to any 
subtree.
\end{lemma}
This is a classical generalization of the characterization of coderivation for cofree coalgebras. Here are two simple but useful examples, for more details see \cite[Section~$6.3.14$]{LodayVallette10}. 
\begin{enumerate}
\item If $\eta$ factors through the projection $\cofree(M)\twoheadrightarrow \cofreen{1}{M}=M$, then $d_\eta$ is given on a decorated tree as a signed sum over the vertices of the tree.  The summand corresponding to a vertex $v$ is the same tree with $\eta$ applied to the decoration of $v$ and all other decorations the same.  The sign is the Koszul sign.
\item 
If $\eta$ factors through the projection $\cofree(M)\twoheadrightarrow \cofreen{2}{M}$, then $d_\eta$ is given on a decorated tree as a signed sum over the internal edges of the tree.  The summand corresponding to an edge $e$ has the edge contraction along $e$ of the original tree as its underlying tree.  The decorations away from the contraction vertex are the same; the decoration on the contraction vertex is given by applying $\eta$ to the two decorated vertices involved in the contraction, viewed as a two-vertex decorated tree in $\cofreen{2}{M}$.  The sign is the Koszul sign.
\end{enumerate}
\subsection{A contracting homotopy for a cofree cooperad}
Let $M$ be the $\sS$-module which is the linear span of elements $\mu$ and $\beta$, both of arity two, in degrees $1$ and $2$ respectively, both with trivial symmetric group action. 
\[
M:=\KK_2\underbrace{ s\bullet}_\mu\oplus \KK_2\underbrace{s\la \; , \, \ra}_{\beta}  
  \]

Let $\psi$ denote the degree one morphism of graded $\sS$-modules $\psi:\cofree(M)\to M$ which first projects $\cofree(M)$ to the cogenerators $M$ and then takes $\mu$ to $\beta$ and $\beta$ to zero.  $\psi$ can be extended uniquely to a degree one coderivation $\dCE$ of $\cofree(M)$ by Lemma~\ref{lemma:cofreecoder}.  We will construct a degree $-1$ chain homotopy $H$ of graded $\sS$-modules on $\cofree(M)$ so that $\dCE H+H\dCE$ is the identity outside arity $1$ and the zero map on arity $1$ (which is one dimensional, spanned by a representative of the coimage of the counit map).

To do this, we will need a combinatorial factor.
\begin{defi}\label{defi:weightfunction}
Let $T$ be a binary tree.  The vertex $v$ has some number of leaves $m_v$ above one of its incoming edges, and another number $n_v$ above the other (we need not concern ourselves which is which).  Let the weight $\omega(v)$ be their product $m_vn_v$.
For an illustration, see Figure~\ref{fig:weight}. 
	\end{defi}
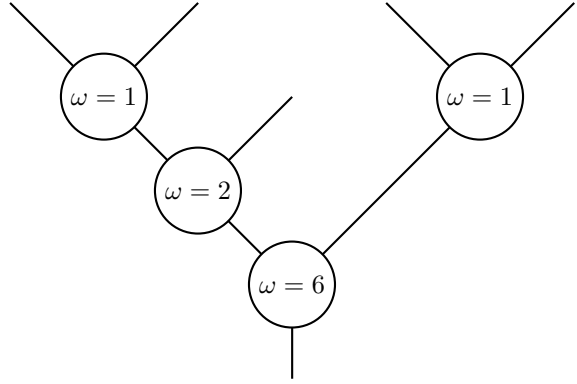
\begin{figure}[h]
\centering

\psset{unit=2.5cm}
\begin{pspicture}(-.5,-.5)(3.5,2.5)
\psline(1.5,0)(1.5,.5)
\psline(1.5,.5)(1,1)
\psline(1.5,.5)(2.5,1.5)
\psline(1,1)(1.5,1.5)
\psline(1,1)(.5,1.5)
\psline(.5,1.5)(0,2)
\psline(.5,1.5)(1,2)
\psline(2.5,1.5)(2,2)
\psline(2.5,1.5)(3,2)
\rput(1.5,.5){\pscirclebox[fillstyle=solid]{$\omega=6$}}
\rput(1,1){\pscirclebox[fillstyle=solid]{$\omega=2$}}
\rput(2.5,1.5){\pscirclebox[fillstyle=solid]{$\omega=1$}}
\rput(.5,1.5){\pscirclebox[fillstyle=solid]{$\omega=1$}}
\end{pspicture}

\caption{A binary tree with the weight $\omega$ indicated at each vertex}
\label{fig:weight}
\end{figure}
J.-L. Loday used this weight function to describe a parameterization of the Stasheff associahedra.
	\begin{lemma}\label{lemma:weightsum}\cite{Loday04a}
		The sum of the weights of all the vertices of a binary tree with $n$ vertices is $\binom{n+1}{2}$.
	\end{lemma}

\begin{defi}
Let $h:M\to M$ be the degree $-1$ morphism of graded $\sS$-modules given by taking $\beta$ to $\mu$ and $\mu$ to $0$.  We will use $h$ to define the contracting homotopy $H$.  

Let the homotopy $H$ be defined on a decorated tree with $n$ vertices in $\cofree(M)$ as a sum over the vertices.  For the vertex $v$, the contribution to the sum is $\frac{\omega(v)}{\binom{n+1}{2}}$ times the decorated tree obtained by applying $h$ to $v$ (including the Koszul sign).  So it has a similar flavor to extending $h$ as a coderivation, but also includes combinatorial factors.
\end{defi}
\begin{lemma}\label{lemma:Hcontracts}
The map $\dCE H+H\dCE$ is zero in arity one and the identity in all other arities.
\end{lemma}
\begin{proof}
First, applied to the coidentity subspace of $\cofree(M)$, the degree zero part of $\cofree(M)(1)$, this sum is clearly zero.  

Next, consider a tree in $\cofree(M)$ with at least one vertex.  The map $H$ acts on it by taking the signed and weighted sum of replacing each $\beta$ with a $\mu$; the coderivation $\dCE$ acts by taking the signed sum over all the $\mu$ and  replacing it with a $\beta$.  To act first with one and then with the other means that either \begin{enumerate}\item The maps $H$ and $\dCE$ act on two distinct vertices of the tree, or
\item they act on the same vertex, changing it first from $\beta$ to $\mu$ or vice versa and then back, ending up with the same tree, with a combinatorial factor.
\end{enumerate}
The first type come in pairs, one from $\dCE H$ and one from $H\dCE$, with the same combinatorial factors.  They have the opposite sign, because the sign conventions for $H$ and $\dCE$ are the same, and in one of the cases, there is one more or fewer $\mu$ than the other in a position that induces a sign.  This means all of these terms cancel.  
	
For the second type, note first of all that the induced signs from $H$ and $\dCE$ will be the same sign acting on whichever vertex we have chosen, and every vertex will be acted on by precisely one of $H\dCE$ and $\dCE H$ nontrivially, depending on whether it begins decorated by $\beta$ or $\mu$, so the final result of acting in this way on every vertex will be the sum over all vertices of the underlying tree $T$:
	\begin{equation*}
	\sum_{v}\frac{\omega(v)}{\binom{n+1}{2}} T=T.
	\end{equation*}
	By Lemma~\ref{lemma:weightsum}, the sum of the coefficients over all the vertices is exactly one, which yields the desired result.
\end{proof}
\subsection{Characterizing the Koszul dual of the Gerstenhaber operad}
Consider the operad $\Ge$ governing Gerstenhaber algebras.  This operad has a presentation as $\free(\susp^{-1}M)/(R)$, where $R$ is a set of quadratic relations in $\bullet=\susp^{-1}\mu$ and $\la \; , \, \ra=\susp^{-1}\beta$.

The Koszul dual cooperad (See \cite[Section~$7.3$]{LodayVallette10}) $\Ge^\antishriek$ is a graded sub $\sS$-module of $\cofree(M)\subset \cofree(\susp\overline{\Ge})$, characterized by being the intersection of $\cofree(M)$ with the kernel of the degree $-1$ coderivation $d_2$ on $\B\Ge:=\cofree(\susp\overline{\Ge})$ induced by the infinitesimal composition map $\gamma_{(1)}:\cofreen{2}{\Ge}\to \Ge$.  

Applying $d_2$ to a decorated tree in $\cofree(M)$ gives a sum of trees, each of which has one special $4$-valent  vertex decorated by an element of $\susp \Ge^{(2)}$ obtained by contraction of one edge and composition of the associated two operations. The rest of the vertices are trivalent and decorated with an element of $\susp\Ge^{(1)}=M$.  Because $M$ is one dimensional in each degree, we can specify that each trivalent vertex is decorated by either $\mu$ or $\beta$, with an overall coefficient on the decoration of the special vertex. Then in order that two separate terms be in the same summand of $\cofree(\susp \overline{\Ge})$ so that they might cancel, the underlying trees must be the same and the decorations on each trivalent vertex must be the same.

\begin{defi}A {\em contraction tree} is a tree with one undecorated $4$-valent vertex and all other vertices trivalent and decorated by either $\mu$ or $\beta$. 
\end{defi}

Note that we can induce a fixed order on the leaves of the special $4$-valent vertex of a contraction tree by using the order on the leaves of the whole tree; order the leaves of the special vertex by the smallest number of a tree leaf above each one, see~\cite[Section~$8.2$]{LodayVallette10}

A sum of decorated trees $\sum c_T T\in\cofree(M)$ is in the kernel of $d_2$ if and only if once we sum over all possible edge contractions, any summands that have the same underlying contraction tree cancel with each other.

Therefore it is important to know which decorated binary trees $T$ can have the same contraction tree.  There are precisely three underlying trees that can give rise to a given contraction trees by an edge contraction, corresponding to the three distinct binary trees with two vertices:
  
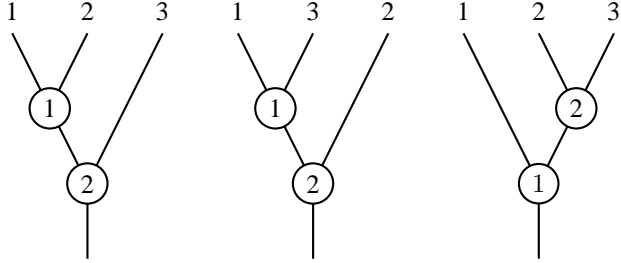
\[  
\psset{unit=2cm}
\begin{pspicture}(-1,-.5)(4,2)
\psline(0,0)(0,.5)
\psline(0,.5)(-.5,1.5)
\psline(0,.5)(.5,1.5)
\psline(-.25,1)(0,1.5)
\uput[u](-.5,1.5){1}
\uput[u](0,1.5){2}
\uput[u](.5,1.5){3}
\rput(-.25,1){\pscirclebox[fillstyle=solid]{1}}
\rput(0,.5){\pscirclebox[fillstyle=solid]{2}}

\psline(1.5,0)(1.5,.5)
\psline(1.5,.5)(1,1.5)
\psline(1.5,.5)(2,1.5)
\psline(1.25,1)(1.5,1.5)
\uput[u](1,1.5){1}
\uput[u](1.5,1.5){3}
\uput[u](2,1.5){2}
\rput(1.25,1){\pscirclebox[fillstyle=solid]{1}}
\rput(1.5,.5){\pscirclebox[fillstyle=solid]{2}}

\psline(3,0)(3,.5)
\psline(3,.5)(2.5,1.5)
\psline(3,.5)(3.5,1.5)
\psline(3.25,1)(3,1.5)
\uput[u](2.5,1.5){1}
\uput[u](3,1.5){2}
\uput[u](3.5,1.5){3}
\rput(3.25,1){\pscirclebox[fillstyle=solid]{2}}
\rput(3,.5){\pscirclebox[fillstyle=solid]{$1$}}
\end{pspicture}
\] 

 \begin{defi}
There are twelve distinct two-vertex binary trees with vertices decorated by $\mu$ and/or $\beta$, which form a $\KK$-linear basis for the twelve dimensional space $\cofreen{2}{M}$.   We will refer to these basis trees with the notation $\bt{a}{b}{i}$ where $a$ and $b$ are each one of the symbols $\mu$ and $\beta$ and $i$ is one of $1$, $2$, and $3$.  The numbers correspond, respectively, to the trees pictured above, while $a$ and $b$ are decorations of the two vertices, following the vertex order convention established in subsection~\ref{subsec:Trees}.  If $S$ is a contraction tree, then $S(\bt{a}{b}{i})$ is the decorated binary tree obtained by replacing the $4$-valent vertex with $\bt{a}{b}{i}$ and $S[\bt{a}{b}{i}]$ is the decorated tree obtained by decorating the $4$-valent vertex with $d_2 (\bt{a}{b}{i})$.
\end{defi}

Up to scale, there are precisely twelve decorated trees that can contract to yield the contraction tree $S$  if each vertex is given either $\mu$ or $\beta$ as a decoration; these are the trees $S(\bt{a}{b}{i})$.  These twelve are clearly in correspondence with the set $\{\bt{a}{b}{i}\}$.  We have shown:
\begin{lemma}The following are equivalent:
\begin{enumerate}
\item The element $\sum c_T T$ is in $\Ge^\antishriek$.
\item For each contraction tree $S$, $\displaystyle\sum c_{S(\bt{a}{b}{i})} S[\bt{a}{b}{i}]=0$ where the sum runs over the twelve decorated trees $S(\bt{a}{b}{i})$ that can yield $S$ as a contraction.
\item For each contraction tree $S$, $\displaystyle\sum c_{S(\bt{a}{b}{i})} d_2(\bt{a}{b}{i})=0$ where the sum runs over the twelve decorated trees $S(\bt{a}{b}{i})$ that can yield $S$ as a contraction.
\end{enumerate}
\end{lemma} 
In words, a sum of decorated trees can only be in $\Ge^\antishriek$ if there is local cancellation for every possible contraction tree.  Global cancellation is sufficient but local cancellation is necessary (this means that they are equivalent, but it will be easier to use local cancellation to check global cancellation in the sequel).
\subsection{Restricting the homotopy to $\Ge^\antishriek$}
\begin{lemma}\label{lemma:HforGe}
	The homotopy $H:\cofree(M)\to \cofree(M)$ restricts to $\Ge^\antishriek$.
\end{lemma}
\begin{proof}
Let $\sum c_T T$ be a sum of decorated trees in $\Ge^\antishriek$.  Then for every contraction tree $S$, the sum over the twelve basis elements
$\sum c_{S(\bt{a}{b}{i})}d_2 (\bt{a}{b}{i})$ is zero.

Let us consider applying $H$ to $\sum c_T T$.  By definition, this is a sum over every vertex of the decorated tree $T$.  To show that the resultant sum is in $\Ge^\antishriek$, we then apply $d_2$ and demonstrate that we get zero.  Applying $d_2$ involves applying the desuspension of the infinitesimal composition map $\gamma_{(1)}$ on each set of two adjacent vertices, summing over all such pairs.  We will confuse such subsets with internal edges, with which they are in bijection, as described in Section~\ref{subsec:codercofree}.

In total, to apply $H$ and then $d_2$ to a decorated tree $T$ involves summing over all choices of a vertex and edge of $T$; each individual summand is the application of first a weighted multiple of $h$ to the chosen vertex, and then infinitesimal composition $\gamma_{(1)}$ to the chosen edge.

This sum splits into those pairs of vertex and edge which are distinct, and those pairs where the chosen edge is incident on the chosen vertex.  We will show that each of these two constituent sums is zero individually.

If the vertex and edge are distinct, then, up to sign, the application of $h$ on the vertex and infinitesimal composition on the edge commute.  For a given contraction tree and choice of trivalent vertex on the contraction tree, the overall sign of commuting the shifted infinitesimal composition on the edge of one of the twelve decorated trees which yields the given contraction tree and $h$ on the corresponding vertex will be independent of the particular choice of decorated tree within the twelve.  
Let $S^v[\bt{a}{b}{i}]$ be obtained from $S[\bt{a}{b}{i}]$ by applying the appropriate weighted multiple of $h$ to the vertex $v$.  
Since the sum $\sum c_{S(\bt{a}{b}{i})}S[\bt{a}{b}{i}]$ over the twelve corresponding decorated trees without any application of $h$ is zero, for each choice of vertex, the sum $\sum c_{S(\bt{a}{b}{i})}S^v[\bt{a}{b}{i}]$ is also zero.

The other case to consider is when the edge involved in the contraction is incident on the vertex where $h$ is applied.  Let us fix a contraction tree $S$; because the original sum is in $\Ge^\antishriek$, it is true that $\sum c_{S(\bt{a}{b}{i})}S[\bt{a}{b}{i}]$ is zero, or, equivalently, $\sum c_{S(\bt{a}{b}{i})}d_2(\bt{a}{b}{i})$ is zero.  We will replace $d_2(\bt{a}{b}{i})$ with the weighted sum of applying $h$ to the top and bottom vertex of $\bt{a}{b}{i}$, followed by $d_2$, and show that the result is still zero. 

The two-vertex component of the Koszul dual cooperad to a quadratic operad is isomorphic to the space of relations, up to a degree shift. In this case, we have:
\begin{lemma}\label{lemma:Ge2}
The kernel of $d_2$ on the linear span of $\bt{a}{b}{i}$ is six dimensional, spanned by the shifted Gerstenhaber relations:
\begin{enumerate}
\item the two dimensional space of associativity relations $\bt{\mu}{\mu}{i}-\bt{\mu}{\mu}{j}$,
\item the three dimensional space of Leibniz relations spanned by 
\begin{gather*}
L_1=\bt{\mu}{\beta}{1}+\bt{\beta}{\mu}{2}+\bt{\mu}{\beta}{3},\\
L_2=\bt{\beta}{\mu}{1}+\bt{\mu}{\beta}{2}+\bt{\mu}{\beta}{3},\\
\text{and}\\
L_3=\bt{\beta}{\mu}{1}+\bt{\beta}{\mu}{2}+\bt{\beta}{\mu}{3}
\end{gather*}
(note that the signs are different than in the usual Leibniz relation because of the shift, and that the presentation is not symmetric in our basis), and
\item the one-dimensional space of the Jacobi relation $\bt{\beta}{\beta}{1}+\bt{\beta}{\beta}{2}+\bt{\beta}{\beta}{3}$.
\end{enumerate}
\end{lemma}

For the weighting of $h$, it is necessary to look at the shape and decorations of the contraction tree $S$.  Choose a representative so that the two vertices involved in the contraction edge are adjacent in the total ordering of vertices.  Let the number of leaves above the edge $i$ of the contraction vertex be $n_i$.  

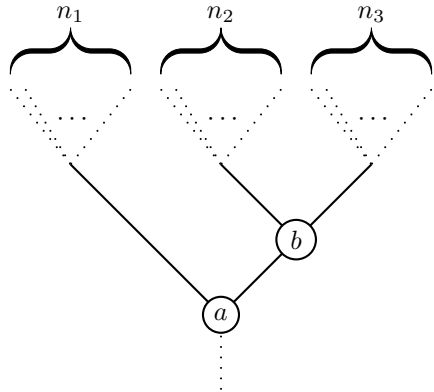
\begin{figure}[h]
\centering
\begin{equation*}
\psset{unit=2cm}
\begin{pspicture}(1.1,-.5)(4.9,3)
\psline[linestyle=dotted](3,0)(3,.5)
\psline(3,.5)(2,1.5)
\psline(3,.5)(4,1.5)
\psline(3.5,1)(3,1.5)
\psline[linestyle=dotted](2,1.5)(2.4,2)
\psline[linestyle=dotted](2,1.5)(1.6,2)
\psline[linestyle=dotted](2,1.5)(1.7,2)
\rput(2.025,1.8){$\cdots$}
\psbrace[rot=-90,ref=t,nodesepB=10pt](2.4,2.1)(1.6,2.1){}
\rput(2,2.5){$n_1$}
\psline[linestyle=dotted](3,1.5)(3.4,2)
\psline[linestyle=dotted](3,1.5)(2.6,2)
\psline[linestyle=dotted](3,1.5)(2.7,2)
\rput(3.025,1.8){$\cdots$}
\psbrace[rot=-90,ref=t,nodesepB=10pt](3.4,2.1)(2.6,2.1){}
\rput(3,2.5){$n_2$}
\psline[linestyle=dotted](4,1.5)(4.4,2)
\psline[linestyle=dotted](4,1.5)(3.6,2)
\psline[linestyle=dotted](4,1.5)(3.7,2)
\rput(4.025,1.8){$\cdots$}
\psbrace[rot=-90,ref=t,nodesepB=10pt](4.4,2.1)(3.6,2.1){}
\rput(4,2.5){$n_3$}
\rput(3.5,1){\pscirclebox[fillstyle=solid]{$b$}}
\rput(3,.5){\pscirclebox[fillstyle=solid]{$a$}}
\end{pspicture}
\end{equation*}\caption{The top part of $S(\bt{a}{b}{3})$}
\end{figure}
Then up to an overall sign and overall factor of $\frac{1}{\binom{n+1}{2}}$, the weighted sum of applying $h$ to both vertices of each of the basis elements is given by:
\begin{enumerate}
\item 
$\bt{\mu}{\mu}{i}\mapsto 0$,
\item 
$\bt{\mu}{\beta}{1}\mapsto n_3(n_1+n_2) \bt{\mu}{\mu}{1}$,\\
$\bt{\mu}{\beta}{2}\mapsto n_2(n_1+n_3) \bt{\mu}{\mu}{2}$,\\
$\bt{\beta}{\mu}{3}\mapsto n_1(n_2+n_3) \bt{\mu}{\mu}{3}$, 
\item 
$\bt{\beta}{\mu}{1}\mapsto -n_1n_2 \bt{\mu}{\mu}{1}$,\\
$\bt{\beta}{\mu}{2}\mapsto -n_1n_3 \bt{\mu}{\mu}{2}$,\\
$\bt{\mu}{\beta}{3}\mapsto -n_2n_3 \bt{\mu}{\mu}{3}$,\\
\item 
$\bt{\beta}{\beta}{1}\mapsto n_3(n_1+n_2) \bt{\beta}{\mu}{1}+n_1n_2\bt{\mu}{\beta}{1}$,\\
$\bt{\beta}{\beta}{2}\mapsto n_2(n_1+n_3) \bt{\beta}{\mu}{2}+n_1n_3\bt{\mu}{\beta}{2}$, and\\
$\bt{\beta}{\beta}{3}\mapsto n_1(n_2+n_3) \bt{\mu}{\beta}{3}+n_2n_3\bt{\beta}{\mu}{3}$.
\end{enumerate}
Applying these formulae to the kernel described above gives:
\begin{enumerate}
\item $\bt{\mu}{\mu}{i}-\bt{\mu}{\mu}{j}\mapsto 0$,
\item \begin{gather*}
L_1\mapsto
n_2n_3(\bt{\mu}{\mu}{1}-\bt{\mu}{\mu}{3})+n_1n_3(\bt{\mu}{\mu}{1}-\bt{\mu}{\mu}{2})
\\
L_2\mapsto
n_1n_2(\bt{\mu}{\mu}{2}-\bt{\mu}{\mu}{1})+n_2n_3(\bt{\mu}{\mu}{2}-\bt{\mu}{\mu}{3})
\\
L_3\mapsto
n_1n_2(\bt{\mu}{\mu}{3}-\bt{\mu}{\mu}{1})+n_1n_3(\bt{\mu}{\mu}{3}-\bt{\mu}{\mu}{2})
,\end{gather*} 
\item and: \begin{eqnarray*}
\bt{\beta}{\beta}{1}+\bt{\beta}{\beta}{2}+\bt{\beta}{\beta}{3}&\mapsto&\\&&
n_1n_2(\bt{\mu}{\beta}{1}+\bt{\beta}{\mu}{2}+\bt{\mu}{\beta}{3})\\&+&
n_1n_3(\bt{\beta}{\mu}{1}+\bt{\mu}{\beta}{2}+\bt{\mu}{\beta}{3})\\&+&
n_2n_3(\bt{\beta}{\mu}{1}+\bt{\beta}{\mu}{2}+\bt{\beta}{\mu}{3}).
\end{eqnarray*}
\end{enumerate}
So the kernel of $d_2$ in this twelve dimensional space is stable under the weighted application of $h$, no matter the particular trees that define the weights.  This means that for each contraction tree $S$, the sum obtained from $\sum c_{S(\bt{a}{b}{i})}d_2(\bt{a}{b}{i})$ by replacing $d_2(\bt{a}{b}{i})$ with the weighted sum of applying $h$ to the top and bottom vertex of $\bt{a}{b}{i}$, followed by $d_2$, is still zero, as desired.
\end{proof}

\subsection{Proof of Theorem~\ref{thm:MainDefRetract}}
\begin{lemma}\label{lemma:coderivation}
Let $\Oo=\free(N)/(R)$ be a quadratic operad with Koszul dual cooperad $\Oo^\antishriek\subset \cofree(sN)$.  Let $d$ be a coderivation of $\cofree(sN)$.  If the composition
\[
  \xymatrix{\Oo^\antishriek\ar@{>->}[r]& \cofree(sN)\ar[r]^-d& \cofree(sN)\ar@{-{>>}}[r]&\cofreen{2}{sN}/{\Oo^\antishriek}^{(2)}}\] is zero, then $d$ restricts to be a coderivation of $\Oo^\antishriek$.
\end{lemma}
\begin{proof}
Since the Koszul dual cooperead $\Oo^\antishriek=\Cc(sN,s^2 R)$ is a quadratic cooperad, this proof is dual to the proof that a derivation of the free operad $\free(N)$ passes to the quotient $\free(N)/(R)$, with $R\subset \free(N)^{(2)}$, if the composite 
\[
  \xymatrix{R\ar@{>->}[r]&\free(N)\ar[r]^d&\free(N)\ar@{-{>>}}[r]&\free(N)/(R)}\]
  is zero.
\end{proof}
\begin{cor}\label{cor:coderrestricts}
The coderivation $d_\psi$ defined on $\cofree(M)$ restricts to $\Ge^\antishriek$.  We will refer to the restriction with the same notation.
\end{cor}
\begin{proof}
In order to check this, we need check only that elements of $\Ge^\antishriek$ which $d_\psi$ takes into $\cofreen{2}{M}$ land in ${\Ge^\antishriek}^{(2)}$.  For degree reasons, such elements must belong to ${\Ge^\antishriek}^{(2)}$, which is described by Lemma~\ref{lemma:Ge2}.  A direct calculation verifies that $d_\psi$ takes an associativity relation to a difference of two Leibniz relations, takes each Leibniz relation to the Jacobi relation, and takes the Jacobi relation to zero.
\end{proof}
\begin{prop}\label{prop:contractible}
The counit map $(\Ge^\antishriek,\dCE)\to (\I,0)$ of the differential graded cooperad $(\Ge^\antishriek,\dCE)$, the coaugmentation $(\I,0)\to (\Ge^\antishriek,\dCE)$, and the homotopy $H$ form the following deformation retract:
\[
\xymatrix{
  *{ \quad \ \  \quad (\Ge^\antishriek, \dCE)\ } 
  \ar@(dl,ul)[]^{H}\ 
  \ar@<0.5ex>[r]& 
  (\I, 0)  
  \ar@<0.5ex>[l]
  }
 \]
\end{prop}
\begin{proof}
This is a direct corollary of Lemma~\ref{lemma:Hcontracts}, Lemma~\ref{lemma:HforGe}, and Corollary~\ref{cor:coderrestricts}.
\end{proof}
\begin{remark}
One can easily check that the dual of the chain complex $(\Ge^{\ac},d_\psi)$ is isomorphic to both the Koszul complex $ Lie^{\ac}\circ_\kappa  Lie$ (see \cite[Section~$7.4$]{LodayVallette10}) and the Chevalley-Eilenberg complex of the free Lie algebra.  This isomorphism along with the preceding proposition implies as a corollary the well-known facts that the Lie and commutative operads are Koszul, and that, equivalently, the Chevalley-Eilenberg homology of the free Lie algebra is trivial. 
\end{remark}
\begin{defi}
We define a map of $\sS$-modules
\begin{equation*}\theta: \cofree(\delta)\otimes \cofree(M)\to \cofree(M\oplus \KK\delta)\end{equation*}
as follows.  We will describe the image of $\delta^m\otimes x$ where $x$ has underlying tree $T$.  Let $\lambda$ range over assignments of a non-negative integer to each edge of $T$ so that the sum of all the integers is $m$.

Then the image of $\delta^m\otimes x$ has underlying tree $T'$ which is obtained from $T$ by inserting $\lambda(e)$ bivalent vertices on each edge $e$, labeled by $\delta$.
\end{defi}

\begin{lemma}\label{lemma:rhodistinverse}
The restriction of $\theta$ to $\cofree(\delta)\otimes \Ge^\antishriek$, still denoted $\theta$, is the inverse to the distributive isomorphism $\rho:\qBV^\antishriek\to\cofree(\delta)\otimes \Ge^\antishriek$.
\end{lemma}
\begin{proof}
First, let $x\in \Ge^\antishriek$.  We will verify that $\theta(\delta^n \otimes x)$ is in $\qBV^\antishriek$ by checking that $d_2\theta$ is zero on $\cofree(\delta)\otimes \Ge^\antishriek$ (here $d_2$ is the differential induced by composition in $\qBV$).

The map $\theta$ inserts vertices decorated by $\delta$, and $d_2$ composes pairs of adjacent vertices.  The sum involved in applying $d_2$ includes compositions involving $0$, $1$, and $2$ vertices decorated by $\delta$.  Each of these vanishes for a different reason.
\begin{enumerate}
\item The insertion of a vertex decorated with $\delta$ commutes up to sign with compositions that do not involve it, so inserting $m$ vertices decorated with $\delta$ and then contracting an edge whose vertices are decorated by $\mu$ or $\beta$ is the same as contracting the edge first and then inserting vertices decorated with $\delta$.  But since $d_2^{\Ge^\antishriek}$ coincides with $d_2^{\qBV^\antishriek}$ on the $\delta^0$ component of $\qBV^\antishriek$, and we are starting in the kernel of $d_2^{\Ge^\antishriek}$ to begin with, this summand is zero.
\item Contracting an edge whose vertices are both decorated by $\delta$ gives a bivalent vertex whose decoration is $s (\Delta \circ \Delta)$, which is zero since $\Delta\circ \Delta=0$ in $\qBV$. 
\item Finally, consider contracting an edge between a vertex $v$ decorated by a $\mu$ or $\beta$ and an adjacent vertex decorated by a $\delta$.  Let $\lambda'$ be a map from the edges of $T$ to the natural numbers so that the sum of the images adds to $m-1$.  There are precisely three choices of $\lambda$ with a $\delta$ adjacent to $v$ which can be forgotten to yield an element whose underlying tree is $T$ with vertices inserted according to $\lambda'$.  The sum of the three contractions with $v$ associated to $\lambda'$ together make up a relation of $\qBV$.  
\end{enumerate}

Now consider $\rhodist \theta(\delta^m\otimes x)$.  Because $\rhodist$ first decomposes and then projects, it is zero on any tree decorated by $\beta$, $\delta$, and $\mu$ unless all of the vertices decorated by $\delta$ are below all of the other vertices.  There is precisely one summand in the sum defining $\theta$ which satisfies this condition.  That is the summand corresponding to the partition $\lambda$ with $\lambda$ of the outgoing edge of the root equal to $m$ and $\lambda$ of every other edge equal to zero. The map $\rhodist$ splits this into two levels and then projects; the only way for the projection to be nonzero is for it to split with $\delta^m$ as the bottom level; then $\rhodist\theta(\delta^m\otimes x)=(\delta^m\otimes x)$.  Because $\rho$ is an isomorphism, a one-sided inverse is an inverse.
\end{proof}

\begin{lemma}\label{lemma:phipsi}
Under the above isomorphism $\theta$, 
%With the decomposition of $\qBV^\antishriek$ as $\cofree(\delta)\otimes \Ge^\antishriek$
 the differential   $\delta^{-1}\otimes d_\psi$ is sent to $d_\varphi$:
  \[
    (\qBV^\antishriek, d_\varphi)\cong (\cofree(\delta)\otimes \Ge^\antishriek,\delta^{-1}\otimes d_\psi)
    \]
\end{lemma}
\begin{proof}
It is enough to prove it on the level of the cofree cooperads. We show that the following diagram is commutative
$$\xymatrix{\cofree(\delta)\otimes \cofree(M)  \ar[r]^-\theta  \ar[d]^{\delta^{-1}\otimes d_\psi}  &  \cofree(M\oplus \KK\delta) \ar[d]^{\tilde{d}_\varphi}\\ \cofree(\delta)\otimes \cofree(M)   \ar[r]^\theta &  \cofree(M\oplus \KK\delta) \ ,   }$$
where $\tilde{d}_\varphi$ is the unique coderivation of the cofree cooperad $\cofree(M\oplus \KK\delta)$ which extends the map $\varphi$. Since $\delta^{-1}\otimes d_\psi$ is a coderivation, it is enough to prove it by projecting onto the space of cogenerators $M\oplus \KK \delta$. We conclude by showing that the only non-trivial component is 
$$\xymatrix{  
\begin{pspicture}(-.7,.43)(.7,1.1)
\psline(0,0)(0,.3)
\psline(.2,.7)(.5,1)
\psline(-.2,.7)(-.5,1)
\rput(0,.5){$\mu$}
\uput[d](0,.1){${\delta}$}
\end{pspicture}  \ar@{|->}[r] & 
\begin{pspicture}(-.7,.43)(.7,1.1)
\psline(0,0)(0,.3)
\psline(.2,.7)(.5,1)
\psline(-.2,.7)(-.5,1)
\rput(0,.5){$\mu$}
\uput[d](0,.1){${\delta}$}
\end{pspicture}
-
\begin{pspicture}(-.7,.43)(.7,1.1)
\psline(0,0)(0,.3)
\psline(.2,.7)(.5,1)
\psline(-.2,.7)(-.5,1)
\rput(0,.5){$\mu$}
\uput[u](-.5,.9){${\delta}$}
\end{pspicture}
-
\begin{pspicture}(-.7,.43)(.7,1.1)
\psline(0,0)(0,.3)
\psline(.2,.7)(.5,1)
\psline(-.2,.7)(-.5,1)
\rput(0,.5){$\mu$}
\uput[u](.5,.9){${\delta}$}
\end{pspicture}  \\
\ar@{|->}[d] & \ar@{|->}[d]\\
\begin{pspicture}(-.7,.43)(.7,1.1)
\psline(0,0)(0,.3)
\psline(.2,.7)(.5,1)
\psline(-.2,.7)(-.5,1)
\rput(0,.5){${\beta}$}
\end{pspicture} \ar@{|->}[r] & 
\begin{pspicture}(-.7,.43)(.7,1.1)
\psline(0,0)(0,.3)
\psline(.2,.7)(.5,1)
\psline(-.2,.7)(-.5,1)
\rput(0,.5){${\beta}$}
\end{pspicture}
 \ . } $$
\end{proof}

\begin{prop}\label{prop:HqBV}
	Under the isomorphism of Lemma~\ref{lemma:phipsi}, the chain complex $(\qBV^\antishriek,d_\varphi)$ admits a degree given by the powers $\delta^m$ of $\delta$, for which:
	\begin{equation*}
	\homol(\qBV^\antishriek)_{(m)}=\left\{\begin{array}{lll}\text{one dimensional, spanned by } (\delta^m\otimes \mathrm{I})&:&m>0\\\text{isomorphic to }1\otimes \Ge^\antishriek/\im \dCE&:&m=0.\end{array}\right.%TK
	\end{equation*}
\end{prop}

\begin{proof}
Write the chain complex as
\[
 \xymatrix{
 \cdots\ar[r]&\delta^m\otimes \Ge^\antishriek\ar[rr]^{\delta^{-1}\otimes \dCE}&&\delta^{m-1}\otimes\Ge^\antishriek\ar[r]&\cdots\ar[r]&\Ge^\antishriek\ar[r]&0}  
  \]
The homology is then one dimensional by Proposition~\ref{prop:contractible} everywhere except at $1\otimes \Ge^\antishriek$, where everything is in the kernel of the differential so the homology is just the quotient by the image of $\dCE$.
\end{proof}

\begin{proof}[Proof of Theorem~\ref{thm:MainDefRetract}]
We prove that the following data 
\begin{eqnarray*}
\xymatrix@C=13pt{     *{
({T}^c(\delta)\otimes{\G}^\antishriek,  \delta^{-1}\otimes d_\psi)  \  \ } \ar@<10ex>@(dl,ul)[]^{\delta\otimes H}\ \ar@<0.5ex>[r]^(0.48)p & *{\ 
({T}^c(\delta)   \oplus \im (H d_\psi), 0). \quad  }  \ar@<0.5ex>[l]  }
\end{eqnarray*}
form a deformation retract, where the projection map $p$ is the sum of the projection onto $T^c(\delta)$ and the projection onto $\G^{\ac}$ composed with $Hd_\psi$.
Assume that $x$ is in the coaugmentation coideal $\overline{\Ge^\antishriek}$.  Since $H$ is a contracting homotopy for $\dCE$, $(\dCE H+ H\dCE) x=x$.  Then $\dCE H\dCE x= -H\dCE^2 x + \dCE x= \dCE x$ so $(x-H\dCE x)$ is closed under $\dCE$.  Since $\Ge^\antishriek$ is contractible and $x$ is in the coaugmentation coideal, this means that $x-H\dCE x$ is in the image of $\dCE$, therefore in the image of $d_\varphi$.  This shows that $H\dCE x$ is in the same homology class as $x$.  It is independent of choice of representative because it gives zero on all of $\im \dCE$.  
A quick calculation verifies that $d_\varphi (\delta \otimes H)-(\delta\otimes H)d_\varphi$ gives the projection onto $\delta^m \otimes \overline{\Ge^\antishriek}$ except on the rightmost factor, where it gives $\id - H\dCE$.  This concludes the proof of the theorem, with the exception of the rightmost identification with the dual to the Gravity operad given in the next section.
\end{proof}

\subsection{The homology of $\BV^{\ac}$ in terms of the moduli space of curves and the Gravity operad}

Let us recall from E. Getzler's papers \cite{Getzler94bis, Getzler95} the definition of the quadratic operad $Grav$ encoding \emph{gravity algebras}. It is generated by skew-symmetric operations $[x_1,\ldots, x_n] $ of degree $2-n$ for any $n \ge 2$, which satisfy  the following relations:
\begin{eqnarray*}
\sum_{1 \leq i < j \leq k} \pm [[x_i, x_j], x_1, \ldots, \widehat{x_i}, \ldots,  \widehat{x_j}, \ldots, {x_k}, y_1, \ldots, y_l]=  \left\lbrace \begin{array}{lll}
[[x_1, \ldots, x_k], y_1, \ldots, y_l] & \text{for} &  l>0,\\
0  &  \text{for}&   l=0.
\end{array}
\right.
\end{eqnarray*}
The sign is the Koszul sign coming from the permutation of the elements. 

We consider the moduli space $\mathcal{M}_{0,n+1}$  of genus $0$ curves with $n+1$ marked points. 
The gluing along two points and the Poincar\'e residue map induce an operad structure on the suspension $sH_\bullet(\mathcal{M}_{0,n+1})$ of its homology, see \cite[Section~$3.4$]{Getzler95}. Let ${\mathcal S}^{-1}$ denote both the desuspension operad and cooperad structure on $\End_{\KK s^{-1}}.$

\begin{prop}[\cite{Getzler94bis}]\label{prop:Grav=HM}
The gravity operad is related to the homology of the moduli space of genus $0$ curves by the following isomorphism of operads: 
$${\mathcal S}^{-1}Grav \cong sH_\bullet(\mathcal{M}_{0,n+1}) \ . $$
\end{prop}

\begin{prop}\label{prop:getzler}
The quotient $\Ge^\antishriek/\im d_\psi$ is a cooperad isomorphic to ${\mathcal S}^{-1}Grav^*$.
\end{prop}
\begin{proof}
This is the cooperadic dual of Theorem~$4.5$ of  \cite{Getzler94}.  The aritywise linear dual of the differential graded quadratic cooperad $(\Ge^\antishriek, d_\psi)$, with degree $1$ coderivation, is a differential graded quadratic operad, with degree $1$ derivation. (We consider the opposite homological degree on the linear dual). By \cite[Theorem~$3.1$]{GetzlerJones94}, the underlying operad is isomorphic to $(\Ge^\antishriek)^* \cong {\mathcal S}^2 \Ge:=\End_{\KK s^2} \otimes_{\textrm{H}} \Ge$, which  admits the same quadratic presentation as the operad $\Ge$ except for the $-2$ degree shift of the generators $s^{-2}\bullet$ and $s^{-2}\langle\, , \rangle$. By the universal property of quadratic operads, the derivation ${}^t d_\psi$ is characterized by the images of these generators, that is 
$s^{-2}\bullet \mapsto s^{-2}\langle\, , \rangle$ and $s^{-2}\langle\, , \rangle \mapsto 0$. Therefore, up to the degree shift, the derivation ${}^t d_\psi$ is equal to the derivation $\Delta$ on $\Ge$ defined in \cite{Getzler94}. Theorem~$4.5$ of loc. cit. states that ${\mathcal S}^{-1} Grav\cong \Ker \Delta$. Dually, it gives $\Ge^\antishriek/\im \dCE\cong {\mathcal S}^{-1} Grav^*$.
\end{proof}
This concludes the proof of Theorem~\ref{thm:MainDefRetract}.

\begin{thm}\label{thm:HBV=Grav}
There exist isomorphisms of graded $\Sy$-modules
$$H_\bullet(\B\, \BV) \cong H_\bullet(\qBV^\antishriek, d_\varphi)\cong  \overline{T}^c(\delta) \oplus {\mathcal S}^{-1} Grav^*\ .$$
\end{thm}

\begin{proof}
The first isomorphism is a general fact about Koszul operads. In the case of an inhomogenous Koszul operad $\Po$, it is proved as follows. The degree $-1$ map $\qPo^{\ac} \epi sV \to V \mono \Po$ is a twisting morphism $\kappa : \Po^{\ac}=(\qPo^{\ac}, d_\varphi) \to \Po\in \Tw(\Po^{\ac}, \Po)$,  see \cite[Appendix~A]{GCTV09} or \cite[Section~$7.8$]{LodayVallette10}. By the general properties of the bar-cobar adjunction \cite[Section~$6.5$]{LodayVallette10}, it induces a morphism of dg cooperads $f_\kappa : \Po^{\ac} \to \B \, \Po$, which is equal to the following composite: 
$ \Po^{\ac} = \qPo^{\ac} \mono \F^c(sV) \to \F^c(s\oPo) =\B \, \Po$. On the right-hand side, the operad $\Po$ comes equipped with a filtration; we consider the induced filtration on the bar construction. On the left-hand side, we consider the filtration given by the weight grading on the cooperad $\qPo^{\ac}$. The coderivation $d_\varphi$ lowers this filtration by $1$ and the morphism $f_\kappa$ preserves the respective filtrations. By the Poincar\'e-Birkhoff-Witt theorem  \cite[Theorem~$39$]{GCTV09}, $\textrm{gr}\, \Po\cong \qPo$, the first page $(E^0, d^0)$ of the right hand-side spectral sequence is isomorphic to $\B\, \qPo$. So the map $f_\kappa$ induces the map
$f_{\bar \kappa} : (\qPo^{\ac}, 0) \to \B\,  \qPo$, on the level of the first pages of the spectral sequences, where $\bar \kappa$ is the twisting morphism associated to the  homogeneous quadratic operad $\qPo$. Since it is Koszul, the morphism $f_{\bar \kappa}$ is a quasi-isomorphism and we conclude by the convergence theorem of spectral sequences associated to bounded below and exhaustive filtrations \cite[Chapter~$11$]{MacLane95}.
The second isomorphism follows from Theorem~\ref{thm:MainDefRetract}.
\end{proof}

\begin{remarks}$\ $

\begin{itemize}
\item[$\diamond$] While we were writing this paper, V. Dotsenko and A. Khoroshkin in \cite{DotsenkoKhoroshkin09} proved, with another method (Gr\"obner bases for shuffle operads), the second isomorphism on the level of graded $\NN$-modules, i.e. without the action of the symmetric groups.

\item[$\diamond$] The cooperad $\Ge^{\ac}$ with the action of $d_\psi$  is the Koszul dual cooperad of the operad $\Ge$ with the action of $\Delta$ is the sense of Koszul duality theory of operads over Hopf algebras, see the Ph.D. Thesis of O. Bellier \cite{Bellier11} for more details. 
\end{itemize}
\end{remarks}

\section{The minimal model of the operad $\BV$}

In this section, we recall the notion of a homotopy cooperad, and we develop a transfer theorem for such structures across homotopy equivalences. We apply this result to  the deformation retract given in the previous section. This allows us to make the minimal model of the operad $\BV$ explicit. 

\subsection{Homotopy cooperad} We recall from \cite{VanderLaan02}  the notion of a {homotopy cooperad}, studied in more detail in \cite[Section~$4$]{MerkulovVallette09I}.

\begin{defi}[Homotopy cooperad]
A \emph{homotopy cooperad} structure on a graded $\Sy$-module $\Cc$ is the datum of a square-zero degree $-1$ derivation $d$ on the free operad $\free(s^{-1}\Cc)$ which respects the augmentation map. An \emph{$\infty$-morphism} $\Cc \rightsquigarrow \mathcal D$ of homotopy cooperads is a morphism of augmented dg operads between the associated  quasi-free operads 
$(\free(s^{-1}\Cc), d) \to (\free(s^{-1}\mathcal{D}), d')$. We denote this category by $\infty\textrm{-}\textsf{coop}_\infty$.
\end{defi}

We consider the isomorphism of $\Sy$-modules $\free(s^{-1}\Cc) \cong \free(\Cc)$ given by 
$$t(s^{-1}c_1, \ldots, s^{-1} c_n)\mapsto (-1)^{(n-1)|c_1|+(n-2)|c_2|+\cdots+|c_{n-1}|}  t(c_1, \ldots, c_n) \ . $$
Since the map $d$ is a derivation on a free operad, it is completely characterized by its image on generators 
$\Delta : \Cc \to \free(\Cc)$, under the above isomorphism.  The substitution of a tree $t$ at the $i^\textrm{th}$ vertex by a tree $t'$ is denoted by $t\circ_i t'$, see \cite[Section~$5.5$]{LodayVallette10} for more details.

\begin{prop}[\cite{MerkulovVallette09I}, Proposition~$24$]\label{prop:EquivDefHC}
The data of a homotopy cooperad $(\free(s^{-1}\Cc), d)$ is equivalent to a family of morphisms of $\Sy$-modules 
$\lbrace \Delta_t :  \Cc \to  t(\Cc) \rbrace_{t\in \Tree}$  such that 
\begin{itemize}
\item[$\diamond$] $\Delta_|=0$, 

\item[$\diamond$] the degree of $\Delta_t$ is equal to the number of vertices of $t$ minus 2, 

\item[$\diamond$] for every $c\in \Cc$, the number of non-trivial $\Delta_t(c)$ is finite, 

\item[$\diamond$] for every $c\in \Cc$,
$$\sum 
(-1)^{i-1+ k(l-i)}
 \, t \circ_i t' (c_1, \ldots, c_{i-1},
c'_1, \ldots, c'_k, c_{i+1}, \ldots, c_l)=0\ ,$$ 
where the sum runs
over the elements $t(c_1, \ldots, c_l)$ and $t'(c'_1, \ldots ,
c'_k)$ such that 
$$\Delta(c)=\sum_{t\in \Tree} \Delta_t(c)=\sum_{t\in\Tree} t(c_1, \ldots, c_l) \quad  \textrm{and} \quad 
\Delta(c_i)=\sum_{t'\in \Tree} \Delta_{t'}(c_i)=\sum_{t'\in\Tree} t'(c'_1, \ldots , c'_k) \ .$$
\end{itemize}
\end{prop}

A homotopy cooperad structure on a graded $\Sy$-module $\overline \Cc$ with vanishing maps $\Delta_t=0$ for trees $t\in \Tree^{(\ge 3)}$ with more than $3$ vertices is equivalent to a coaugmented dg cooperad structure on $\Cc:= \overline \Cc \oplus \I$. 
In this case, the definition in terms of a square-zero derivation on
the free operad is equivalent to the differential of the cobar
construction
 $\Omega\, \Cc$.

In the same way, the datum of an $\infty$-morphism $F : (\free(s^{-1}\Cc), d) \to (\free(s^{-1}\mathcal{D}), d')$ is equivalent to a morphism of $\Sy$-modules $f_\infty :  \Cc \to \free(\mathcal{D})$, that is, a family of morphisms $\lbrace f_t :  \Cc \to  t(\mathcal{D}) \rbrace_{t\in \Tree}$, satisfying some relations. An interpretation in terms of Maurer-Cartan elements is given in \cite[Section~$4.7$]{MerkulovVallette09I}.

The projection  $\Cc \to \free(\mathcal{C}) \epi \Cc$ of $d$ on the graded $\Sy$-module $\Cc$ endows it with a differential denoted by $d_\Cc$, which  is equal to the sum $d_\Cc=\sum \Delta_t$ over the corollas $t$. 
The images on corollas of any $\infty$-morphism  define a morphism of {dg} $\Sy$-modules $ 
(\Cc, d_\Cc) \to (\mathcal{D}, d_\mathcal{D})$. When this latter map is a quasi-isomorphism, the $\infty$-morphism  is called an \emph{$\infty$-quasi-isomorphism}.

\subsection{Homotopy transfer theorem for homotopy cooperads}

\begin{thm}\label{thm:HTTCoop}
Let $(\Cc, \lbrace \Delta_t \rbrace)$ be a homotopy cooperad. Let 
$(\mathcal{H}, d_\mathcal{H})$ be a dg $\Sy$-module, which is a homotopy retract of the dg $\Sy$-module $(\Cc, d_\Cc)$:
\begin{eqnarray*}
&\xymatrix{     *{ \quad \ \  \quad (\Cc, d_\Cc)\ } \ar@(dl,ul)[]^{h}\ \ar@<0.5ex>[r]^-{p} & *{\
(\mathcal{H}, d_\mathcal{H}) \ .\quad \ \  \ \quad }  \ar@<0.5ex>[l]^-{i}} &
\end{eqnarray*}
There is a homotopy cooperad structure on the dg $\Sy$-module $(\mathcal{H}, d_\mathcal{H})$, which extends the transferred composition maps $t(p)\circ \Delta_t\circ i$ and such that the map $p$ extends to an $\infty$-quasi-isomorphism.
\end{thm}

\begin{proof}
For any corolla $t$, the transferred structure map $\widetilde{\Delta}_t$ on $\mathcal{H}$ is given by the differential $d_\mathcal{H}$. For any tree $t\in \Tree$ with at least $2$ vertices, we consider all the possible ways of writting it by successive substitutions of trees with at least $2$ vertices: 
$$t=(((t_1\circ_{j_1} t_2)\circ_{j_2} t_3) \cdots )\circ_{j_k} t_{k+1}    \ .$$ 
The transferred structure map $\widetilde{\Delta}_t :  \mathcal{H}\to t(\mathcal{H})$ is then given  by 
$$
\widetilde{\Delta}_t:=\sum \pm\,   t(p)\circ \big( 
(\Delta_{t_{k+1}} h)     \circ_{j_k}   (\cdots   (\Delta_{t_3} h)     \circ_{j_2}  ((\Delta_{t_2} h)  \circ_{j_1} \Delta_{t_1} )   ) 
     \big)\circ i \ , $$
where the notation $(\Delta_{t'} h)  \circ_j \Delta_{t}$ means here the composite of $\Delta_t$ with $\Delta_{t'} h$ at the $j^\textrm{th}$ vertex of the tree $t$. 

The extension of the map $p :  \Cc\to \mathcal{H}$ into an $\infty$-morphism $p_\infty :  \Cc\to \free(\mathcal{H})$ is given by the same kind of formula. On corollas, it is given by the map $p$, and for any tree $t\in \Tree^{(\ge 2)}$ with at least $2$ vertices, it is given by 
$$
p_t:=\sum \pm\,   t(p)\circ \big( 
(\Delta_{t_{k+1}} h)     \circ_{j_k}   (\cdots   (\Delta_{t_3} h)     \circ_{j_2}  ((\Delta_{t_2} h)  \circ_{j_1} \Delta_{t_1} )   ) 
     \big)\circ h \ . $$

When $\Cc$ is a dg cooperad, these formulae are the exact duals to the ones  given by \cite{Granaker07} for dg (pr)operads. The rest of the proof is straightforward, following the ideas of  loc. cit.  
\end{proof}

\subsection{The homotopy cooperad structure on $H(\B\,  \BV)$}
Let us denote the graded $\Sy$-module 
$$\mathcal{H}:=\overline{T}^c(\delta) \oplus S^{-1}\overline{Grav}^* \cong H_\bullet (\overline{\B}\,  \BV)\cong H_\bullet(\overline{\qBV}^\antishriek, d_\varphi)\ .$$

Theorem~\ref{thm:MainDefRetract} provides us with the following deformation retract in the category of dg $\Sy$-modules:
\begin{eqnarray*}
\xymatrix{     *{ \qquad \qquad\qquad\qquad\qquad \qquad\qquad\quad 
({\qBV}^\antishriek\cong{T}^c(\delta)\otimes{\G}^\antishriek, d_\varphi\cong \delta^{-1}\otimes d_\psi)   \ } \ar@(dl,ul)[]^{h:=\delta\otimes H}\ \ar@<1ex>[r]^(0.77){p} & *{\
(\mathcal{H} \oplus \I, 0) \ . \quad \ \  \ \quad }  \ar@<1ex>[l]^(0.23){i}} &  & \qquad 
\end{eqnarray*}

\begin{cor}\label{cor:HomoCoopH}
The graded $\Sy$-module ${\cH}:=\overline{T}^c(\delta) \oplus S^{-1} \overline{Grav}^*$ is endowed with a homotopy cooperad structure and with an $\infty$-quasi-isomorphism from  the dg cooperad ${\BV}^{\ac}=({\qBV}^{\ac}, d_\varphi)$.
\end{cor}

\begin{proof}
This is a direct application of the Homotopy Transfer Theorem~\ref{thm:HTTCoop} for homotopy cooperads.
\end{proof}
 
\subsection{The minimal model of the operad $\BV$}

\begin{defi} A \emph{minimal operad} is a quasi-free dg operad $(\free(X), d)$
\begin{itemize}

\item[$\diamond$] with a decomposable differential, that is  $d : X \to \free^{(\ge 2)}(X)$, and 

\item[$\diamond$] such that the generating degree graded $\Sy$-module admits a decomposition into $X=\bigoplus_{k\ge 1} X^{(k)}$ satisfying $d(X^{(k+1)})\subset \free(\bigoplus_{i=1}^k X^{(i)})$.
\end{itemize}

A \emph{minimal model} of a dg operad $\Po$ is the data of a minimal operad $(\free(X), d)$ together with a quasi-isomorphism of dg operads $\xymatrix@C=18pt{(\free(X), d) \ar@{->>}[r]^(0.6)\sim& \Po}$, which is an epimorphism. (This last condition is always satisfied when the differential of $\Po$ is trivial). 
\end{defi}

The generalization of the notion of a minimal model  from dg commutative algebras 
\cite{DGMS75, Sullivan77} to dg operads was initiated by M. Markl in \cite{Markl96bis}, see also \cite[Section~$\textrm{II}.3.10$]{MSS}. Notice however that the aforementioned definition is strictly more general than loc. cit. 
and includes the crucial case of dg associative algebras, since we do not require that  $X(1)=0$ here. (A minimal operad in the sense of Markl is minimal in the present sense: the extra grading is given by the arity grading $X^{(k)} := X(k+1)$).
The present definition faithfully follows Sullivan's ideas: the increasing filtration $F_k:=\bigoplus_{i=1}^k X^{(i)}$ is the \emph{Sullivan triangulation} assumption. The extra grading $X^{(k)}$ is called the \emph{syzygy} degree. Notice that any non-negatively graded quasi-free operad with decomposable differential is minimal; one only has to
consider $X^{(k)} := X_{k-1}$.

The following lemma compares the two approaches of Quillen (cofibrant) and Sullivan (minimal) of homotopical algebra. 

\begin{lemma}
A minimal operad is cofibant in the model category given by V. Hinich \cite{Hinich97}.
\end{lemma}

\begin{proof}
This is a particular case of \cite[Corollary~$40$]{MerkulovVallette09II}.
\end{proof}

Since the definition is different, one needs a more general proof for the uniqueness of minimal models. 

\begin{prop}
Let $\Po$ be a dg operad. When it exists, the minimal model of the operad $\Po$ is unique up to isomorphism. 
\end{prop}

\begin{proof}
We work with the model category structure on dg operads defined by V. Hinich in \cite{Hinich97}. Let $\mathcal{M}$ and $\mathcal{M}'$ be  two minimal models of the graded operad $\Po$. They are cofibrant operads by the preceding proposition. Since  the quasi-isomorphism $\xymatrix@C=18pt{\mathcal{M}' \ar@{->>}[r]^(0.5)\sim& \Po}$  is an epimorphism, it is a trivial fibration. By the lifting property of a model category, there exists a quasi-isomorphism $f : \mathcal{M}=(\free(X), d)\qi \mathcal{M}'=(\free(X'), d')$ of dg operads. It induces a quasi-isomorphism of dg $\Sy$-modules between the space of generators $(X,d_X) \qi (X',d_{X'})$ by \cite[Proposition~$43$]{MerkulovVallette09I}. Since the differentials are decomposable, we get $d_X=0$ and $d_{X'}=0$. So the aforementioned quasi-isomorphism  is actually an isomorphism of graded $\Sy$-modules $X\cong X'$. Therefore, the map $f$ is an isomorphism of dg operads. 
\end{proof}

\begin{thm}
The data of Corollary~\ref{cor:HomoCoopH} provide us with the minimal model of the operad $\BV$:
$$  \big(\free(s^{-1}(\overline{T}^c(\delta) \oplus {\mathcal S}^{-1} \overline{Grav}^*)), {d}\big) \qi \BV  \ ,$$
where this quasi-isomorphism is  defined by $s^{-1}\delta \mapsto \Delta$ and by $s^{-1}\mu \mapsto \bullet$.
\end{thm}

\begin{proof}
First, the quasi-free operad  $\big(\free(s^{-1}(\overline{T}^c(\delta) \oplus {\mathcal S}^{-1} \overline{Grav}^*)), d \big)$ is minimal since it is non-negatively graded with the decomposable differential coming from the transferred homotopy cooperad structure on $\overline{\cH}$. 

Then, the $\infty$-quasi-isomorphism $p_\infty : \overline{\qBV}^{\ac} \rightsquigarrow{\cH}$ 
 of Corollary~\ref{cor:HomoCoopH} induces a morphism of dg operads 
$P : \Omega \, \BV^{\ac}  \to (\free(s^{-1}{\cH}),d)$. It is a quasi-isomorphism by the following argument. 
We consider the filtration $F_\bullet$ on $\Omega\, \BV^{\ac}$, and respectively $F'_\bullet$ on $(\free(s^{-1}{\cH}),d)$, given by the number of vertices of the underlying tree:
$$F_{-k}:= \bigoplus_{t \in \Tree^{(\ge k)}} t(s^{-1}\overline{\qBV}^{\ac})\quad \textrm{and} \quad 
F'_{-k}:= \bigoplus_{t \in \Tree^{(\ge k)}} t(s^{-1}{\cH})
 Ê\ .    $$
The first terms of the respective associated spectral sequences are $(E^0,d^0) \cong (\free(s^{-1} \overline{\qBV}^{\ac} ), d_1)$ and $(E'^{0}, d'^0)\cong (\free(s^{-1}{\cH}), 0)$. The morphism of dg operads $P$ preserves the aforementioned filtrations. Moreover, it satisfies $E^0(P)=\free({s^{-1}p})$. So it  is a quasi-isomorphism by the K\"unneth formula. The two filtrations are obviously exhaustive. At fixed arity, they are bounded below: for a fixed degree, the number of vertices is limited since the generator of arity one have degree greater or equal to $1$. 
We conclude the argument by means of the classical convergence theorem
for spectral sequences \cite[Chapter~$11$]{MacLane95}.

Finally, we define a morphism of operads $F : \free(s^{-1}{\cH}) \to \BV$  by 
$$s^{-1}\delta \mapsto \Delta, \quad  s^{-1}{\mathcal S}^{-1} \overline{Grav}^*(2) \cong s^{-1} \im H d_\psi (2) \cong \KK s^{-1} \mu \to  \KK \bullet \ ,  $$ 
and the rest being sent to $0$. We now check the commutativity of the differentials on the generators. It is straightforward on $s^{-1} \delta^m$. 

The only elements of $\im Hd_\psi$ whose image under $d$ are trees
with vertices labeled only by $\mu$ and $\delta$ are in $\im
Hd_\psi(3)$. Indeed, let $t$ be an element of $\im Hd_\psi(n)$, which
is the sum of trees with $k$ vertices labeled by $\mu$ and with
$n-1-k$ vertices labeled by $\beta$. To get trees labeled only by
$\mu$ and $\delta$, one has to apply $h=\delta\otimes H$ a total of
$n-1-k$ times. This introduces the $n-1-k$ power of $\delta$ and
applies the coproduct of the cooperad $T^c(\delta)\otimes \G^{\ac}$ a
total of $n-k$ times. In the end, we get trees labeled by $n-1$ copies
of $\mu$ and $n-1-k$ copies of $\delta$ split $n-k$ times. To get
totally split trees, we should have $n-k=2n-3-k$, which implies $n=3$.

The one-dimensional  space $ Lie_1^{\ac}(3)$, generated by the Jacobi relation, lives in $\im d_\psi=\Ker d_\psi$. The image under $d$ of the corresponding element in $H  Lie_1^{\ac}(3)$ is a sum of $7$ trees with $3$ vertices ($d_3$), whose image in the operad $\BV$ is the $7$-term relation 
$$ \Delta(\textrm{-} \bullet \textrm{-} \bullet \textrm{-})  + 
( \Delta(\textrm{-} \bullet \textrm{-}) \bullet \textrm{-}).(\id + (123) + (321))+
( \Delta(\textrm{-}) \bullet \textrm{-} \bullet \textrm{-}).(\id + (123) + (321))=0
 \ , $$
 which is a consequence of the definition of the operad $\BV$.
The two-dimensional space $Com^{\ac}(3)$ is generated by (the suspension of) the associators of $\bullet$. The composite $H d_\psi$ acts on it as the identity. Its image under $d$ is equal to $d_2$, which produces the associativity relation in the operad $\BV$. So the map $F : (\free(s^{-1}{\cH}),d) \to \BV$ is a 
morphism of dg operads. 

It remains to show that the following diagram is commutative
$$\xymatrix{\Omega \BV^{\ac} \ar[rr]^{\sim} \ar[rd]^{\sim}_P & & \BV \\
& (\free(s^{-1}{\cH}),d) \ , \ar[ur]_F   &} $$
to conclude that  $F$ is a quasi-isomorphism. It is enough to check it on the generators, which is equivalent to the commutativity of  the following diagram 
$$\xymatrix{ (\qBV^{\ac}, d_\varphi) \ar[r]^{\kappa} \ar[d]_{p_\infty}  & \BV \\
 (\free(s^{-1}{\cH}),d)  \ar@{->>}[r]   &   (\free(\KK s^{-1} \delta \oplus \KK  s^{-1}\mu),0)  \ar[u]_F     \ .    } $$
It is easily checked on $\delta$, $\mu$, and $\beta$. Both maps vanish on the rest of $\qBV^{\ac}$ by the same arguments as above: 
the only element which produces a non-trivial element in $\free(\KK s^{-1} \delta \oplus \KK  s^{-1}\mu)$ under $p_\infty$ is $\beta$, which concludes the proof. 
\end{proof}

We denote by 
$ d_n : s^{-1}{\cH} \to \free(s^{-1}{\cH})^{(n)}$ the part of the differential $d$, which splits elements into $n$ pieces. The component $d_2$ coincides with the decomposition map  on the cooperad $S^{-1}\overline{\Grav}^*$.

\begin{prop}\label{prop:Shapeofd}
The differential of the minimal model of the operad $\BV$ has the following shape: 
\begin{eqnarray*}
d_2 : s^{-1}\delta^m \mapsto \sum_{m_1+m_2=m}  s^{-1}\delta^{m_1} \otimes s^{-1}\delta^{m_2} 
\quad \textrm{and} \quad d_n : s^{-1}\delta^m \mapsto 0, \ \textrm{for} \ n\ge 3. 
\end{eqnarray*}
Up to the desuspension $s^{-1}$, the image of an element of degree $k$ of  ${\mathcal S}^{-1}{\overline{Grav}_k}^*$ under the map $d_n$ is a sum of trees with $n$ vertices 
labeled by elements of ${\mathcal S}^{-1}\overline{Grav}^*$ and of $\overline{T}^c(\delta)$, such that the total degree of the elements from 
${\mathcal S}^{-1}\overline{Grav}^*$ is equal to $k-n+2$ and such that the total weight, i.e. the total power, of the elements coming from $\overline{T}^c(\delta)$ is equal to $n-2$. For instance, this induces  
$$d_n({\mathcal S}^{-1}{\overline{Grav}_k}^*)=0 \quad \textrm{for}\quad  n>k+2 \ .$$
\end{prop}

\begin{proof}
By direct inspection of the various formulae.
\end{proof}

We denote the minimal model of the operad $\BV$ by 
$$BV_\infty:=\big(\free(s^{-1}(\overline{T}^c(\delta) \oplus {\mathcal S}^{-1} \overline{Grav}^*)), {d}\big)\ .$$

\begin{remarks}$Ê\ $

\begin{itemize}
\item[$\diamond$] The results about TCFT, two-fold loop spaces, and the cyclic Deligne conjecture, obtained in \cite{GCTV09} using the cofibrance property of the Koszul resolution of the operad $\BV$ hold as well with this minimal model. The proof of the Lian-Zuckerman conjecture with this minimal model requires further work and will be the subject of another paper.
\item[$\diamond$] The same method can be applied to \cite{HoefelLivernet11} to make explicit the minimal model of the inhomogeneous quadratic operad $H_0(\mathcal S \mathcal C)$ , where the operad $\mathcal S \mathcal C$ is Kontsevich Swiss-cheese operad. 
\end{itemize}
\end{remarks}

\section{Skeletal homotopy BV-algebras}

We call algebras over the minimal model of the operad $\BV$ skeletal homotopy
BV-algebras. We make this notion explicit and we give a description in terms of Maurer-Cartan elements in a homotopy Lie algebra. 

\subsection{Second definition of homotopy BV-algebras}

\begin{defi}
A \emph{skeletal homotopy Batalin-Vilkovisky algebra} is an algebra over the minimal operad $BV_\infty$. 
\end{defi}

Recall that a skeletal homotopy BV-algebra structure on a dg module $(A,d_A)$ is the datum of a morphism of dg operads $BV_\inftyÊ\to \End_A$. The differential $\partial_A$ of the operad $\End_A$ is equal to 
$\partial_A(f):=d_A \circ f - (-1)^{|f|}\sum_{i=1}^n f\circ_i d_A$. We denote $\widetilde{\mu}$ the image of an element $\mu$ of $BV_\infty$ into $\End_A$.

\begin{prop}\label{prop:RedHoBValg}
A skeletal homotopy BV-algebra  is a chain complex $(A,d_A)$ endowed with operations 
$$ \Delta^m : A \to A, \  \text{of degree}\ 2m-1,\ \text{for}\ m\ge 1, $$
and 
$$  \widetilde{\mu} : A^{\otimes n} \to A, \ \text{of degree}\  |\mu|+n-2,
 \ \text{for any} \  \mu \in \overline{Grav}^*(n), $$ 
such that 
$$\partial_A(\Delta^m)=\sum_{m_1+m_2=m} \Delta^{m_1} \circ \Delta^{m_2}, \ 
\text{for} \ m\ge 1, $$
and
$$    \partial_A(\widetilde{\mu})=\sum \pm \, \widetilde{\mu}_1 \circ_i \widetilde{\mu}_2 +  \sum   t(\widetilde{\nu}_1, \ldots, \widetilde{\nu}_k, \underbrace{\Delta^{m_1}, \ldots, \Delta^{m_l}}_{\ge 1}) \ ,  $$
where the first sum runs over the decomposition product of the cooperad structure on $Grav^*$, $\Delta_{Grav^*}(\mu)=\sum  {\mu_1} \circ_i {\mu_2}$, and where the second sum corresponds to composites of at least three operations with at least one $\Delta^m$. 
\end{prop}

\begin{proof}
This is a direct corollary of Proposition~\ref{prop:Shapeofd}.
\end{proof}

A BV-algebra is a skeletal homotopy BV-algebra with vanishing operations $\Delta^m$, for $m\ge 2$, and  
$\widetilde{\mu}$, for $\mu \in Grav^*(n)$, $n\ge 3$. The aforementioned quasi-isomorphism 
$$P:   \BV_\infty:=\Omega \BV^{\ac}   \xrightarrow{\sim}    BV_\infty:=(\free(s^{-1}{\cH}), d)$$ shows how  a skeletal homotopy BV-algebra carries  a  homotopy BV-algebra. Theorem~$4.7.4$ of \cite{Hinich97} implies the functor 
$$ P^* : \textsf{skeletal homotopy BV-algebras} \to \textsf{homotopy BV-algebras}$$
induces equivalences of the associated homotopy categories 
$$\mathsf{Ho}(\textsf{homotopy BV-algebras})\cong \mathsf{Ho}(\textsf{skeletal homotopy BV-algebras})
\cong \mathsf{Ho}(\textsf{BV-algebras}) \ . $$

\smallskip

Recall that a \emph{hypercommutative algebra} \cite{KontsevichManin94, Getzler95} is a chain complex  equipped with a totally symmetric $n$-ary operation $(x_{1}, \ldots, x_{n})$ of degree $2(n-2)$ for any $n\geq 2$, which satisfy 
 $$\sum_{S_{1}\sqcup S_{2}=\{1,\ldots , n\}} ((a, b, x_{S_{1}}), c, x_{S_{2}} )= \sum_{S_{1}\sqcup S_{2}=\{1,\ldots , n\}} (-1)^{|c||x_{S_{1}}|}(a, (b, x_{S_{1}}, c), x_{S_{2}})\ ,$$
for any $n\ge 0$. We denote the associated operad  by $HyperCom$. It is  isomorphic to the homology operad of  the Deligne-Mumford-Knudsen compactification of the moduli space of genus $0$ curves $H_\bullet(\overline{\mathcal{M}}_{0, n+1})$. It is Koszul dual to the operad Gravity: $HyperCom^!\cong Grav$.

\begin{prop}\label{prop:redBValg-Hypercom}
A skeletal homotopy BV-algebra with vanishing operators $\Delta^m$, for $m\ge 1$, is a homotopy hypercommutative algebra. 
\end{prop}

\begin{proof}
This is a direct corollary of Proposition~\ref{prop:RedHoBValg} together with the fact that the operad $HyperCom$ is Koszul, that is $\Omega (\mathcal{S}^{-1} Grav^*) \qi HyperCom$, see \cite{Getzler95}.
\end{proof} 

In operadic terms, this means that 
$$\big(  s^{-1}\overline{T}^c(\delta)     \big) \mono BV_\infty \epi HyperCom_\infty $$ 
is a short exact sequence of dg operads, where $\big(  s^{-1}\overline{T}^c(\delta)     \big) $ is the  ideal of $BV_\infty$ generated by $s^{-1}\overline{T}^c(\delta)$. Equivalently, the short sequence of homotopy cooperads
$$\overline{T}^c(\delta) \mono \cH \epi {\mathcal S}^{-1}\overline{Grav}^*\cong sH^\bullet(\mathcal{M}_{0,n+1})$$ 
is  exact, i.e. $\cH$ is an extension of the (non-unital) cooperads $\overline{T}^c(\delta)=H^\bullet(S^1)^{\ac}$ and $sH^\bullet(\mathcal{M}_{0,n+1})=H_\bullet(\overline{\mathcal{M}}_{0,n+1})^{\ac}$.

\begin{thm}
The operad $HyperCom$ is a representative of the homotopy quotient of the operad $BV$ by $\Delta$ in the homotopy category of dg operads.
\end{thm}

\begin{proof}
Let $D_\infty$ denote $T(s^{-1}\overline{T}^c(\delta))$.  This is
the minimal resolution of the algebra of dual numbers $D$. The pushout
of  $ I \gets D_\infty \mono BV_\infty$ gives a representative of the
homotopy quotient of $BV$ by $\Delta$ since $D_\infty \mono BV_\infty$
is a cofibration and since all the operads in the diagram are
cofibrant, see \cite[Chapter~$15$]{Hirshhorn03}. A map from this
diagram to an operad is the same thing as a map of the generators of
$HyperCom_\infty$ that respects the differentials; since the
augmentation ideal of $D_\infty$ vanishes in any map from this
diagram, the differentials coincide with those of $HyperCom_\infty$.
So the image of $HyperCom$ in the homotopy category of dg operads
gives the homotopy quotient of $BV$ by $\Delta$.
\end{proof} 

We refer the reader to \cite{Markarian09, DotsenkoKhoroshkin09, KhoroshkinMarkarianShadrin11}  further studies on this topic. 
 This result on the level of homology allows us to conjecture that 
 the homotopy quotient of the framed little disk by the circle is the compactified
moduli space of genus zero stable curves
 $\overline{\mathcal{M}}_{0, n+1}$. This will be the subject of another paper. 

\begin{remark} Since the generators of $HyperCom_\infty$ form a cooperad,  one can define the notion of \emph{$\infty$-morphism of homotopy hypercommutative algebras} using \cite[$10.2$]{LodayVallette10}. In the case of the operad $BV_\infty$, to define the notion of \emph{$\infty$-morphism of skeletal homotopy BV-algebras},  one has to refine the arguments, using the homotopy pullback of endomorphism operads for instance. With these definitions, Proposition~\ref{prop:redBValg-Hypercom} shows that the category of homotopy hypercommutative algebras with $\infty$-morphisms is a subcategory of the category of skeletal homotopy BV-algebras with $\infty$-morphisms, but not a full subcategory. 
\end{remark}

\subsection{Maurer-Cartan interpretation}

Recall from \cite[Theorem~$28$]{MerkulovVallette09I} that the module of morphisms of $\Sy$-modules 
$$\Hom_\Sy({\cH}, \End_A):=\prod_{n\ge 1} \Hom_{\Sy_n}({\cH}(n), \End_A(n)) $$
carries an $L_\infty$-algebra structure, $\lbrace \ell_n \rbrace_{n\ge 1}$, given in terms of the homotopy cooperad structure on $\overline{\cH}$ by the formula: 
$$\ell_n(f_1, \ldots, f_n):=\sum_{t\in \Tree^{(n)}\atop \sigma\in \Sy_n}  \pm\,       {\gamma}_{\End_A} \circ  t(f_{\sigma(1)}, \ldots, f_{\sigma(n)})       \circ \Delta_t \ , $$ 
where ${\gamma}_{\End_A}$ is the composition map of operations of ${\End_A}$ and where the sign is the Koszul sign due to the permutation of the graded elements $\lbrace f_i\rbrace $. 

The solutions to the (generalized) Maurer-Cartan equation 
$$\sum_{n \ge 1} {\footnotesize \frac{1}{n!}}\, \ell_n(\alpha, \ldots, \alpha) =0 , \quad \textrm{with} \quad |\alpha|=-1,$$
in this convolution $L_\infty$-algebra are called the (generalized) \emph{twisting morphisms} and denoted by $\Tw_\infty({\cH}, \End_A)$. 

\begin{prop}\label{prop:ReducedHBV-TW}
There is a natural bijection 
$$ \Hom_{\mathsf{dg} \, \mathsf{op}}(BV_\infty, \End_A)\cong  \Tw_\infty({\cH}, \End_A) \ .$$
\end{prop}

\begin{proof}
This follows from Theorem~$54$ of \cite{MerkulovVallette09I}.
\end{proof}

This result gives an interpretation of skeletal homotopy BV-algebra structures in term of Maurer-Cartan elements in an $L_\infty$-algebra. \\

We denote by $(\im H d_\psi )^{[k]}$ the subspace of $\im H d_\psi$ spanned by the tree monomials with $k$ vertices labelled by $\mu$. 

\begin{lemma}\label{lem:WeightGrav} 
The isomorphism of Theorem~\ref{thm:HBV=Grav} preserves the respective gradings: 
$${\mathcal S}^{-1} {Grav^*}^{(k)}\cong (\im Hd_\psi)^{[k]} \ .$$
\end{lemma}

\begin{proof}
By direct inspection. 
\end{proof}

This result allows us to organize the operations of a skeletal homotopy BV-algebra into strata. The first stratum is described as follows. Since ${\mathcal S}^{-1} {Grav^*}^{(1)}\cong H  Lie_1^{\ac} $, in weight $1$, is equal to the trivial representation of $\Sy_n$, then we get  $\Hom_\Sy({{Grav}^*}^{(1)}, \End_A)\cong \Hom({S^c}^{(\ge 2)} (A), A)$, up to suspension. 

$\Delta_t$ lands entirely in weight one only when $t$ has
precisely two vertices.

So  a twisting element $\alpha$ vanishing outside the weight $1$ part of ${Grav}^*$ actually satisfies the truncated Maurer-Cartan equation $\partial \alpha +\frac{1}{2}\ell_2(\alpha, \alpha)=0$. This corresponds to the definition of a Frobenuis manifold in terms of a hypercommutative algebra structure, see Y.I. Manin \cite{Manin99}. 

\section{Homotopy bar-cobar adjunction}

In this section, we introduce a new bar-cobar adjunction between the category of augmented dg operads and the category of homotopy cooperads. This bar construction relies on the notion of a cofree homotopy cooperad, which we make explicit in terms of nested trees. 

\subsection{Cofree homotopy cooperad}

We consider now the category of homotopy cooperads with (strict) morphisms. 

\begin{defi}
A \emph{morphism} $f : (\Cc, \lbrace \Delta_t\rbrace) \to (\mathcal{D}, \lbrace \Delta'_t\rbrace)$ of homotopy cooperads is a morphism of graded $\Sy$-modules $\Cc \to \mathcal{D}$  which commutes with the structure maps. 
\end{defi}

A morphism of homotopy cooperads is an $\infty$-morphism with vanishing components $\Cc \to \free (\mathcal{D})^{(n)}$ for $n\ge 2$.
The associated category is denoted by $\textsf{coop}_\infty$. There is a forgetful functor 
$$\mathcal{U} : \textsf{coop}_\infty \to \textsf{dg-}\Sy\textsf{-Mod}, \quad (\Cc, \lbrace \Delta_t\rbrace) \mapsto (\Cc,  d_\Cc) \ , $$ 
which retains only the underlying dg $\Sy$-module structure of a homotopy cooperad.

\begin{defi}
A \emph{nested tree} is a  tree $t\in\Tree \backslash \lbrace | \rbrace $ equipped with a set of  subsets  of vertices $\lbrace T_i \rbrace_{i}$, called \emph{nests}, such that:
\begin{itemize}
\item[$\diamond$] each nest $T_i$ corresponds to a subtree of the tree $t$,
\item[$\diamond$]  each nest $T_i$ has at least two elements,
\item[$\diamond$] if $T_i\cap T_j\ne \emptyset$, then $T_i\subset T_j$ or $T_j\subset T_i$, and
\item[$\diamond$] the full subset corresponding to the tree $t$ is a nest as long as $t$ has more than one vertex.
\end{itemize}
The associated category is denoted by $\mathsf{Nested Tree}$. See Figure~\ref{Fig:NestedTree} for an example. 
\end{defi}

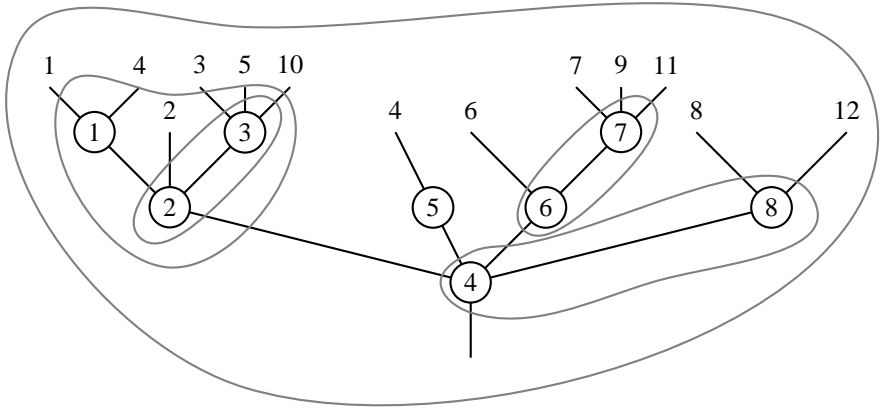
\begin{figure}[h]
\centering
\psset{unit=2cm}
\begin{pspicture}(-1.6,-.4)(4.3,2.5)
\psline(1.5,0)(1.5,.5)
\psline(1.5,.5)(1.25,1)
\psline(1.5,.5)(2,1)
\psline(1.5,.5)(3.5,1)
\psline(1.5,.5)(-.5,1)
\psline(-.5,1)(-1,1.5)
\psline(-.5,1)(-.5,1.5)
\uput[u](-.5,1.5){2}
\psline(-.5,1)(0,1.5)
\psline(-1,1.5)(-1.3,1.8)
\uput[u](-1.3,1.8){1}
\psline(-1,1.5)(-.7,1.8)
\uput[u](-.7,1.8){4}
\psline(0,1.5)(-.3,1.8)
\uput[u](-.3,1.8){3}
\psline(0,1.5)(0,1.8)
\uput[u](0,1.8){5}
\psline(0,1.5)(.3,1.8)
\uput[u](.3,1.8){10}
\psline(1.25,1)(1,1.5)
\uput[u](1,1.5){4}
\psline(2,1)(1.5,1.5)
\uput[u](1.5,1.5){6}
\psline(2,1)(2.5,1.5)
\psline(2.5,1.5)(2.2,1.8)
\uput[u](2.2,1.8){7}
\psline(2.5,1.5)(2.5,1.8)
\uput[u](2.5,1.8){9}
\psline(2.5,1.5)(2.8,1.8)
\uput[u](2.8,1.8){11}
\psline(3.5,1)(3,1.5)
\uput[u](3,1.5){8}
\psline(3.5,1)(4,1.5)
\uput[u](4,1.5){12}
\rput(1.5,.5){\pscirclebox[fillstyle=solid]{4}}
\rput(1.25,1){\pscirclebox[fillstyle=solid]{5}}
\rput(2,1){\pscirclebox[fillstyle=solid]{6}}
\rput(3.5,1){\pscirclebox[fillstyle=solid]{8}}
\rput(-.5,1){\pscirclebox[fillstyle=solid]{2}}
\rput(-1,1.5){\pscirclebox[fillstyle=solid]{1}}
\rput(0,1.5){\pscirclebox[fillstyle=solid]{3}}
\rput(2.5,1.5){\pscirclebox[fillstyle=solid]{7}}
\psccurve[linecolor=gray](4.2,1.6)(0,2.2)(-1.5,2.1)(0,-.2)
\psccurve[linecolor=gray](-1.2,1.8)(-.5,1.75)(.3,1.7)(-.5,.6)
\psccurve[linecolor=gray](-.7,.8)(-.4,1.4)(.2,1.7)(-.1,1.1)
\psccurve[linecolor=gray](1.85,.85)(2.1,1.4)(2.7,1.7)(2.4,1.1)
\psccurve[linecolor=gray](1.5,.3)(1.3,.5)(1.5,.7)(1.8,.75)(3.8,1)(2.7,.5)
\end{pspicture}

\caption{Example of a nested tree} \label{Fig:NestedTree}
\end{figure}

We consider the following total order on nests. The innermost nests are the largest ones. We compare them using their minimal element. Then we forget about these nests and proceed in the same way until reaching the full nest, which is the minimal nest. In the example of Figure~\ref{Fig:NestedTree}, it gives 
$$
T_1=\lbrace   1, 2, 3, 4, 5, 6, 7, 8     \rbrace <
T_2=\lbrace  1, 2, 3      \rbrace <
T_3=\lbrace       2, 3 \rbrace <
T_4=\lbrace      4, 8   \rbrace <
T_5=\lbrace       6, 7   \rbrace 
 \ .$$
 
 \smallskip

To any dg $\Sy$-module $(V, d_V)$, we associate the $\Sy$-module spanned by nested trees with vertices labeled by elements of $V$. It is denoted by 
$$\NT(V):=\bigoplus_{t\in \NeTr} t(V)\ . $$
Using the order on vertices given in Section~\ref{subsec:Trees} and the above order on nests, we write a simple element of $\NT(V)$ by 
$$ t(T_1, T_2, \ldots, T_N; v_1, v_2, \ldots, v_n)\ .$$
Its homological degree  is 
equal to 
$\sum_{k=1}^n |v_k|+N-n+1$. So the degree of a labeled corolla $t(v)$ is equal to $|v|$.

Two nests $T_j \subsetneq T_i$ are called \emph{consecutive} if $T_j \subset T_k \subset T_i$ implies either $T_k=T_i$ or $T_k=T_j$. We define a differential $d_\mathcal{N}$  by 
$$d_\mathcal{N}  (t):=\sum_{\text{consecutive pairs}\ T_j\subsetneq T_i} 
\pm \, 
t(T_1, \ldots, {T}_i, \ldots, {\widehat{T}_j}, \ldots,  T_N; v_1, \ldots, v_n)\ , $$
where the notation $\widehat{T_i}$ means that we forget the nest $T_i$. The sign is given as usual by the Koszul rule as follows. To every nest $T_i$, we associate the tree $t_i$ obtained from the subtree of $t$ defined by $T_i$ after contracting all its proper subnests. Each vertex thereby obtained is labeled by the least element of the contracted nest. 
The degree of a nest $T_i$ is equal to $|T_i|:=2-\# t_i$, where $\#t_i$ stands for the number of vertices of the tree $t_i$. 
(In the example of Figure~\ref{Fig:NestedTree}, one has $|T_1|=-2$.) 
If $T_j\subsetneq T_i$, then $i<j$. So we first permute $T_j$ with the nests $T_{j-1}, \ldots, T_{i+1}$ to bring it next to $T_i$. Then we apply the differential to the pair $(T_i, T_j)$, that is we forget about the nest $T_j$. This comes with a sign equal to $(-1)$ to the power $\#t_i + \#t_j + k + \textrm{des}(t_j,t_i)$, 
where $k$ is the number of vertices of $t_i$ smaller than the smallest vertex of $t_j$ and where $\textrm{des}(t_j,t_i)$ is the number of descents, that is the number of pairs $(a,b)$ of vertices of $t_j$ and $t_i$ respectively such that $a>b$. But the differential has to ``jump over'' the nests $T_1, \ldots, T_{i-1}$. In the end, it produces the sign 
$(-1)^\varepsilon$, with 
$$\varepsilon:={|T_1|+\cdots+|T_{i-1}|+|T_j|(|T_{i+1}|+\cdots+|T_{j-1}|) + \#t_i +\#t_j+ k + \textrm{des}(t_j,t_i)} \ .$$

We consider the differential on $\NT(V)$ given by the sum over all the vertices of the image of the labeling element of $V$ under $d_V$. By a slight abuse of notation, it is still denoted $d_V$:
$$d_V(t):=\sum_{i=1}^n 
(-1)^{N-n+1+|v_1|+\cdots+|v_{i-1}|} \, 
t(T_1, \ldots, T_N; v_1, \ldots, d_V(v_i), \ldots, v_n)\ . $$

\smallskip

We consider maps $\lbrace \Delta_t : \NT(V) \to t(\NT(V))
\rbrace_{t\in \Tree^{(\ge 2)}}$ defined as follows. Let $\tau$ be a simple element of $\NT(V)$. We consider the aforementioned tree $t_1$ associated to the full nest $T_1$, which is obtained by contracting all the subtrees corresponding to the interior nests. If $t\neq t_1$, then $\Delta_t(\tau):=0$. Otherwise, if $t= t_1$, the image of $\tau$ under $\Delta_t$ is equal to the tree $t_1$ with vertices labeled by the nested trees obtained from $\tau$ by forgetting its full nest.

\begin{prop}\label{prop:CofreeHoCoop}
For any dg $\Sy$-module $(V,d_V)$, the data $(\NT(V), d_V+d_\mathcal{N}, \lbrace \Delta_t \rbrace_{t\in \Tree^{(2)}})$ form a homotopy cooperad. This defines a functor $\NT : \mathsf{dg}\textsf{-}\Sy\textsf{-}\mathsf{Mod}\to \mathsf{coop}_\infty$ which is right adjoint to the forgetful functor 
$\mathcal{U} : \mathsf{coop}_\infty \to \mathsf{dg}\textsf{-}\Sy\textsf{-}\mathsf{Mod}$.
\end{prop}

\begin{proof}
The three first points of the equivalent definition of a homotopy cooperad given in Proposition~\ref{prop:EquivDefHC} are trivially satisfied by $\NT(V)$. The last point is  straightforward to check. 

Let $\Cc$ be a homotopy cooperad. We consider the morphism of $\Sy$-modules $\Delta^{\textrm{iter}} :  \Cc \to \NT(\Cc)$ defined as follows. 
For any tree $t$, the extra data given by the nests $t(T_1, \ldots, T_N)$ is equivalent to the decomposition of $t$ into successive substitutions 
$$t=(((t_1\circ_{i_1} t_2)\circ_{i_2} t_3) \cdots )\circ_{i_{N-1}} t_{N}    \ ,$$
where the trees $\lbrace t_i \rbrace$ are associated to the nests $\lbrace T_i\rbrace $ as defined above. 
The image of the map $\Delta^\textrm{iter}$ on a nested tree $t(T_1, \ldots, T_N)$ is defined by 
$$ \Delta^\textrm{iter}_t := \Delta_{t_{N}}      \circ_{i_{N-1}}   (\cdots   (\Delta_{t_3}      \circ_{i_2}  (\Delta_{t_2} \circ_{i_1} \Delta_{t_1} ))) \ . $$
Let $V$ be a dg $\Sy$-module. To any morphism of dg $\Sy$-modules 
$f : \mathcal{U}(\Cc) \to V$, we associate a morphism $F : \Cc \to \NT(V)$ defined by the composite 
$$F :=   \Cc \xrightarrow{\Delta^\textrm{iter}}  \NT(\Cc)    \xrightarrow{\NT(f)} \NT(V) \ .   $$
The map $F$ is a morphism of homotopy cooperads which satisfies the following universal property 
$$\xymatrix{  V    &  \ar@{->>}[l]  \NT(V) \\
 &   \ar[ul]^f  \Cc  \ar@{..>}[u]_{\exists ! \, F} \ ,    } $$
which concludes the proof. 
\end{proof}

Hence the homotopy cooperad $\NT(V)$ is called the \emph{cofree homotopy cooperad on $V$}. 

\begin{remarks} $ \ $

\begin{itemize}
\item[$\diamond$]
The endofunctor $\mathcal{U}\circ  \NT$ in 
$\mathsf{dg}\textsf{-}\Sy\textsf{-}\mathsf{Mod}$ can be endowed with a comonad structure: decompose a nested tree into all the possible ways of seeing it as a nested tree of nested subtrees. 
Proposition~\ref{prop:CofreeHoCoop} and its proof are equivalent to saying that the category of homotopy cooperads is the category of coalgebras over the comonad $\mathcal{U}\circ  \NT$.

\item[$\diamond$] Recall that the notion of an $A_\infty$-algebra can be encoded geometrically by the Stasheff polytopes, also called the associahedra. In the same way, the notion of a homotopy cooperad can be encoded  by a family of polytopes, defined by by means of \emph{graph associahedra} labelled by nested trees as introduced by M.P.  Carr and S.L.  Devadoss in 
\cite{CarrDevadoss06, DevadossForcey08}. Notice that this notion generalizes the nested sets of 
C. De Concini and C. Procesi \cite{DeConciniProcesi95}. 
For instance, the chain subcomplex of nested trees with fixed underlying tree $t$ should be isomorphic to 
the cochain complex
$$ (\NT_t, d_\N)\cong C^\bullet(\text{graph associahedron associated to}\  t)$$
This surely deserves further study, which we leave to a future work or to the interested reader. 
\end{itemize}
\end{remarks}

\subsection{Homotopy bar-cobar adjunction}

\begin{defi}
Let $(\Po, \gamma, d_\Po)$ be an augmented dg operad. The underlying $\Sy$-module of the  \emph{bar construction} $\B_\pi \Po$ is given by the cofree homotopy cooperad $\NT(s\oPo)$ on the suspension of the augmentation ideal of $\Po$. We define the differential $d_\gamma$ by 
\begin{multline*}
d_\mathcal{\gamma}  (t(T_1, \ldots, T_N, s\mu_1, \ldots, s\mu_n)):=\\
\sum_{\text{innermost}\ 
T_i=\lbrace  i_1, \ldots, i_k   \rbrace}
\pm \, 
t(T_1, \ldots, {\widehat{T}_i}, \ldots, T_N; s\mu_1, \ldots, 
s\gamma(t_i(\mu_{i_1}, \ldots, \mu_{i_k})), \ldots, \widehat{s\mu_{i_2}}, \ldots, \widehat{s\mu_{i_k}}, \ldots, 
s\mu_n).
\end{multline*}
We consider
$$\B_\pi \Po :=(\NT(s\oPo), d_\Po+d_\mathcal{N}+d_\gamma, \lbrace \Delta_t \rbrace_{t\in \Tree^{(\ge 2)}})\ . $$
\end{defi}

\begin{prop}
The data $(\NT(s\oPo), d_\Po+d_\mathcal{N}+d_\gamma, \lbrace \Delta_t \rbrace_{t\in \Tree^{(\ge 2)}})$ form a homotopy cooperad.
\end{prop}

\begin{proof}
Checking this is a straightforward calculation.
\end{proof}

\begin{defi}
The \emph{cobar construction} $\Omega_\pi \Cc$ of a homotopy cooperad $\Cc$ is the augmented dg operad $\Omega_\pi \Cc := (\free(s^{-1}\Cc), d)$. 
\end{defi}

\begin{thm}\label{thm:PiBarCobarAdj}
There are natural bijections 
$$ \Hom_{\mathsf{dg} \, \mathsf{op}}(\Omega_\pi \Cc, \Po)\cong  \Tw_\infty(\Cc, \Po)\cong 
\Hom_{\mathsf{coop}_\infty}(\Cc, \B_\pi \Po) \ .$$
In plain words, the pair of functors $\Omega_\pi$ and $\B_\pi$ are adjoint and this adjunction is represented by the twisting morphism bifunctor. 
\end{thm}

\begin{proof}
The first natural bijection is given by  \cite[Theorem~$54$]{MerkulovVallette09I}. The second one is described as follows. 
Proposition~\ref{prop:CofreeHoCoop} already provides us with a natural bijection 
$$\Hom_{\mathsf{coop}_\infty}\big(\Cc, (\NT(s\oPo), d_\mathcal{N},  \lbrace \Delta_t \rbrace_{t\in \Tree^{(\ge 2)}}    )\big)\cong \Hom_\Sy(\Cc, s\oPo), \quad F \mapsto f\ .$$ 
Under this bijection, a morphism of $\Sy$-modules $f : \Cc \to s\oPo$ induces a morphism of homotopy cooperads 
$F : \Cc \to \B_\pi \Po$ if and only if the following diagram commutes 
$$\xymatrix@C=40pt{\Cc \ar[r]^(0.4)F \ar[d]^{d_\Cc} & \NT(s\oPo) \ar[r]^{d_\Po +d_\gamma} & \NT(s\oPo) \ar@{->>}[d] \\
\Cc \ar[rr]^f & & s\oPo \ .} $$
This last condition is equivalent to $f d_\Cc= d_\Po f + \sum_{t\in \Tree^{(\ge 2)}} {\gamma} \circ t(f) \circ \Delta_t$, which is exactly the Maurer-Cartan equation 
$$\sum_{n \ge 1} {\footnotesize \frac{1}{n!}}\, \ell_n(s^{-1}f, \ldots, s^{-1}f) =0$$ satisfied by $s^{-1}f$ in the convolution $L_\infty$-algebra $\Hom_\Sy(\Cc, \Po)$. 
\end{proof}

\begin{remark}
The universal operadic twisting morphism $\pi : \B ({S} As) \to {S} As$ induces a pair of adjoint functors $\B_\pi$ and $\Omega_\pi$ between the category of dg associative algebras and the category of homotopy coalgebras by \cite{GetzlerJones94}, see also \cite[Chapter~$11$]{LodayVallette10}.
One can prove that it coincides with the restriction of the above bar and cobar constructions $\B_\pi$ and $\Omega_\pi$  to $\Sy$-modules concentrated in arity one, which  explains the notation.
\end{remark}

\section{Homotopy transfer theorem}

In this section, we prove the homotopy transfer theorem and the rectification theorem for skeletal homotopy BV-algebras. 

\subsection{Universal morphism of homotopy cooperads}
Let $(H, d_H)$ be a homotopy retract of a chain complex $(A, d_A)$:
\begin{eqnarray*}
&\xymatrix{     *{ \quad \ \  \quad (A, d_A)\ } \ar@(dl,ul)[]^{h}\ \ar@<0.5ex>[r]^-{p} & *{\
(H,d_H) \ . \quad \ \  \ \quad }  \ar@<0.5ex>[l]^-{i}}&
\end{eqnarray*}

Recall that the homotopy transfer theorem for homotopy algebras over a Koszul operad of 
\cite[Appendix~$\textrm{B}.3$]{GCTV09} and of \cite[Section~$10.3$]{LodayVallette10}
relies on the classical bar-cobar adjunction 
$$ \Hom_{\mathsf{dg} \, \mathsf{op}}(\Omega\,  \Po^{\ac},   \End_A)\cong  \Tw(\Po^{\ac}, \End_A)\cong 
\Hom_{\mathsf{dg} \, \mathsf{coop}}(\Po^{\ac}, \B\,  \End_A) $$
and on the quasi-isomorphism of dg cooperads 
$$\Psi : \B\, \End_A \qi \B\, \End_H$$ introduced by P. Van der Laan in \cite{VanDerLaan03}, see also \cite[Section~$10.3.3$]{LodayVallette10}. Such a map is characterized by its projection $\B\, \End_A =\cofree (s \End_A) \to s \End_H $ onto the space of generators. The Van der Laan map $\Psi$ is explicitly given by labeling the leaves of every tree by the map $i$, the root by the map $p$ and the interior edges by the homotopy $h$.

We consider the map $G(\End_A) : \B\, \End_A \to \B_\pi \End_A$ defined, for any $t\in \Tree$,  by 
$$t(sf_n, \ldots, s f_1)\in \cofree(s\End_A) \mapsto \sum \pm\,  t(T_1, \ldots, T_{n-1}; sf_n, \ldots, sf_1)\in \NT(s\End_A) \ , $$
 where the sum runs over all the maximal nestings, that is the ones with a maximal number of nests.  Since the bar construction $\B \, \End_A$ is a cooperad, it carries a homotopy cooperad structure; the map $G(\End_A)$ is a quasi-isomorphism of homotopy cooperads. 

\begin{prop}\label{prop:PHI}
Let $(H, d_H)$ be a homotopy retract of a chain complex $(A, d_A)$. 
There exists a quasi-isomorphism of homotopy cooperads 
$$\Phi : \B_\pi \End_A \qi \B_\pi \End_H $$
such that the following diagram, made up of quasi-isomorphisms of homotopy cooperads, is commutative,
$$\xymatrix@C=45pt{\B \, \End_A   \ar[r]^{G(\End_A)} \ar[d]^{\Psi} &  \B_\pi \End_A \ar[d]^{\Phi} \ \\
\B \, \End_H    \ar[r]^{G(\End_H)} &  \B_\pi \End_H \ .} $$
\end{prop}

\begin{proof}
Let us first give the proof in arity $1$; so here $\End_A=\Hom(A,A)$. We consider the quasi-isomorphism of cooperads $G : As^c \qi B ({\mathcal S} As)$. The map $G(\End_A)$ is equal to 
$$G(\End_A) = G \circ \id \ : \ As^c \circ_{\kappa'} {\mathcal S} As  \circ_{{\mathcal S} As} s\End_A     \qi B ({\mathcal S} As) \circ_\pi  {\mathcal S} As \circ_{{\mathcal S} As} s\End_AÊ\ , $$
where $\kappa' :={\mathcal S} \kappa : As^c = {\mathcal S} As^{\ac} \to {\mathcal S} As$ is the Koszul morphism coming from the Koszul duality of the operad $As$.  By the Comparison Lemma \cite[Lemma~$6.4.13$]{LodayVallette10}, the quasi-isomorphism $G$ induces a quasi-isomorphism 
$$\id \circ \,  G \circ \id \ : \ {\mathcal S} As \circ_{\kappa'} As^c \circ_{\kappa'} {\mathcal S} As       \qi {\mathcal S} As \circ _\pi B ({\mathcal S} As) \circ_\pi  {\mathcal S} As  $$
of quasi-free left ${\mathcal S} As$-modules (or equivalently of quasi-free anti-associative algebras in the category of $\Sy$-modules). 
By the left lifting property, it admits a homotopy inverse quasi-isomorphism 
$$F \ : \  {\mathcal S} As \circ _\pi B ({\mathcal S} As) \circ_\pi  {\mathcal S} As  \qi  {\mathcal S} As \circ_{\kappa'} As^c \circ_{\kappa'} {\mathcal S} As\ . $$
Under the bar-cobar adjunction,  the quasi-isomorphism of cooperads $\Psi$ is equivalent to the quasi-isomorphism 
of operads $\widetilde{\Psi} :  \Omega \B\,  \End_A \qi \End_H$. 
Finally, we define the morphism of homotopy cooperads $\Phi : \B_\pi \End_A \qi \B_\pi \End_H$ to be the map corresponding to the quasi-isomorphism of operads 
$$ \Omega_\pi \B_\pi \End_A \xrightarrow{F\circ_{{\mathcal S} As} s\End_A} \Omega \B\,  \End_A \xrightarrow{\widetilde \Psi} \End_H $$
under the homotopy bar-cobar adjunction. 

One extends these arguments to higher arity by using the colored Koszul operad  of \cite{VanDerLaan03}, which encodes operads, instead of the Koszul (non-symmetric) operad $As$ which encodes associative algebras. 

By definition, the following diagram is commutative 
$$\xymatrix@C=60pt{\B_\pi \End_A   \ar[r]^{F \circ_{S As} s\End_A}  \ar@/^2pc/[rr]^{\Phi}   &   \Omega \B \ \End_A \ar[r]^{\widetilde{\Psi}}   &  \End_H \\   
\B \ \End_A  \  , \ar[u]^{G(\End_A)}   \ar@{>->}[ur]  \ar@/_1pc/[urr]_{\Psi} &&  } $$
which concludes the proof. 
\end{proof}

The morphism of homotopy cooperads 
$\Phi : \B_\pi \End_A \qi \B_\pi \End_H $ is completely characterized by its projection onto the space of cogenerators, which we denote by 
$\phi : \NT(s\End_A) \to s\End_H$.

\subsection{Homotopy transfer theorem for skeletal homotopy BV-algebras}

\begin{thm}\label{thm:HTT}
Let A be a skeletal homotopy BV-algebra and let 
$(H, d_H)$ be a homotopy retract of the chain complex $(A, d_A)$:
\begin{eqnarray*}
&\xymatrix{     *{ \quad \ \  \quad (A, d_A)\ } \ar@(dl,ul)[]^{h}\ \ar@<0.5ex>[r]^-{p} & *{\
(H,d_H)\ . \quad \ \  \ \quad }  \ar@<0.5ex>[l]^-{i}}&
\end{eqnarray*}
There is a skeletal homotopy BV-algebra on $(H,d_H)$, which extends the transferred operations $p \widetilde{\mu} i^{\otimes n}$, for any $\mu \in  {\cH}$. If we denote by $\alpha \in \Tw_\infty({\cH}, \End_A)$ the skeletal homotopy BV-algebra structure on $A$, such a transferred skeletal homotopy BV-algebra structure on $H$ is given by 
$${\cH}\xrightarrow{\Delta^{\textrm{iter}}}  \NT({\cH})   \xrightarrow{\NT(s \alpha)} \NT( s\End_A) \xrightarrow{s^{-1}\phi} \End_H \ .$$
\end{thm}

\begin{proof}
We apply  the bar-cobar adjunction of Theorem~\ref{thm:PiBarCobarAdj} to 
$$ \Hom_{\mathsf{dg} \, \mathsf{op}}(\Omega_\pi{\cH},   \End_A)\cong  \Tw_\infty({\cH}, \End_A)\cong 
\Hom_{\mathsf{coop}_\infty}({\cH}, \B_\pi \End_A) \ .$$
So a skeletal homotopy BV-algebra structure $\alpha :{\cH}\to  \End_A$ on $A$ is equivalently given by a morphism of homotopy cooperads $F_\alpha :{\cH} \to \B_\pi \End_A$. The transferred skeletal homotopy BV-algebra on $H$ is then obtained by pushing along the morphism $\Phi$: 
$$\Phi \circ F_\alpha  : {\cH}\to \B_\pi \End_A \to \B_\pi \End_H
\ ,$$
which is equivalent to the following twisting morphism 
$${\cH}\xrightarrow{\Delta^{\textrm{iter}}}  \NT({\cH})   \xrightarrow{\NT(s \alpha)} \NT( s\End_A) \xrightarrow{s^{-1}\phi} \End_H \ .$$
\end{proof}

\begin{remark}
We proved the homotopy transfer theorem for homotopy BV-algebras, i.e. for the Koszul model $\BVK$ of the operad $\BV$ in \cite[Theorem~$33$]{GCTV09}. Since the $\Sy$-module of generators of the minimal model $BV_\infty$ forms a homotopy cooperad and not a cooperad, we cannot apply the arguments of \cite[Appendix~$\textrm{B}.3$]{GCTV09} and of \cite[Section~$10.3$]{LodayVallette10} based on the classical bar-cobar adjunction. Neither can we use the homological perturbation lemma of \cite{Berglund09}. Notice that the existence of the homotopy transferred structure follows from model category arguments by \cite{Rezk96, BergerMoerdijk03}. But we need here an explicit formula for the application to Frobenius manifolds in the next section.
\end{remark}

Needless to say that the Homotopy Transfer Theorem~\ref{thm:HTT} holds for any algebras over a quasi-free operad generated by a homotopy cooperad. In the case of a quasi-free operad generated by a dg cooperad, Koszul models or bar-cobar resolutions for instance, we recover the formulae of \cite{GCTV09} and of \cite[Chapter~$10$]{LodayVallette10} as follows.

\begin{prop} Let $\Po$ be a Koszul operad, eventually inhomogeneous.  Let $A$ be a homotopy $\Po$-algebra and let  $(H, d_H)$ be a homotopy retract of the chain complex $(A, d_A)$.

The transferred homotopy $\Po$-algebra structure on $H$ given by \cite[Theorem~$47$]{GCTV09} and by  \cite[Theorem~$10.3.6$]{LodayVallette10} is equal to the transferred homotopy $\Po$-algebra structure on $H$ given by Theorem~\ref{thm:HTT}.
\end{prop}

\begin{proof}
The proof relies on the following diagram being commutative:
$$\xymatrix@C=16pt{
\Tw(\Po^{\ac}, \End_A) \ar@{=}[d]  \ar[r]^(0.4){\cong} &  \Hom_{\mathsf{dg} \, \mathsf{coop}}(\Po^{\ac}, \B\,  \End_A) \ar[d]^{G(\End_A)_*} \ar[r]^{\Psi_*}  &  
\Hom_{\mathsf{dg} \, \mathsf{coop}}(\Po^{\ac}, \B\,  \End_H) \ar[d]^{G(\End_H)_*}  
\ar[r]^(0.58){\cong} & \Tw(\Po^{\ac}, \End_H) \ \   \ar@{=}[d]  \\
\Tw_\infty(\oPo^{\ac}, \End_A)  \ar[r]^(0.4){\cong} &  \Hom_{\mathsf{coop}_\infty}(\oPo^{\ac}, \B_\pi  \End_A)  
\ar[r]^{\Phi_*} &  
\Hom_{\mathsf{coop}_\infty}(\oPo^{\ac}, \B_\pi  \End_H) \ar[r]^(0.55){\cong}  & \Tw_\infty(\oPo^{\ac}, \End_H)\ . 
}$$
\end{proof}

The two homotopy transfer theorems for  homotopy BV-algebras and skeletal  homotopy BV-algebras 
commute under the functor 
$ P^* : \textsf{skeletal homotopy BV-algebras} \to \textsf{homotopy BV-algebras}$ as follows. 

\begin{prop}\label{prop:HTT-P*}
Let $(H, d_H)$ be a homotopy retract of a chain complex $(A, d_A)$. Consider a skeletal homotopy BV-algebra structure on $A$. The associated homotopy BV-algebra structure $P^*(A)$ on $A$ transfers to a homotopy BV-algebra to $H$ by  Theorem~$33$ of \cite{GCTV09}. This homotopy BV-algebra structure on $H$ is equal to the homotopy BV-algebra associated, under $P$, to the transferred skeletal BV-algebra given by Theorem~\ref{thm:HTT}.
\end{prop}

\begin{proof}
The proof relies on the commutativity of the following diagram:
$$\xymatrix@C=10pt{   
\Hom_{\mathsf{dg} \, \mathsf{op}}(\Omega_\pi \cH,   \End_A)   \ar[d]^{\cong}  \ar@/_10pc/[dddddd]_{P^*}&
\Hom_{\mathsf{dg} \, \mathsf{op}}(\Omega_\pi \cH,   \End_H)   \ar@/^10pc/[dddddd]^{P^*}  \\
\Hom_{\mathsf{coop}_\infty}(\cH, \B_\pi   \End_A)    \ar[r]^{\Phi_*} \ar@{>->}[d]& 
\Hom_{\mathsf{coop}_\infty}(\cH, \B_\pi   \End_H)  \ar@{>->}[d] \ar[u]^{\cong}\\
 \Hom_{\infty\textrm{-}\mathsf{coop}_\infty}(\cH, \B_\pi   \End_A)    \ar[r]^{\Phi_*} \ar[d]^{p_\infty^*}& 
\Hom_{\infty\textrm{-}\mathsf{coop}_\infty}(\cH, \B_\pi   \End_H) \ar[d]^{p_\infty^*}  \\
  \Hom_{\infty\textrm{-}\mathsf{coop}_\infty}(\overline{\BV}^{\ac}, \B_\pi    \End_A)    \ar[r]^{\Phi_*} & 
\Hom_{\infty\textrm{-}\mathsf{coop}_\infty}(\overline{\BV}^{\ac}, \B_\pi    \End_H) \\
 \Hom_{\mathsf{coop}_\infty}(\overline{\BV}^{\ac}, \B_\pi    \End_A)  \ar@{>->}[u]  \ar[r]^{\Phi_*} & 
\Hom_{\mathsf{coop}_\infty}(\overline{\BV}^{\ac}, \B_\pi    \End_H) \ar@{>->}[u]  \\
\Hom_{\mathsf{dg} \, \mathsf{coop}}(\BV^{\ac}, \B\,    \End_A)  \ar@{>->}[u]^{G(\End_A)_*}    \ar[r]^{\Psi_*} & 
\Hom_{\mathsf{dg} \, \mathsf{coop}}(\BV^{\ac}, \B\,    \End_H) \ar[d]^{\cong} \ar@{>->}[u]^{G(\End_H)_*}  &\\
\Hom_{\mathsf{dg} \, \mathsf{op}}(\Omega\, \BV^{\ac},   \End_A)   \ar[u]^{\cong}  &
\Hom_{\mathsf{dg} \, \mathsf{op}}(\Omega\, \BV^{\ac} ,   \End_H)  \ .
}
$$
\end{proof}

\subsection{Rectification theorem for skeletal homotopy BV-algebras}

We proved in \cite[Proposition~$32$]{GCTV09} the following Rectification Theorem:  for any homotopy BV-algebra $A$, there is an $\infty$-quasi-isomorphism $A \stackrel{\sim}{\rightsquigarrow} \Omega_\kappa \B_\iota A$ of homotopy BV-algebras, where $\Omega_\kappa \B_\iota A:=\BV ( \BV^{\ac}( A))$ is a dg BV-algebra. We refer to loc. cit. and to \cite[Chapter~$11$]{LodayVallette10} for more details. 

To every skeletal homotopy BV-algebra $H$, we define its \emph{rectified} dg BV-algebra by 
$$\mathrm{Rec}(H):= \Omega_\kappa \B_\iota P^*(H)\ .$$

\begin{thm}\label{prop:HoEquivTransf}
Let $(H, d_H)$ be a homotopy retract of a chain complex $(A, d_A)$. We consider a dg BV-algebra structure on $A$ together with the transferred skeletal homotopy BV-algebra on $H$ given by Theorem~\ref{thm:HTT}. 
The dg BV-algebra $\mathrm{Rec}(H)$
is homotopy equivalent to $A$ in the category of dg BV-algebras. 
\end{thm}

\begin{proof}
By Proposition~\ref{prop:HTT-P*}, the homotopy BV-algebra structure $P^*(H)$ is equal to the one produced by the homotopy transfer theorem for homotopy BV-algebras \cite[Theorem~$33$]{GCTV09}. Hence, there exists an $\infty$-quasi-isomorphism of homotopy BV-algebras $A \stackrel{\sim}{\rightsquigarrow} P^*(H)$ by Theorem~$10.4.7$ of \cite{LodayVallette10}. The Rectification Theorem for homotopy BV-algebras provides us with an $\infty$-quasi-isomorphism
$P^*(H) \stackrel{\sim}{\rightsquigarrow} \Omega_\kappa \B_\iota P^*(H)$. Finally, the two  dg BV-algebras 
$$A \stackrel{\sim}{\longleftarrow} \bullet \stackrel{\sim}{\longrightarrow}  \bullet\  \cdots\  \bullet 
\stackrel{\sim}{\longleftarrow} \bullet \stackrel{\sim}{\longrightarrow}  \Omega_\kappa \B_\iota P^*(H)=\mathrm{Rec}(H) $$
are linked by a zig-zag of quasi-isomorphism of dg BV-algebras by Theorem~$11.4.14$ of \cite{LodayVallette10}.
\end{proof}

This theorem gives homotopy control of the transferred structure. It plays a key role in the interpretation of the main result in the next section. 

\section{From BV-algebras to homotopy Frobenius manifolds}

We apply the Homotopy Transfer theorem to endow the underlying homology of a dg BV-algebra with Massey products. When the induced action of $\Delta$ is trivial, we recover and extend up to homotopy the Barannikov-Kontsevich-Manin Frobenius manifold structure. Applications of this general result are given
in Poisson geometry and Lie algebra cohomology and to the Mirror Symmetry conjecture.

\subsection{Massey products} 
Working over a field $\KK$, one can always write the underlying homology $(H_\bullet(A, d_A), 0)$ of a dg BV-algebra $A$ as a deformation retract of $(A,d_A)$.  

\begin{defi}
We call \emph{Massey-Batalin-Vilkovisky products} the operations composing the transferred skeletal homotopy BV-algebra structure on the homology $H(A)$ of a dg BV-algebra given by the Homotopy Transfer Theorem~\ref{thm:HTT}.
\end{defi}

Recall that the homology of any dg commutative (associative) algebra carries \emph{partial Massey products}, see \cite{Massey58}. For instance, the partial Massey triple-product $\langle x, y, z\rangle$ is defined for three homology classes $x,y,z \in H(A)$ such that $ x  y = 0 =  y z$ as follows. 
Let  $\bar x,\bar  y,\bar  z \in A$ be 
cycles which
represent $x$, $y$, and $z$ respectively
and let $a, b \in A$ such that $\bar x  \bar y = da , \bar y \bar z = db$. Then the chain $a\bar z - (-1)^{|\bar x|} \bar x b$
 is a cycle. So it defines an element $\langle x, y, z\rangle$ in  $H(A)/( x   H(A) + H(A)  z)$. When the partial Massey products are defined, they are given by the same formulae as the (uniform) Massey products, see \cite[Sections~$9.4$ and $10.3 $]{LodayVallette10}. For dg Lie algebras, partial Massey products were defined by V.S. Retakh in \cite{Retakh93}. The present Massey-Batalin-Vilkovisky products generalize both the partial commutative and Lie Massey products. 
 
% (General Massey products for algebras over a Koszul operads were defined in \cite{GCTV09, Berglund09, LodayVallette10}).
 
Theorem~\ref{prop:HoEquivTransf} shows that the data of the Massey products allow one to reconstruct the homotopy type of the initial dg BV-algebra.

%\begin{defi}
%A dg algebra $(A, d_A)$ over an operad $\Po$ is \emph{formal} (respectively \emph{smooth formal}) if it is related to its homology, equipped with the transferred $\Po$-algebra structure (respectively equipped with the trivial $\Po$-algebra structure), by a zig-zag of quasi-isomorphisms of dg $\Po$-algebras 
%$$(A, d_A)  \stackrel{\sim}{\longleftarrow} \bullet \stackrel{\sim}{\longrightarrow}  \bullet\  \cdots\  \bullet 
%\stackrel{\sim}{\longleftarrow} \bullet \stackrel{\sim}{\longrightarrow} (H_\bullet(A, d_A), 0) \ . $$
%\end{defi}

%\begin{prop}
%If a dg commutative algebra is formal (respectively smooth formal), then its higher (respectively all) partial Massey products vanish. 
%In the other way round, if the higher (respectively all) uniform Massey-Com products vanish, then the dg commutative algebra is formal (respectively smooth formal).  
%\end{prop}

%\begin{proof}
%The first assertion is proved in \cite[Proposition~$3.A.33$]{Huybrechts05}. The second one follows directly from 
%\cite[Propostion~$11.4.15$]{LodayVallette10}.
%\end{proof}

\subsection{Trivialization of the action of $\Delta$}\label{subsec:Trivialization}

\begin{prop}\label{prop:NCHdRHyperCom}
Let $A$ be a dg BV-algebra. If there exists a homotopy retract to the homology, which  satisfies 
$p (\Delta h)^{m-1}\Delta i=0$, 
 for $m\ge 1$, then the transferred skeletal homotopy BV-algebra on homology forms a homotopy hypercommutative algebra 
\end{prop}

\begin{proof}
The transferred operations $\Delta^m$ under Theorem~\ref{thm:HTT} are given by 
$\Delta^m:=p (\Delta h)^{m-1}\Delta i$. Then, one concludes with  Proposition~\ref{prop:redBValg-Hypercom}.
\end{proof}

A \emph{mixed chain complex} is a graded vector space $A$ equipped with two anti-commuting square-zero operators $d$ and $\Delta$ of respective degree $-1$ and $1$. 

\begin{defi}
Let $(A, d,\Delta)$ be a mixed chain complex.   {\em Non-commutative Hodge-to-de Rham degeneration data}  consists of a deformation retract 
\begin{eqnarray*}
&\xymatrix{     *{ \quad \ \  \quad (A, d)\ } \ar@(dl,ul)[]^{h}\ \ar@<0.5ex>[r]^{p} & *{\
(H(A),0)\ ,\quad \ \  \ \quad }  \ar@<0.5ex>[l]^{i}}
\end{eqnarray*}
 such that 
$$p (\Delta h)^{m-1}\Delta i=0 \ , $$
 for $m\ge 1$.
\end{defi}

\begin{defi}
The compatibility relation 
$$\Ker  d \cap \Ker \Delta \cap (\im d + \im \Delta)=\im d\Delta = \im \Delta d$$
between the operators $d$ and $\Delta$ of a mixed chain complex is called the \emph{$d\Delta$-condition}. 
\end{defi}

\begin{lemma}\cite[Proposition~$5.17$]{DGMS75}\label{lemma:dDEquiv}
A mixed chain complex $(A_\bullet, d, \Delta)$ satisfies the $d\Delta$-condition if and only if there exist two sub-graded modules $H_\bullet$ and $S_\bullet$ of $A_\bullet$ such that 
$$A_n\cong H_n \oplus S_n \oplus d S_{n+1} \oplus \Delta S_{n-1} \oplus d\Delta S_n \ , $$
where $d_{H_n}= 0$, $\Delta_{H_n}= 0$,  
and where the maps of the following commutative diagram are isomorphisms 
$$\xymatrix@C=30pt{S_n \ar[r]_{\Delta}^\cong \ar[d]^\cong_{d} & \Delta S_n \ar[d]^\cong_{d}\\
d S_n \ar[r]^\cong_{-\Delta} & d\Delta S_n\ .} $$
\end{lemma}

A dg BV-algebra, which satisfies this condition, is called a \emph{Hodge dg BV-algebra} by A. Losev and S. Shadrin in \cite{LosevShadrin07}. (In this case, the obvious homotopy $h$, which contracts $A$ to its homology $H$, is such that $\lbrack h, \Delta \rbrack = h\Delta + \Delta h=0$.)

\begin{defi}\cite{Park07}
A mixed chain complex is called \emph{semi-classical} if every homology class  has a representative in  the kernel of $\Delta$.
\end{defi}

\begin{prop}
Let $(A, d_A,\Delta)$ be a mixed chain. The following implications hold 
$$(d\Delta\text{-condition}) \Longrightarrow (\text{semi-classical}) \Longrightarrow (\text{NC Hodge-to-de Rham degeneration data}) \ . $$
\end{prop}

\begin{proof}
The first assertion is given by  Lemma~\ref{lemma:dDEquiv}. To prove the second one, it is enough to write the homology $H(A)$ as a deformation retract of $A$, with representatives in $\Ker \Delta$. In this case, $\Delta i=0$, which concludes the proof. 
\end{proof}

The existence of NC Hodge-to-de Rham degeneration
data is therefore the most general condition that naturally supports
this notion of the trivialization of the action of $\Delta$ on the
homology of a dg BV-algebra

\begin{examples}$ \ $
\begin{itemize}

\item[$\diamond$] Let $\mathcal M$ be a compact K\"ahler manifold, with complex structure denoted by $J$. The space of differential forms $(\Omega^\bullet(\mathcal M),  d_{DR}, \Delta:=J d_{DR} J)$ forms a dg BV-algebra which satisfies the $d\Delta$-condition, see P. Deligne, P. Griffiths, J. Morgan and D. Sullivan \cite{DGMS75}. (Notice that here the operator $\Delta$ has order less than $1$). 

\item[$\diamond$]
Let $\mathcal M$ be a Calabi-Yau manifold. The Dolbeault complex of anti-holomorphic differential forms with coefficients into 
holomorphic polyvector fields 
$(\Gamma(M, \wedge^\bullet \bar{T}_{\mathcal M}^* \otimes \wedge^\bullet {T}_{\mathcal M}), d:=\bar \partial,  \wedge, \Delta:=\textrm{div}, \langle \, , \rangle_S)$ is a dg BV-algebra satisfying the $d\Delta$-condition, see S. Barannikov and M. Kontsevich \cite{BarannikovKontsevich98}.  This is an extension, from vector fields to polyvector fields, of the Kodaira-Spencer dg Lie algebra \cite{KodairaSpencer58, KodairaSpencer60}, which encodes the complex structures of a manifold. 

\item[$\diamond$] Let $(\mathcal M, w)$ be  a Poisson manifold. The space of differential forms $(\Omega^\bullet(\mathcal M), d_{DR},  \wedge, \Delta:=[i_w, d_{DR}])$ form a dg  BV-algebra,  see \cite{Koszul85, Brylinski88}. 
When $(\mathcal M, \omega)$ is a compact symplectic manifold of dimension $n$, O. Mathieu proved in \cite{Mathieu95} that 
$\mathcal M$ satisfies  the hard Lefschetz condition, i.e. the cup product $[\omega^k] : H^{n-k}(M) \to H^{n+k}(M)$ is an isomorphism, for $k\leq n/2$, 
 if and only if this dg BV-algebra is semi-classical. S. Merkulov further proved that this is equivalent to the $d\Delta$-condition in \cite{Merkulov98}. This is the case when $\mathcal M$ is a K\"ahler manifold, see \cite{Brylinski88}.

\item[$\diamond$] Let $V$ be finite dimensional vector space with basis $\lbrace v_i\rbrace_{1\leq i\leq n}$. We consider the free commutative algebra $A:=S(V\oplus s^{-1}V^*)$ of functions on the cotangent bundle of $V^*$, equipped with  the order $2$ and degree $1$ operator 
$\Delta:=\sum_{i=1}^n \frac{\partial}{\partial v_i}\frac{\partial}{\partial v_i^*}$. These data define the prototypical example of BV-algebras, see \cite{BatalinVilkovisky81}. Any element $w$ of degree $-2$ such that $\Delta(w)=\langle w, w\rangle=0$ gives rise to a dg BV-algebra $(A, d_w:=\langle w, - \rangle, \bullet, \Delta, \langle\,  , \rangle )$.
One can find dg BV-algebras of this type equipped with  NC Hodge to de Rham degeneration data but which does not satisfy the $d\Delta$-condition, see \cite[Example~$9$]{Park07} and \cite[Section~$3.2$]{Terilla08}. 
\end{itemize}
\end{examples}

\subsection{Homotopy Frobenius manifold}

\begin{thm}\label{thm:HomotopyFrob}
Let $(A, d,  \bullet, \Delta, \langle\,  , \rangle)$ be a dg BV-algebra with non-commutative Hodge-to-de Rham degeneration data. 

The underlying homology groups $H(A,d)$ carry a homotopy hypercommutative algebra structure, which extends the hypercommutative algebras of M. Kontsevich and S. Barannikov \cite{BarannikovKontsevich98}, Y.I. Manin \cite{Manin99},  A. Losev and S. Shadrin \cite{LosevShadrin07}, and J.-S. Park \cite{Park07}, and such that 
the rectified  dg BV-algebra $\mathrm{Rec}(H(A))$
is homotopy equivalent to $A$ in the category of dg BV-algebras. 
\end{thm}

\begin{proof}
The transferred skeletal homotopy BV-algebra structure on homology given by  Theorem~\ref{thm:HTT} forms a homotopy hypercommutative algebra by Proposition~\ref{prop:NCHdRHyperCom}. 

We make explicit the various constructions of \cite{BarannikovKontsevich98} as follows. When a dg BV-algebra satisfies the $d\Delta$-condition, there is a zig-zag of quasi-isomorphisms of dg Lie algebras  (smooth formality)
$$\xymatrix{(A,d, \langle\,  , \rangle)   &  \ar[l]_(0.53){\sim} (\Ker \Delta, d, \langle\,  , \rangle) 
\ar[r]^(0.37){\sim} &  (H^\bullet(A, \Delta)\cong(H_\bullet(A, d)), 0, 0) }\ . $$
By \cite[Theorem~$10.4.7$]{LodayVallette10}, there exists an $\infty$-quasi-isomorphism of dg Lie algebras $H \stackrel{\sim}{\rightsquigarrow} \Ker \Delta$,  explicitly given by sums of binary trees with vertices labelled by $\bullet$ and with edges and root labelled by $h\Delta$. Normalizing each sum of trees of arity $n$ by a factor $\frac{1}{n!}$, this provides a solution $\gamma$ to the Maurer-Cartan equation in the dg Lie algebra $\Hom(\bar  S^c(H), \Ker \Delta)$, where $\bar S^c$ stands for the non-counital cofree cocommutative coalgebra. The twisted data $(\Hom(\bar  S^c(H), A), d_\gamma:=d + \langle \gamma, - \rangle, \bullet, \Delta, \langle\,  , \rangle)$ form a dg BV-algebra over the ring of formal power series $\widehat S(H^*)$ without constant term. Its homology with respect to $d_\gamma$ is equal to  $\Hom(\bar  S^c(H), H)\cong \widehat S(H^*) \otimes H$. The transferred commutative product on homology $\widehat S(H^*) \otimes H$ provides us with the desired hypercommutative algebra structure on $H$, see \cite[Chapters~$0$ and $3$]{Manin99} for the various equivalent definitions of a formal Frobenius manifold. Tracing through the aforementioned constructions, one can see that the associated potential is given by the same kind of sums of labelled trees but with a normalizing coefficient given by the number of automorphisms of the trees. We recover the explicit formula of \cite{LosevShadrin07}. Manin \cite{Manin99} and Park \cite{Park07} use obstruction theory, for which choices can be made to produce the above structure. 

The first stratum of operations composing the transferred homotopy hypercommutative algebra is equal to the tree formulae of Losev-Shadrin as follows. Lemma~\ref{lem:WeightGrav} shows that 
the weight $1$ part of $Grav^*$ is isomorphic  to $H Lie_1^{\ac}$. For any $n\ge 2$, the space $Lie_1^{\ac}(n)$ is one dimension and generated by the element, which in $\cofree(\beta)$ is the sum of all binary tree with vertices labeled by $\beta$. The image of such trees under the formula of Theorem~\ref{thm:HTT} is made up of  binary trees with each vertex labelled by $\bullet$, one  leaf labelled by $\Delta$, and with edges labelled by $h$. (One can see that the image of a maximal nesting under the map $\Phi$ is given by labeling all interior edges by $h$.)  
Under the $d\Delta$-condition, the relations $p\Delta=\Delta i=h\Delta+\Delta h=\Delta^2=0$ make many trees cancel  and this produces the aforementioned Losev-Shadrin formulae. 

The last assertion is a direct corollary of Theorem~\ref{prop:HoEquivTransf}.
\end{proof}

\begin{remarks}$ \ $

\begin{itemize}
\item[$\diamond$] 
First, this theorem conceptually explains the result of Barannikov-Kontsevich, Manin, Losev-Shadrin, and Park in terms of the homotopy transfer theorem, thereby answering a question asked by the referee of \cite[Section~$5$]{Park07}. 
%For instance, this gives another proof that the trees structure operations introduced in \cite{LosevShadrin07} satisfy the WDVV-equation. 

\item[$\diamond$]
Since there is no differential on homology, the first stratum of operations of this  homotopy hypercommutative algebra 
satisfies the relations of an hypercommutative algebra. So Theorem~\ref{thm:HomotopyFrob} 
proves the existence of such a structure under a weaker
condition (NC Hodge-to-de Rham degeneration data) than in \cite{BarannikovKontsevich98, Manin99, Park07} ($d\Delta$-condition, semiclassical). 

\item[$\diamond$]
Unlike the framework of Frobenius manifolds, we do not work here with cyclic unital BV-algebras. First, a cyclic BV-algebra is equipped with a non-degenerate bilinear form which forces its dimension to be finite. The present method works in the infinite dimensional case. Then, the operad which encodes BV-algebras with unit is not augmented, so it does not admit a minimal model. To make a cofibrant replacement explicit, one would need to use the more general Koszul duality theory developed by J. Hirsh and J. Mill\`es in \cite{HirschMilles10}.

\item[$\diamond$] 
Finally, Theorem~\ref{thm:HomotopyFrob} provides higher structure on homology, which  is shown to be \emph{necessary} to recover the homotopy type of the original dg BV-algebra and not to lose any homotopy data when passing to homology, see also Example~\ref{subsec:EX} below. 
\end{itemize}
\end{remarks}

In geometrical terms, we have lifted the action of the Deligne-Mumford-Knudsen moduli space of genus $0$ curves to an action of the open moduli space of genus $0$ curves as follows. 
$$\xymatrix@C=30pt{H^{\bullet+1}({\mathcal{M}}_{0, n+1})   \ar@{..>}[r]^(0.58)\alpha  \ar[d]^\kappa&   \End_{H(A)} \\
H_\bullet (\overline{\mathcal{M}}_{0, n+1})  \ar[ur]_{\ \quad f} \ .&  } $$
The map $f$ is the morphism of operads given by \cite{BarannikovKontsevich98, Manin99, LosevShadrin07, Park07}. The map $\kappa$ is the twisting Koszul morphism  from the cooperad $H^{\bullet+1}({\mathcal{M}}_{0, n+1})$ given in \cite{Getzler95}. It sends the cohomological class corresponding to $H_0({\mathcal{M}}_{0, n+1})$ to the fundamental class of $\overline{\mathcal{M}}_{0, n+1}$. 
The construction given in Theorem~\ref{thm:HomotopyFrob} corresponds to the map $\alpha$, which is 
a  twisting morphism from the cooperad $H^{\bullet+1}({\mathcal{M}}_{0, n+1}) $. 
The map $\kappa$ vanishes outside the top dimensional classes and the \emph{restriction} of the map $\alpha$ to these top dimensional  classes is equal to the composite $f \circ \kappa$. 
Such a morphism of operads $f$ defines the genus zero part of what Kontsevich-Manin call a  Cohomological Field Theory in \cite{KontsevichManin94}. 

\begin{defi}
An \emph{genus $0$ extended cohomological field theory} is a graded vector space $H$ equipped with an operadic twisting morphism 
$H^{\bullet+1}({\mathcal{M}}_{0, n+1}) \to \End_{H}$. 
\end{defi}

\subsection{An example}\label{subsec:EX}
Let us consider the following non-unital dg commutative algebra $A$ generated by the $5$ generators
$$x_3, \ y_3, \ z_7, \ u_7, \ \text{and} \ v_8Ê\ ,  $$
where the subscript indicates the homological degree, satisfying the  relations 
$$A:=\bar S(x,y,z,u,v)/(xu, yu, zu, xv, yv, zv, uv, v^2)\ . $$
(The product by $u$ and by $v$ is equal to zero.) The differential map is defined on the generators by 
$$d z := xy, \quad d v :=uÊ\ , $$
and by $0$ otherwise. 

The algebra $A$ is finite dimensional and spanned by the $9$ elements: 
$x, \ y, \ xy, \ z, \ u, \ v, \ xz, \ yz, \ xyz$.
Its underlying homology $H_\bullet(A,d)$ is five dimensional and spanned by the classes of: 
$x, \ y,  \  xz, \ yz, \ xyz$.

We define the degree $+1$ operator $\Delta$ on the aforementioned elements by 
$$\Delta(xy):=u, \quad \Delta(z):=-v \ ,$$
and by $0$ otherwise. 

\begin{prop}
The dg commutative algebra $(A,  d, \Delta)$ is a dg BV-algebra, which satisfies the $d\Delta$-condition. 
\end{prop}

\begin{proof}
It is straightforward to see that $\Delta$ commutes with $d$, that it has order less than $2$ (but not less than $1$) and that it squares to $0$. 

A decomposition such as the one of Lemma~\ref{lemma:dDEquiv} is given by 
$$H_\bullet := \KK x \oplus \KK  y \oplus \KK  xz \oplus \KK yz \oplus \KK xyz
 \ \text{and} \  S_\bullet=    \KK z   \ . $$
Therefore, this dg BV-algebra satisfies the $d\Delta$-condition. 
\end{proof}

The first Massey product in the second stratum of the transferred homotopy hypercommutative algebra structure  is the first homotopy in the associated $C_\infty$-algebra structure, since 
$S^{-1}{Grav^*}^{(2)}(3)\cong {Com^{\ac}}(3)$. In the present example, this product is not trivial since it is equal to $-yz$ on the elements $x, y, y$. 
So this provides an example of a dg BV-algebra, which satisfies the $d\Delta$-condition, the strongest condition, and for which the Barannikov-Kontsevich-Manin structure of a Frobenius manifold on homology is not enough to recover the original homotopy type of the dg BV-algebra. 

\subsection{Application to Poisson geometry and Lie algebra cohomology}
Let $\mathcal M$ be an $n$-dimensional manifold. We consider the Gerstenhaber algebra 
of polyvector fields $A:=\Gamma(\mathcal M, \Lambda^\bullet T_{\mathcal M})$
 on $\mathcal M$, equipped with the Schouten-Nijenhuis bracket $\langle\,  ,  \rangle_{SN}$. 
 Recall from J.-L. Koszul \cite[Proposition~$(2.3)$]{Koszul85} that any torsion-free connection $\nabla $ on $T_{\mathcal M}$ which induces a flat connection on $\Lambda^n  T_{\mathcal M}$  gives rise to a square-zero order $2$ operator $D_\nabla$ making $(A, \wedge, D_\nabla, \langle\,  ,  \rangle_{SN})$ into a BV-algebra. For instance, this is the case when $\mathcal M$ is orientable with volume form $\Omega$ or when $\mathcal M$ is a Riemannian manifold with the Levi-Civita connection. 

Moreover, if $\mathcal M$ carries a Poisson structure, i.e. $w\in \Gamma(\mathcal M, \Lambda^2 T_{\mathcal M})$ satisfying $\langle w  , w \rangle_{SN}=0$, such that the infinitesimal automorphism $D_\nabla(w)=0$ vanishes, then the twisted differential $d_w:=  \langle w  , - \rangle_{SN}$ induces a dg BV-algebra 
$$(\Gamma(\mathcal M, \Lambda^\bullet T_{\mathcal M}), d_w, 
\wedge, D_\nabla, \langle\,  ,  \rangle_{SN}  ) \ . $$
For instance, this is the case when $\mathcal M$ is orientable with unimodular Poisson stucture, i.e. $D_\Omega(w)=0$.
The homology groups associated to the differential $d_w$ form the \emph{Poisson cohomology} of the manifold $\mathcal M$, see \cite{Lichnerowicz77}.
(For similar constructions in non-commutative geometry, we refer the reader to \cite{GinzburgSchedler10}).

\begin{prop}\cite{Koszul85}\label{prop:ISOdRPV}
When  $\mathcal M$ is a symplectic manifold, the contraction with the symplectic form $\omega$ induces  an isomorphism of dg BV-algebras 
$$(\Omega^{\bullet}(\mathcal M), d_{DR},  \wedge, \Delta, \langle\,  ,  \rangle) \cong  
(\Gamma(\mathcal M, \Lambda^\bullet T_{\mathcal M}), d_w, 
\wedge, D, \langle\,  ,  \rangle_{SN}  )  \ , $$
where $D:=[i_\omega, d_w]$. 
\end{prop}

Recall that the homology groups associated to the differential $\Delta$ on the left-hand side form the \emph{Poisson homology} of the manifold $\mathcal M$. The Poisson homology and cohomology are proved to be isomorphic under the weaker condition that the Poisson manifold is orientable and  unimodular, see P. Xu in \cite{Xu99}. 

\begin{thm}\label{thm:HomoStructure}
The de Rham cohomology of a Poisson manifold $\mathcal M$ carries a skeletal homotopy BV-algebra, whose rectified dg BV-algebra is homotopy equivalent to the dg BV-algebra $(\Omega^\bullet(\mathcal M), d_{DR}, \wedge,  \Delta)$.
The Poisson cohomology of an orientable Poisson manifold $\mathcal M$ carries a skeletal homotopy BV-algebra, whose rectified dg BV-algebra is homotopy equivalent to the dg BV-algebra $(\Gamma(\mathcal M, \Lambda^\bullet T_{\mathcal M}), d_w, 
\wedge, \Delta, \langle\,  ,  \rangle_{SN}  )$. 

The de Rham cohomology and the Poisson cohomology of a symplectic manifold are isomorphic skeletal homotopy BV-algebras. 
When the manifold $\mathcal M$ is compact and satisfies the hard Lefsechtz condition, this  isomorphism 
reduces to an isomorphism of 
homotopy hypercommutative algebras. 
\end{thm}

\begin{proof}
This is a direct corollary of Theorem~\ref{thm:HomotopyFrob} and Proposition~\ref{prop:ISOdRPV}.
\end{proof}

Let us now describe the linear case. Under the same notations as in the last example of Section~\ref{subsec:Trivialization}, when $V=\g^*$ is the linear dual of a finite dimensional Lie algebra, the transpose of the bracket produces a degree $-2$ element $w$ in $\g \otimes \Lambda^2 \g^*$  satisfying $\langle w, w\rangle=0$, by the Jacobi relation. In this case, the twisted differential $d_w$ is equal to the Chevalley-Eilenberg differential on $A\cong S(\g)\otimes \Lambda(\g^*)\subset C^\infty(\g^*)\otimes \Lambda (\g^*)$, which computes the cohomology of $\g$ with coefficients in $S(\g)$ and the adjoint action. If the Lie algebra $\g$ is unimodular, that is $\textrm{Tr}(\langle x  , - \rangle)=0$, for any $x\in\g$, then $\Delta(w)=0$ and the Chevalley-Eilenberg complex $(S(\g)\otimes \Lambda(\g^*), d_w, \bullet, \Delta,  \langle\,  , \rangle )$ 
is a dg BV-algebra. 

\begin{thm}
The Chevalley-Eilenberg cohomology $H_{CE}^\bullet(\g , S(\g))$ of a finite dimensional unimodular Lie algebra $\g$, with coefficients in $S(\g)$ with adjoint action, carries a skeletal homotopy BV-algebra, whose rectified dg BV-algebra is homotopy equivalent to the dg BV-algebra
$(S(\g)\otimes \Lambda(\g^*), d_w, \bullet, \Delta,  \langle\,  , \rangle )$. 
\end{thm}

\begin{remark}
It would be now interesting to study the relationship with the  Duflo isomorphism, the analogue of the space of  differential forms, and the  symplectic and the hard Lefschetz condition, in this linear case. 
\end{remark}

\subsection{Application to Mirror Symmetry}

\begin{thm}\label{thm:Dol@deRham}
The Dolbeault cohomology of a Calabi-Yau manifold carries a homotopy hypercommutative algebra structure, which extends the hypercommutative algebra structure of \cite{BarannikovKontsevich98} and 
 whose rectified dg BV-algebra is homotopy equivalent to the Dolbeault complex 
 $(\Gamma(M, \wedge^\bullet \bar{T}_{\mathcal M}^* \otimes \wedge^\bullet {T}_{\mathcal M}), \bar \partial, \wedge,  \text{div}, \langle\,  , \rangle_S)$.
\end{thm}

The moduli space $\mathscr{M}$ of Maurer-Cartan elements associated to the Dolbeault complex is an extension of the moduli space $\mathscr{M}^{\text{classical}}$ associated to the Kodaira-Spencer dg Lie subalgebra, which encodes deformations of complex structures. The notion of \emph{generalized complex geometry} was introduced by N. Hitchin in \cite{Hitchin03} and then developed by his students M. Gualtieri \cite{Gualtieri04} and G.R. Cavalcanti \cite{Cavalcanti05} as a framework  which encompasses both complex and symplectic geometries. In this sense, the moduli space 
 $\mathscr{M}$ was shown by Gualtieri to correspond to deformations of generalized complex structures. Several versions of the $d\Delta$-condition were shown to hold in this setting, see \cite{Gualtieri07, Cavalcanti07}. Finally the dg BV algebra structure of \cite{Li05} allows us to apply the same argument which produces a version of 
 Theorem~\ref{thm:Dol@deRham} in the context of generalized complex geometry.
 
 S. Barannikov generalized in \cite{Barannikov02} the notions of periods and  variations of Hodge structure from $\mathscr{M}^{\text{classical}}$ to $\mathscr{M}$. He showed, for instance, that the image of these generalized periods on $H^\bullet(\mathcal M, \mathbb C)$ coincide with the Gromov-Witten invariants. This is based on the fact that the Dolbeault cohomology admits not one but a family of Frobenius manifold structures. 
This remark coincides with the present approach: there are many choices in the Homotopy Transfer theorem. Moreover, the various transferred structures are related by the group of $\infty$-isomorphisms, see \cite[Theorem~$10.3.15$]{LodayVallette10}. In the case of homotopy BV-algebras, this group should be related to the Givental group  \cite{Givental01, Givental01bis}. 

The Mirror Symmetry conjecture \cite{Kontsevich95} claims that the Fukaya $A_\infty$-category of Lagrangian submanifolds of a Calabi-Yau manifold $\mathcal M$ (A-side) should be equivalent to the bounded derived category of coherent sheaves on a dual Calabi-Yau manifold $\widetilde{\mathcal M}$ (B-side). The tangent space of the  moduli space of $A_\infty$-deformations of the Fukaya category is conjectured to be given by the de Rham cohomology $H^\bullet_{DR}(\mathcal M, \mathbb C)$ of $X$.
By the Kontsevich formality \cite{Kontsevich03}, the $A_\infty$-deformations of the latter category are encoded by the Dolbeault complex. So the de Rham cohomology  equipped with the Gromov-Witten invariants should be isomorphic to the Dolbeault cohomology  $H^\bullet(\widetilde{\mathcal{M}}, \Lambda^\bullet T_{\widetilde{\mathcal{M}}})$ as Frobenius manifolds. The following conjecture of Cao-Zhou \cite{CaoZhou01}, similar to Proposition~\ref{prop:ISOdRPV},  gives a way to study this question: 
there is a quasi-isomorphism of dg BV-algebras 
$$(\Omega^{n-\bullet}(\mathcal M), d_{DR},  \wedge, \Delta, \langle\,  ,  \rangle) \xrightarrow{\sim}  
(\Gamma(M, \wedge^\bullet \bar{T}_{\widetilde{\mathcal{M}}}^* \otimes \wedge^\bullet {T}_{\widetilde{\mathcal{M}}}),  \bar \partial, \wedge, \text{div}, \langle\,  , \rangle_S) \ . $$
The results of the present paper show that it is actually enough to  prove the existence of an $\infty$-quasi-isomorphism of dg BV-algebras to get the aforementioned isomorphism on the cohomology level and to relate the two associated deformation functors.

%Finally, we plan to study the Lian-Zuckerman conjecture \cite{GCTV09} with the minimal model of the operad $\BV$. This would develop even further the relationship between vertex algebras and moduli space of curves, in the same way as \cite{FrenkelBenZvi04}.

\begin{center}
\textsc{Acknowledgements}
\end{center}
We are grateful to Damien Calaque, Cl\'ement Dupont, Vladimir Dotsenko, Jean-Louis Loday and Jim Stasheff for their useful comments on the first version of this paper.

B.V. would like to express his deep gratitude to the Max-Planck Institute f\"ur Mathematik in Bonn for the long term invitation and for the excellent working conditions. Both authors are happy to acknowledge the generous support provided by the Simons Center for Geometry And Physics. 

G. C. D.-C. was supported by the National Science Foundation under Award No. DMS-1004625 and B.V. was supported by the ANR grant JCJC06 OBTH.

%The authors would like to thank Jim Simons for the excellent drinking .... no ... working conditions at the Broadway penthouse. 

%%%%%%%%%%%%%%%   BIBLIOGRAPHY   %%%%%%%%%%%%%%%%%%%

\bibliographystyle{amsalpha}
\bibliography{bib}

\newcommand{\etalchar}[1]{$^{#1}$}
\def\cprime{$'$}
\providecommand{\bysame}{\leavevmode\hbox to3em{\hrulefill}\thinspace}
\providecommand{\MR}{\relax\ifhmode\unskip\space\fi MR }
% \MRhref is called by the amsart/book/proc definition of \MR.
\providecommand{\MRhref}[2]{%
  \href{http://www.ams.org/mathscinet-getitem?mr=#1}{#2}
}
\providecommand{\href}[2]{#2}
\begin{thebibliography}{DGMS75}

\bibitem[AG07]{Gualtieri07}
Vestislav Apostolov and Marco Gualtieri, \emph{Generalized {K}\"ahler
  manifolds, commuting complex structures, and split tangent bundles}, Comm.
  Math. Phys. \textbf{271} (2007), no.~2, 561--575.

\bibitem[Bar02]{Barannikov02}
Serguei Barannikov, \emph{Non-commutative periods and mirror symmetry in higher
  dimensions}, Comm. Math. Phys. \textbf{228} (2002), no.~2, 281--325.

\bibitem[BB09]{BataninBerger09}
M.~A. Batanin and C.~Berger, \emph{The lattice path operad and {H}ochschild
  cochains}, Alpine perspectives on algebraic topology, Contemp. Math., vol.
  504, Amer. Math. Soc., Providence, RI, 2009, pp.~23--52.

\bibitem[BCK{\etalchar{+}}66]{BCKQRS66}
A.~K. Bousfield, E.~B. Curtis, D.~M. Kan, D.~G. Quillen, D.~L. Rector, and
  J.~W. Schlesinger, \emph{The {${\rm mod}-p$} lower central series and the
  {A}dams spectral sequence}, Topology \textbf{5} (1966), 331--342.

\bibitem[BD04]{BeilinsonDrinfeld04}
Alexander Beilinson and Vladimir Drinfeld, \emph{Chiral algebras}, American
  Mathematical Society Colloquium Publications, vol.~51, American Mathematical
  Society, Providence, RI, 2004.

\bibitem[Bel11]{Bellier11}
Olivia Bellier, \emph{Koszul duality of operads over a {H}opf algebra}, work in
  progress (2011).

\bibitem[Ber09]{Berglund09}
A.~Berglund, \emph{Homological perturbation theory for algebras over operads},
  \texttt{arXiv:0909.3485} (2009).

\bibitem[BF09]{BehrendFantechi09}
Kai Behrend and Barbara Fantechi, \emph{Gerstenhaber and {B}atalin-{V}ilkovisky
  structures on {L}agrangian intersections}, Algebra, arithmetic, and geometry:
  in honor of {Y}u. {I}. {M}anin. {V}ol. {I}, Progr. Math., vol. 269,
  Birkh\"auser Boston Inc., Boston, MA, 2009, pp.~1--47.

\bibitem[BG10]{BaranovkyGinzburg10}
Vladimir Baranovsky and Victor Ginzburg,
  \emph{Gerstenhaber-{B}atalin-{V}ilkovisky structures on coisotropic
  intersections}, Math. Res. Lett. \textbf{17} (2010), no.~2, 211--229.

\bibitem[BK98]{BarannikovKontsevich98}
Sergey Barannikov and Maxim Kontsevich, \emph{Frobenius manifolds and formality
  of {L}ie algebras of polyvector fields}, Internat. Math. Res. Notices (1998),
  no.~4, 201--215.

\bibitem[BM03]{BergerMoerdijk03}
C.~Berger and I.~Moerdijk, \emph{Axiomatic homotopy theory for operads},
  Comment. Math. Helv. \textbf{78} (2003), no.~4, 805--831.

\bibitem[Bor86]{Borcherds86}
Richard~E. Borcherds, \emph{Vertex algebras, {K}ac-{M}oody algebras, and the
  {M}onster}, Proc. Nat. Acad. Sci. U.S.A. \textbf{83} (1986), no.~10,
  3068--3071.

\bibitem[Bry88]{Brylinski88}
Jean-Luc Brylinski, \emph{A differential complex for {P}oisson manifolds}, J.
  Differential Geom. \textbf{28} (1988), no.~1, 93--114.

\bibitem[BV81]{BatalinVilkovisky81}
I.~A. Batalin and G.~A. Vilkovisky, \emph{Gauge algebra and quantization},
  Phys. Lett. B \textbf{102} (1981), no.~1, 27--31.

\bibitem[{Cav}05]{Cavalcanti05}
G.~R. {Cavalcanti}, \emph{{New aspects of the ddc-lemma}}, arXiv:math/0501406
  (2005).

\bibitem[Cav07]{Cavalcanti07}
Gil~R. Cavalcanti, \emph{Formality in generalized {K}\"ahler geometry},
  Topology Appl. \textbf{154} (2007), no.~6, 1119--1125.

\bibitem[CD06]{CarrDevadoss06}
Michael~P. Carr and Satyan~L. Devadoss, \emph{Coxeter complexes and
  graph-associahedra}, Topology Appl. \textbf{153} (2006), no.~12, 2155--2168.

\bibitem[CG11]{CostelloGwilliam11}
Kevin Costello and Owen Gwilliam, \emph{Factorization algebras in perturbative
  quantum field theory}, {available on the home pages of the authors} (2011).

\bibitem[Cos07]{Costello07}
Kevin Costello, \emph{Topological conformal field theories and {C}alabi-{Y}au
  categories}, Adv. Math. \textbf{210} (2007), no.~1, 165--214.

\bibitem[CS99]{ChasSullivan99}
M.~{Chas} and D.~{Sullivan}, \emph{{String Topology}}, ArXiv Mathematics
  e-prints (1999).

\bibitem[CZ01]{CaoZhou01}
Huai-Dong Cao and Jian Zhou, \emph{D{GBV} algebras and mirror symmetry}, First
  {I}nternational {C}ongress of {C}hinese {M}athematicians ({B}eijing, 1998),
  AMS/IP Stud. Adv. Math., vol.~20, Amer. Math. Soc., Providence, RI, 2001,
  pp.~279--289.

\bibitem[DCP95]{DeConciniProcesi95}
C.~De~Concini and C.~Procesi, \emph{Wonderful models of subspace arrangements},
  Selecta Math. (N.S.) \textbf{1} (1995), no.~3, 459--494.

\bibitem[DCV09]{DrummondColeVallette09}
Gabriel Drummond-Cole and Bruno Vallette, \emph{$\infty$-operads,
  ${B}{V}_\infty$, and {H}ypercommutative${}_\infty$}, Oberwolfach Report
  (2009), no.~28, 1566--1569.

\bibitem[DF08]{DevadossForcey08}
Satyan Devadoss and Stefan Forcey, \emph{Marked tubes and the graph
  multiplihedron}, Algebr. Geom. Topol. \textbf{8} (2008), no.~4, 2081--2108.

\bibitem[DGMS75]{DGMS75}
Pierre Deligne, Phillip Griffiths, John Morgan, and Dennis Sullivan, \emph{Real
  homotopy theory of {K}\"ahler manifolds}, Invent. Math. \textbf{29} (1975),
  no.~3, 245--274.

\bibitem[DK09]{DotsenkoKhoroshkin09}
V.~{Dotsenko} and A.~{Khoroshkin}, \emph{{Free resolutions via Gr\"obner
  bases}}, ArXiv e-prints (2009).

\bibitem[DK10]{DotsenkoKhoroshkin10}
Vladimir Dotsenko and Anton Khoroshkin, \emph{Gr\"obner bases for operads},
  Duke Math. J. \textbf{153} (2010), no.~2, 363--396.

\bibitem[FBZ04]{FrenkelBenZvi04}
Edward Frenkel and David Ben-Zvi, \emph{Vertex algebras and algebraic curves},
  second ed., Mathematical Surveys and Monographs, vol.~88, American
  Mathematical Society, Providence, RI, 2004.

\bibitem[GCTV09]{GCTV09}
I.~Galvez-Carrillo, A.~Tonks, and B.~Vallette, \emph{Homotopy
  {B}atalin-{V}ilkovisky algebras}, to appear in {J}ournal of {N}oncommutative
  {Geometry} (2009), \texttt{arXiv:0907.2246}.

\bibitem[Get94a]{Getzler94}
E.~Getzler, \emph{Batalin-{V}ilkovisky algebras and two-dimensional topological
  field theories}, Comm. Math. Phys. \textbf{159} (1994), no.~2, 265--285.

\bibitem[Get94b]{Getzler94bis}
\bysame, \emph{Two-dimensional topological gravity and equivariant cohomology},
  Comm. Math. Phys. \textbf{163} (1994), no.~3, 473--489.

\bibitem[Get95]{Getzler95}
\bysame, \emph{Operads and moduli spaces of genus {$0$} {R}iemann surfaces},
  The moduli space of curves ({T}exel {I}sland, 1994), Progr. Math., vol. 129,
  Birkh\"auser Boston, Boston, MA, 1995, pp.~199--230.

\bibitem[Gin06]{Ginzburg06}
V.~Ginzburg, \emph{Calabi-{Y}au algebras}, \texttt{arXiv.org:math/0612139}
  (2006).

\bibitem[Giv01a]{Givental01}
Alexander~B. Givental, \emph{Gromov-{W}itten invariants and quantization of
  quadratic {H}amiltonians}, Mosc. Math. J. \textbf{1} (2001), no.~4, 551--568,
  645, Dedicated to the memory of I. G. Petrovskii on the occasion of his 100th
  anniversary.

\bibitem[Giv01b]{Givental01bis}
\bysame, \emph{Semisimple {F}robenius structures at higher genus}, Internat.
  Math. Res. Notices (2001), no.~23, 1265--1286.

\bibitem[GJ94]{GetzlerJones94}
E.~Getzler and J.~D.~S. Jones, \emph{{Operads, homotopy algebra and iterated
  integrals for double loop spaces}}, \texttt{hep-th/9403055} (1994).

\bibitem[Gra07]{Granaker07}
Johan Gran{\aa}ker, \emph{Strong homotopy properads}, Int. Math. Res. Not. IMRN
  (2007), no.~14.

\bibitem[GS10]{GinzburgSchedler10}
Victor Ginzburg and Travis Schedler, \emph{Differential operators and {BV}
  structures in noncommutative geometry}, Selecta Math. (N.S.) \textbf{16}
  (2010), no.~4, 673--730.

\bibitem[{Gua}04]{Gualtieri04}
M.~{Gualtieri}, \emph{{Generalized complex geometry}}, arXiv:math/0401221
  (2004).

\bibitem[Hin97]{Hinich97}
Vladimir Hinich, \emph{Homological algebra of homotopy algebras}, Comm. Algebra
  \textbf{25} (1997), no.~10, 3291--3323, see also Erratum
  \texttt{arXiv.org:math/0309453}.

\bibitem[Hir03]{Hirshhorn03}
Philip~S. Hirschhorn, \emph{Model categories and their localizations},
  Mathematical Surveys and Monographs, vol.~99, American Mathematical Society,
  Providence, RI, 2003.

\bibitem[Hit03]{Hitchin03}
Nigel Hitchin, \emph{Generalized {C}alabi-{Y}au manifolds}, Q. J. Math.
  \textbf{54} (2003), no.~3, 281--308.

\bibitem[HL11]{HoefelLivernet11}
E.~{Hoefel} and M.~{Livernet}, \emph{{{O}{C}{H}{A} and {L}eibniz pairs, towards
  a {K}oszul duality}}, arXiv:math/1104.3607 (2011).

\bibitem[HM10]{HirschMilles10}
Joseph Hirsh and Joan Mill\`es, \emph{Curved {K}oszul duality theory},
  \texttt{arXiv.org:1008.5368} (2010).

\bibitem[Hof10]{Hoffbeck10}
Eric Hoffbeck, \emph{A {P}oincar\'e-{B}irkhoff-{W}itt criterion for {K}oszul
  operads}, Manuscripta Math. \textbf{131} (2010), no.~1-2, 87--110.

\bibitem[Kau08]{Kaufmann04}
R.~M. Kaufmann, \emph{A proof of a cyclic version of {D}eligne's conjecture via
  {C}acti}, Mathematical Research Letters \textbf{15} (2008), no.~5, 901--921.

\bibitem[KM94]{KontsevichManin94}
M.~Kontsevich and Yu. Manin, \emph{Gromov-{W}itten classes, quantum cohomology,
  and enumerative geometry}, Comm. Math. Phys. \textbf{164} (1994), no.~3,
  525--562.

\bibitem[KMS11]{KhoroshkinMarkarianShadrin11}
A.~Khoroshkin, N.~Markarian, and S.~Shadrin, \emph{On quasi-isomorphism of
  $hycom$ and $bv/\delta$}, in preparation (2011).

\bibitem[Kon95]{Kontsevich95}
Maxim Kontsevich, \emph{Homological algebra of mirror symmetry}, Proceedings of
  the {I}nternational {C}ongress of {M}athematicians, {V}ol.\ 1, 2 ({Z}\"urich,
  1994) (Basel), Birkh\"auser, 1995, pp.~120--139.

\bibitem[Kon03]{Kontsevich03}
\bysame, \emph{Deformation quantization of {P}oisson manifolds}, Lett. Math.
  Phys. \textbf{66} (2003), no.~3, 157--216.

\bibitem[Kos85]{Koszul85}
Jean-Louis Koszul, \emph{Crochet de {S}chouten-{N}ijenhuis et cohomologie},
  Ast\'erisque (1985), no.~Numero Hors Serie, 257--271, The mathematical
  heritage of {\'E}lie Cartan (Lyon, 1984).

\bibitem[KS58]{KodairaSpencer58}
K.~Kodaira and D.~C. Spencer, \emph{On deformations of complex analytic
  structures. {I}, {II}}, Ann. of Math. (2) \textbf{67} (1958), 328--466.

\bibitem[KS60]{KodairaSpencer60}
\bysame, \emph{On deformations of complex analytic structures. {III}.
  {S}tability theorems for complex structures}, Ann. of Math. (2) \textbf{71}
  (1960), 43--76.

\bibitem[KS95]{YKS95}
Y.~Kosmann-Schwarzbach, \emph{Exact {G}erstenhaber algebras and {L}ie
  bialgebroids}, Acta Appl. Math. \textbf{41} (1995), no.~1-3, 153--165,
  Geometric and algebraic structures in differential equations.

\bibitem[KS09]{KontsevichSoibelman09}
M.~Kontsevich and Y.~Soibelman, \emph{Notes on {$A_\infty$}-algebras,
  {$A_\infty$}-categories and non-commutative geometry}, Homological mirror
  symmetry, Lecture Notes in Phys., vol. 757, Springer, Berlin, 2009,
  pp.~153--219.

\bibitem[{Li}05]{Li05}
Y.~{Li}, \emph{{On Deformations of Generalized Complex Structures: the
  Generalized Calabi-Yau Case}}, arXiv:hep-th/0508030 (2005).

\bibitem[Lic77]{Lichnerowicz77}
Andr{\'e} Lichnerowicz, \emph{Les vari\'et\'es de {P}oisson et leurs alg\`ebres
  de {L}ie associ\'ees}, J. Differential Geometry \textbf{12} (1977), no.~2,
  253--300.

\bibitem[Lod04]{Loday04a}
J.-L. Loday, \emph{Realization of the {S}tasheff polytope}, Arch. Math. (Basel)
  \textbf{83} (2004), no.~3, 267--278.

\bibitem[LS07]{LosevShadrin07}
A.~Losev and S.~Shadrin, \emph{From {Z}wiebach invariants to {G}etzler
  relation}, Comm. Math. Phys. \textbf{271} (2007), no.~3, 649--679.

\bibitem[LV10]{LodayVallette10}
Jean-Louis Loday and Bruno Vallette, \emph{Algebraic operads}, first draft
  available at \texttt{http://math.unice.fr/$\sim$brunov/Operads.pdf}, 2010.

\bibitem[LZ93]{LianZuckerman93}
Bong~H. Lian and Gregg~J. Zuckerman, \emph{New perspectives on the
  {BRST}-algebraic structure of string theory}, Comm. Math. Phys. \textbf{154}
  (1993), no.~3, 613--646.

\bibitem[Man99]{Manin99}
Yuri~I. Manin, \emph{Frobenius manifolds, quantum cohomology, and moduli
  spaces}, American Mathematical Society Colloquium Publications, vol.~47,
  American Mathematical Society, Providence, RI, 1999.

\bibitem[Mar96]{Markl96bis}
Martin Markl, \emph{Models for operads}, Comm. Algebra \textbf{24} (1996),
  no.~4, 1471--1500.

\bibitem[Mar09]{Markarian09}
Nikita Markarian, \emph{$hycom=bv/\delta$,}, A blog post available through the
  URL \texttt{http://nikitamarkarian.wordpress.com/2009/11/22/hycommbv?/}
  (2009).

\bibitem[Mas58]{Massey58}
W.~S. Massey, \emph{Some higher order cohomology operations}, Symposium
  internacional de topolog\'\i a algebraica {I}nternational symposium on
  algebraic topology, Universidad Nacional Aut\'onoma de M\'exico and UNESCO,
  Mexico City, 1958, pp.~145--154.

\bibitem[Mat95]{Mathieu95}
Olivier Mathieu, \emph{Harmonic cohomology classes of symplectic manifolds},
  Comment. Math. Helv. \textbf{70} (1995), no.~1, 1--9.

\bibitem[Men09]{Menichi09}
L.~Menichi, \emph{Batalin-{V}ilkovisky algebra structures on {H}ochschild
  cohomology}, Bulletin de la Soci\'et\'e Math\'ematique de France \textbf{137}
  (2009), no.~2, 277--295.

\bibitem[Mer98]{Merkulov98}
S.~A. Merkulov, \emph{Formality of canonical symplectic complexes and
  {F}robenius manifolds}, Internat. Math. Res. Notices (1998), no.~14,
  727--733.

\bibitem[ML95]{MacLane95}
Saunders Mac~Lane, \emph{Homology}, Classics in Mathematics, Springer-Verlag,
  Berlin, 1995, Reprint of the 1975 edition.

\bibitem[MSS02]{MSS}
Martin Markl, Steve Shnider, and Jim Stasheff, \emph{Operads in algebra,
  topology and physics}, Mathematical Surveys and Monographs, vol.~96, American
  Mathematical Society, Providence, RI, 2002.

\bibitem[MV09a]{MerkulovVallette09I}
Sergei Merkulov and Bruno Vallette, \emph{Deformation theory of representations
  of prop(erad)s. {I}}, J. Reine Angew. Math. \textbf{634} (2009), 51--106.

\bibitem[MV09b]{MerkulovVallette09II}
\bysame, \emph{Deformation theory of representations of prop(erad)s. {II}}, J.
  Reine Angew. Math. \textbf{636} (2009), 123--174.

\bibitem[Par07]{Park07}
Jae-Suk Park, \emph{Semi-classical quantum field theories and {F}robenius
  manifolds}, Lett. Math. Phys. \textbf{81} (2007), no.~1, 41--59.

\bibitem[Pri70]{Priddy70}
S.B. Priddy, \emph{Koszul resolutions}, Trans. Amer. Math. Soc. \textbf{152}
  (1970), 39--60.

\bibitem[PS94]{PenkavaSchwarz94}
Michael Penkava and Albert Schwarz, \emph{On some algebraic structures arising
  in string theory}, Perspectives in mathematical physics, Conf. Proc. Lecture
  Notes Math. Phys., III, Int. Press, Cambridge, MA, 1994, pp.~219--227.

\bibitem[Qui69]{Quillen69}
D.~Quillen, \emph{Rational homotopy theory}, Ann. of Math. (2) \textbf{90}
  (1969), 205--295.

\bibitem[Ran97]{Ran97}
Z.~Ran, \emph{Thickening {C}alabi-{Y}au moduli spaces}, Mirror symmetry, {II},
  AMS/IP Stud. Adv. Math., vol.~1, Amer. Math. Soc., Providence, RI, 1997,
  pp.~393--400.

\bibitem[Ret93]{Retakh93}
Vladimir~S. Retakh, \emph{Lie-{M}assey brackets and {$n$}-homotopically
  multiplicative maps of differential graded {L}ie algebras}, J. Pure Appl.
  Algebra \textbf{89} (1993), no.~1-2, 217--229.

\bibitem[Rez96]{Rezk96}
Charles~W. Rezk, \emph{Spaces of algebra structures and cohomology of operads},
  Ph.D. thesis, MIT, 1996.

\bibitem[Rog09]{Roger09}
Claude Roger, \emph{Gerstenhaber and {B}atalin-{V}ilkovisky algebras;
  algebraic, geometric, and physical aspects}, Arch. Math. (Brno) \textbf{45}
  (2009), no.~4, 301--324.

\bibitem[Sch93]{Schwarz93}
Albert Schwarz, \emph{Geometry of {B}atalin-{V}ilkovisky quantization}, Comm.
  Math. Phys. \textbf{155} (1993), no.~2, 249--260.

\bibitem[Sta98]{Stasheff98}
Jim Stasheff, \emph{The (secret?) homological algebra of the
  {B}atalin-{V}ilkovisky approach}, Secondary calculus and cohomological
  physics ({M}oscow, 1997), Contemp. Math., vol. 219, Amer. Math. Soc.,
  Providence, RI, 1998, pp.~195--210.

\bibitem[Sul77]{Sullivan77}
D.~Sullivan, \emph{Infinitesimal computations in topology}, Inst. Hautes
  \'Etudes Sci. Publ. Math. (1977), no.~47, 269--331 (1978).

\bibitem[{Sul}10]{Sullivan10}
D.~{Sullivan}, \emph{{Algebra, Topology and Algebraic Topology of 3D Ideal
  Fluids}}, ArXiv e-prints (2010).

\bibitem[Ter08]{Terilla08}
John Terilla, \emph{Smoothness theorem for differential {BV} algebras}, J.
  Topol. \textbf{1} (2008), no.~3, 693--702.

\bibitem[Tra08]{Tradler08}
T.~Tradler, \emph{The {B}atalin-{V}ilkovisky algebra on {H}ochschild cohomology
  induced by infinity inner products}, Annales de l'institut Fourier
  \textbf{58} (2008), no.~7, 2351--2379.

\bibitem[TT00]{TamarkinTsygan00}
D.~Tamarkin and B.~Tsygan, \emph{Noncommutative differential calculus, homotopy
  {BV} algebras and formality conjectures}, Methods Funct. Anal. Topology
  \textbf{6} (2000), no.~2, 85--100.

\bibitem[TTW11]{TTW10}
J.~{Terilla}, T.~{Tradler}, and S.~O. {Wilson}, \emph{{Homotopy DG algebras
  induce homotopy BV algebras}}, ArXiv e-prints (2011).

\bibitem[TZ06]{TradlerZeinalian06}
T.~Tradler and M.~Zeinalian, \emph{On the cyclic {D}eligne conjecture}, J. Pure
  Appl. Algebra \textbf{204} (2006), no.~2, 280--299.

\bibitem[VdL02]{VanderLaan02}
P.~Van~der Laan, \emph{{Operads up to Homotopy and Deformations of Operad
  Maps}}, \texttt{arXiv:math.QA/0208041} (2002).

\bibitem[VdL03]{VanDerLaan03}
\bysame, \emph{Coloured {K}oszul duality and strongly homotopy operads},
  \texttt{arXiv:math.QA/0312147} (2003).

\bibitem[Wan67]{Wang67}
John S.~P. Wang, \emph{On the cohomology of the {${\rm mod}-2$} {S}teenrod
  algebra and the non-existence of elements of {H}opf invariant one}, Illinois
  J. Math. \textbf{11} (1967), 480--490.

\bibitem[Wit90]{Witten90}
Edward Witten, \emph{A note on the antibracket formalism}, Modern Phys. Lett. A
  \textbf{5} (1990), no.~7, 487--494.

\bibitem[Wit92]{Witten92}
\bysame, \emph{Ground ring of two-dimensional string theory}, Nuclear Phys. B
  \textbf{373} (1992), no.~1, 187--213.

\bibitem[WZ92]{WittenZwiebach92}
Edward Witten and Barton Zwiebach, \emph{Algebraic structures and differential
  geometry in two-dimensional string theory}, Nuclear Phys. B \textbf{377}
  (1992), no.~1-2, 55--112.

\bibitem[Xu99]{Xu99}
Ping Xu, \emph{Gerstenhaber algebras and {BV}-algebras in {P}oisson geometry},
  Comm. Math. Phys. \textbf{200} (1999), no.~3, 545--560.

\bibitem[Zwi93]{Zwiebach93}
Barton Zwiebach, \emph{Closed string field theory: quantum action and the
  {B}atalin-{V}ilkovisky master equation}, Nuclear Phys. B \textbf{390} (1993),
  no.~1, 33--152.

\end{thebibliography}

\end{document}